\documentclass[a4paper, 12pt, english]{article}

\usepackage[utf8]{inputenc} 
\usepackage[T1]{fontenc} 
\usepackage{babel} 

\usepackage{amsmath} 
\usepackage{amssymb} 
\usepackage{amsthm}
\usepackage{esint} 

\newcommand{\diff}{\mathrm{d}}
\newcommand{\R}{\mathbf{R}}
\DeclareMathOperator{\Imag}{Im}
\DeclareMathOperator{\ind}{Ind}
\DeclareMathOperator{\nul}{Null}
\DeclareMathOperator{\End}{End}
\DeclareMathOperator{\Tr}{Tr}
\DeclareMathOperator{\supp}{supp}
\newcommand{\N}{\mathbf{N}}
\newcommand{\Sph}{\mathbf{S}}
\newcommand{\loc}{\mathrm{loc}}
\newcommand{\Vect}[1]{\mathrm{Vect}\left( #1 \right)}
\newcommand{\gdo}[1]{\mathrm{O}\left( #1 \right)}

\newif\ifdraftmode

\draftmodetrue 

\ifdraftmode
\usepackage[textsize=small,  textwidth=\marginparwidth]{todonotes}
\usepackage{geometry}
\reversemarginpar
\else  
\newcommand{\todo}[1]{}
\fi

\usepackage[
backend=biber,
url=false,
doi=false,
isbn=false,
eprint=false
]{biblatex}
\usepackage{csquotes} 

\usepackage[colorlinks = true,
linkcolor = red,
urlcolor  = black,
citecolor = red,
anchorcolor = black]{hyperref}
\usepackage{cleveref}

\newtheorem{theorem}{Theorem}[section]
\newtheorem{proposition}[theorem]{Proposition}
\newtheorem{lemma}[theorem]{Lemma}
\newtheorem{corollary}[theorem]{Corollary}
\newtheorem{definition}[theorem]{Definition}
\newtheorem{propdef}[theorem]{Proposition-Definition}
\theoremstyle{remark}
\newtheorem{remark}{Remark}[section]

\title{Morse index stability for $p$-Yang-Mills connections}
\author{Mario Gauvrit, Paul Laurain, Tristan Rivière}
\date{}

\begin{document}
	
	\maketitle

	
	\begin{abstract}
		{\bf Abstract }: We establish the lower semi continuity of the Morse index and the upper continuity of the Morse Index plus nullity of sequences of critical points of the Sacks-Uhlenbeck type relaxation of the Yang-Mills Energy in 4 dimension. The result is known not to be true in general for the ``cousin problem'' of hamonic maps from surfaces into arbitrary manifolds. This result is stressing the more stable behaviour of Yang-Mills Fields compare to harmonic maps as observed in other contexts such as the flow. The Morse Index control at the limit of critical points to Sacks Uhlenbeck relaxations of Yang-Mills Lagrangian is a central result in the implementation of minmax operation on this Lagrangian.
	\end{abstract}
	
	\section*{Introduction}
	
	The calculus of variations of Yang-Mills energy for connections of principal $G-$bundles over  $4-$dimensional Riemannian manifolds  where $G$ is a compact Lie Group has received a close attention of the geometric analysis community
	in the early 80's in relation with important differential topology/geometry questions.  While the main focus was concentrated on special absolute minimizing solutions called Self-Dual/Anti-self-dual solutions or {\it instantons}, the construction of non zero Morse index Yang-Mills fields was also a subject of great interest in particular in relation with the so called Atiyah-Jones conjecture. The conjecture is about to know whether the homotopy type of the Moduli space ${\mathcal M}_n$ of instantons of topological charge $n$ modulo the gauge group action, for a $SU(2)-$bundle  with second Chern class $n$ over a four dimensional Riemannian manifold $M^4$, is ``comparable'' to the homotopy type of the space of general $SU(2)-$connections ${\mathcal A}_n$   modulo gauge group action with second Chern class $n$. More specifically, a question is about to know whether,  the low degree homotopy groups of the two spaces coincide or not. 
	This question is intimately related to the existence  of non zero Morse index Yang-Mills Fields. This general consideration can be sketched as follows : assuming a non zero $k-$homotopy class within the Moduli space of Instantons $\alpha\in \pi_k({\mathcal M}_n)$ would ``trivialize'' (would be mapped to zero) 
	for the canonical inclusion map $\iota$ into the general space of instantons $\iota_\ast\alpha=0$ in $\pi_k({\mathcal A}_n)$, one would easily produce a {\it min-max operation} whose underlying admissible space would be made of any continuous map from the $k+1$ ball $B^{k+1}$ into ${\mathcal A}_n$ whose boundary $\partial B^{k+1}$ is mapped into the space of instantons (which happen to realise the minimal Yang-Mills energy $8\pi^2 |n|$ of the given $SU(2)$ bundle) and the restriction to the boundary would realise the class $\alpha\ne 0$ in $\pi_k({\mathcal M}_n)$. Because the minimal energy $8\pi^2\,|n|$  is characterising  the membership to the moduli space ${\mathcal M}_n$, the ``width'' of such a min-max operation would be non trivial (strictly larger than $8\pi^2\, |n|$) otherwise one would contradicts the non triviality of $\alpha$ in $\pi_k({\mathcal M}_n)$. It is then ``natural'' to expect  this width to be realised by a non-zero Morse index $k+1$ Yang-Mills Fields.   In order to produce  a rigorous proof of this statement major analytical challenges have to be solved. The main reason is that, due in particular to its conformal invariance, the Yang-Mills energy cannot fulfil the necessary requirements for implementing Palais-Smale deformation theory in infinite dimensional space as such. Numerous publications have been devoted to this difficulty that Yang-Mills energy is sharing with several famous Lagrangians such as the Dirichlet energy of maps from a surface into  closed Riemannian manifolds whose critical points are called ``harmonic maps''.  The main motivation behind the present work is to bring a new piece of understanding of how the absence of Palais-Smale property for Yang-Mills energy can be ``overcome'' while implementing general min-max operations in particular. Before describing our main result we conclude this preamble by mentioning that there has been numerous contributions in the direction of Atiyah Jones conjecture and that we will omit to report on since this is not the object of the present paper. 
	
	In a pioneered work in concentration compactness theory (see \cite{SaU}), J.Sacks and K.Uhlenbeck, were looking for absolute minimizers of the Dirichlet energy of maps from the two sphere $S^2$ into an arbitrary closed sub-manifold $N^n$ of an euclidean space $R^N$realising a non zero free homotopy class of $\pi_2({\N}^n)$. Because of the absence of Sobolev embeddings of $W^{1,2}(S^2,N^n)$ into $C^0$ the existence of a minimizer was a-priori not guaranteed. To remedy to this difficulty, the two authors introduced an ``enhanced'' version of the Dirichlet energy known nowadays as ``Sacks-Uhlenbeck approximation'' :
	\[
	E_p(u):=\int_{S^2}(1+|d u|^2_{S^2})^{p/2}\ dvol_{S^2}\ .
	\] 
	for $p>2$. Maps of finite $E_p$ energy are exactly elements from the Sobolev space $W^{1,p}(S^2,N^n)$ which this time, for $p>2$, embeds compactly into $C^{0}(S^2,N^n)$. Because of this last fact the existence of a minimizers in a given free homotopy class is straightforward. For similar reasons the $p-$energy is sub-critical and satisfies the Palais Smale condition and general min-max operation can be implemented in the Banach Manifold ${\frak M}:=W^{1,p}(S^2,N^n)$ equipped with the Finsler structure given simply by the $W^{1,p}$ norm restricted to the tangent space $T_u{\frak M}$ to ${\frak M}$ at every element $u\in W^{1,p}(S^2,N^n)$. Then, once the existence of  minimizers or min-max critical points to the enhanced energies have been established, for any $p>2$,  the general strategy  introduced first in \cite{SaU} consists in ``following'' the obtained critical points to $E_{p_k}$ as $p_k\rightarrow 2^{+}$ and hopefully converge to critical points to the Dirichlet energy itself. In order to pass to the limit a  delicate analysis has to be performed. The central result in this analysis is the so called $\epsilon-$regularity theorem, saying roughly that below some universal positive threshold of energy (depending only on $N^n$ and not on $p$), every norm of every critical point of the $p-$energy is controlled. This permits, modulo extraction of a subsequence, to pass to the limit away from finitely many points in the domain while at these points so called ``bubbles'' (i.e. concentrated critical points) can be formed at the limit. These bubbles are connected to each other by regions which are called ``neck regions''.
	
	In a similar way, in order to overcome the absence of Palais-Smale property for the Yang-Mills energy, one introduces the $p-$Yang-Mills energy of any $G-$connection $A$ (where $G$ is a comp[act Lie group) over a 4 dimensional Riemannian manifold
	$ (M^4,h)$ :
	\[
	{\mathcal Y}{\mathcal M}_p(A):=\int_{M^4}(1+|F_A|^2)^{p/2}\, \mbox{vol}_h\ .
	\]
	This energy is well defined in the space of Sobolev $W^{1,p}$ $G-$connections over $M^4$ (see \cite{FU}), in the present paper we will work in a slightly more general context of weak $W^{1,p}$ $G-$connections denoted $\mathfrak {U}_G^{p}(M^4)$, see section \ref{preli}. The $\epsilon-$regularity for $p-$Yang-Mills critical points over $4-$dimensional manifolds has been first established in \cite{HS}. For the convenience of the reader we present in the appendix A.3 of the present work an alternative proof.
	In the present work we study the passage to the limit for $p_k-$Yang-Mills critical points of uniformly bounded energy as $p_k\rightarrow 2$. Thanks to the epsilon regularity we deduce a strong convergence/bubble tree convergence property our first main result is about the absence of loss of energy in the neck regions connecting the bubbles between themselves. Precisely we have
	\begin{theorem} 
		\label{bubbletree}(See also Theorem 2 of \cite{HS}) 	Let $ (M^4,h)$ be a closed four-dimensional Riemannian manifold and  $A_k\in\mathfrak {U}_G^{p_k}(M^4)$ a sequence of $p_k$-Yang-Mills connection with uniformly bounded energy\footnote{It means that $\mathcal{YM}_{p_k}(A_k)$ is uniformly bounded.} and such that $p_k \underset{k\to+\infty}{\rightarrow} 2^+$. Then we have, up to a sub-sequence,
		\begin{enumerate}
			\item There exists finitely many points $\{p_1,\dots, p_N\}$ and a Yang-Mills connection $A_\infty\in \mathfrak U_G(M^4)$ such that $A_k$ converges to $A_\infty$ in $\mathfrak C^\infty_{G,\mathrm{loc}}(M\setminus \{p_1,\dots, p_N\})$.
			\item For each $i\in \{1, \dots, N\}$, there exists $N_i\in \N$ sequences of points $(p_k^{i,j})_k$ converging to $p_i$,  $N_i$ sequences of scalars $(\lambda_k^{i,j})_k$ converging to zero, $N_i$ non-trivial Yang-Mills connections, called bubbles,  $A_\infty^{i,j} \in\mathfrak{U}_G(S^4)$ such that,
			$$(\phi_{k}^{i,j})^*(A_k) \rightarrow \widehat{A}_\infty^{i,j}=\pi_*A_\infty^{i,j}  \text{ into } \mathfrak C^\infty_{G,\mathrm{loc}}(\R^4 \setminus \{\text{ finitely many points}\}),$$
			where $\phi_{k}^{i,j} (x)= p_k^{i,j} +\lambda _k^{i,j} x$ in local coordinates and $\pi$ is the stereographic projection.
			\item Moreover there is no loss of energy, i.e
			$$ \lim_{k\rightarrow +\infty}\int_{M^4} \vert F_{A_k}\vert^2_h\, \mathrm{vol}_h  = \int_{M^4} \vert F_{A_\infty}\vert^2_h\, \mathrm{vol}_h +\sum_{i=1}^N \sum_{j=1}^{N_i} \int_{S^4} \vert F_{A_\infty^{i,j}}\vert^2\, \mathrm{vol}. $$
		\end{enumerate}
	\end{theorem}
	The result as such has already been proven in \cite{HS}. Our approach in the present work for proving theorem~\ref{bubbletree} is following the general strategy  originally introduced in the seminal work \cite{RiviereLin}.
	It consists in looking for  interpolation spaces estimates in neck regions. In fact this strategy allows for an up-grade of the energy quantization result to a much sharper $\mathrm{L}^{2,1}$energy  quantization (see corollary~\ref{L2quantization}). The first $L^{2,1}-$energy quantization result has been obtained by the second and the third author in \cite{laurainriviere2014angular}. We would like to stress at this stage that theorem~\ref{bubbletree}  and even more corollary~\ref{L2quantization} should be very striking to experts in conformal geometric analysis. Indeed, the $L^2$ energy quantization and even more the $L^{2,1}$ energy quantization are notoriously known to fail in general for the Sacks Uhlenbeck approximation of the Dirichlet Energy (\cite{LiW}) (except when the target is a symmetric space see \cite{LiZhu}  and \cite{DaRiSchla}). Our result is another illustration of the fact that Yang-Mills energy is enjoying more stability properties in bubble tree analysis than the ``cousin Lagrangian'' given by the Dirichlet energy in 2 dimension. It goes in the same direction of Waldron's theorem asserting that  Yang-Mills heat flow  never blows-up in finite time \cite{Waldron}. contrary to the harmonic map flow.
	
	In the second part of the paper, we implement the $L^{2,1}-$energy quantization in neck regions for proving a Morse index stability result following a strategy introduced by \cite{DGR22}. The following result is generalising to $p-$Yang-Mills the previous result obtained by the two first authors in \cite{GL24}.
	\begin{theorem}
		Let $ (M^4,h)$ be a closed four-dimensional Riemannian manifold and $A_k\in\mathfrak{U}_G^{p_k}(M^4)$ a sequence of $p_k$-Yang-Mills connections with uniformly bounded energy and such that $p_k \underset{k\to+\infty}{\rightarrow }2^+$. Let $A_\infty\in \mathfrak U_G(M^4)$ and $A_\infty^{i,j} \in\mathfrak U_G(S^4) $ be its bubble-tree limit in the sense of theorem \ref{bubbletree}. Then, for $k$ large enough, we have
		\begin{align}
			\mathrm{ind}_{\mathcal{YM}_{p_k}}(A_k) &\geq \mathrm{ind}_\mathcal{YM}(A_{\infty})+ \sum_{i=1}^N\sum_{j=1}^{N_i} \mathrm{ind}_\mathcal{YM}(A_{\infty}^{i,j}), \\
			\mathrm{ind}^0_{\mathcal{YM}_{p_k}}(A_k) &\leq \mathrm{ind}^0_\mathcal{YM}(A_{\infty})+ \sum_{i=1}^N \sum_{j=1}^{N_i} \mathrm{ind}^0_\mathcal{YM}(A_{\infty}^{i,j}),
		\end{align} where $\mathrm{ind}^0_{\mathcal{YM}_{p}} (A)=\ind_{\mathcal{YM}_p}(A) + \dim \left(\ker Q_{A,p} \cap \ker \diff_A^*\right)$ is the extended index of the connection and $Q_{A,p}$ is the quadratic form associated to the second variation.
	\end{theorem}
	
	This result is the Yang-Mills counterpart of the same result obtained in \cite{DaRiSchla}  for Sacks Uhlenbeck relaxation of the Dirichlet energy of maps into spheres and symmetric spaces in general see \cite{Schlagen}.
\medskip

	The paper is organised as follows. After that we introduce, in section 1 and 2,  notations and main notion for the definition of indices, we perform a fine analysis in neck region, section 3 to 5, here for the sake of clarity, we will state the results  in the Euclidean framework. They remain true on small enough geodesic balls of a Riemannian manifold, up to slight modification of constant which as no effect on the desired result. Finally, in section 6, we prove our main theorem in a general setting.

	\medskip
	
	\noindent{\bf Acknowledgements} This work has been initiated during a one semester visit of the first author at the ETH Z\"urich. He would like to thank the mathematics department at ETH for its hospitality.
	\setcounter{tocdepth}{2}
	\tableofcontents
	
	\newpage
	\section{Preliminaries on weak-connections}
	\label{preli}
	We work in the context of weak-connection introduce by Petrache-Rivière \cite{PR14}, see also the introduction and the appendix of the work of the first and second author \cite{GL24}. Here, we just have to extend the theory to $W^{2,2}$-bundles with $W^{1,p}$-connection with $p>2$, which is mostly obvious since every thing which is true for $W^{1,2}$-connection will be automatically true for $W^{1,p}$-connections. 
	Let $(M^4, h)$ be an arbitrary Riemannian manifold, $G$ a compact Lie subgroup of $\mathrm{SU}(n)$ with Lie algebra $\mathfrak{g}$, $p\geq 2$, a weak $W^{1,p}$-connection is a $\mathfrak{g}$-valued 1-form $A$ (called the connection form) such that 
	\begin{itemize}
		\item $A\in \mathrm{L}^{4,\infty}(M,T^*M\otimes \mathfrak{g})$\footnote{See \cite{Grafakos1} for a complete introduction to Lorentz spaces.},
		\item $F_A= \diff A + A \wedge A$ (defined as a distribution) is in $\mathrm{L}^p$,
		\item locally, there exists a $\mathrm{W}^{1,(4,\infty)}$-gauge $g$ satisfying $A^g\in \mathrm{W}^{1,p}$, where $ A^g := g^{-1} A g+g^{-1}\diff g$ is the expression of $A$ after the gauge change $g$,
	\end{itemize}  
	We denote by $\mathfrak A^p_G(M^4)$ the space of such connections. The main advantage of this formulation, is that the connection form is globally defined, which includes the classical case, since by \cite[theorem A]{PR14}, every $\mathrm{W}^{2,2}$-bundle is a trivial $\mathrm{W}^{1,(4,\infty)}$-bundle, see also the appendix of \cite{GL24}. Moreover, we can also define an associated notion of convergence, i.e. smooth convergence in local good gauge, we denote it by $\mathfrak{C}_{G,loc}^\infty$- convergence , see Proposition-definition A.8 of \cite{GL24} for a precise definition.

	\section{Second variation : computation and finiteness of the Morse index}
	We start by reminding the expression of the first variation of the $p$-Yang-Mills functional in our framework.
	
	\begin{proposition}[Section 2 of \cite{HS}] For every connection $A\in\mathfrak U_G^p(M)$,  for all $a\in \mathrm{W}^{1,p}(M, T^*M\otimes \mathfrak{g})$, 
		\begin{equation} \frac{\diff}{\diff t}_{|t=0}  \mathcal{YM}_p(A + ta) = p\int_{M}(1+|F_A|_h^2)^{\frac{p}{2}-1} \langle F_{A}, \diff_{A} a
			\rangle_h \mathrm{vol}_h. \label{variationpremiere}
		\end{equation}
		Hence $A\in\mathfrak U_G^p(M)$ is a $p$-Yang-Mills connections, if it satisfies in the distributional sense the following equation
		\begin{equation} \label{equationdepYMdiv}
			\diff_A^* \left((1+|F_A|_h^2)^{p/2-1}F_A\right) =0 .
		\end{equation}
	\end{proposition}
	
	\begin{remark}
		Recall that  the Bianchi identity is satisfies by any connection, i.e. \begin{equation}
			\label{equationdeBianchi} \diff_A F_A =0.
		\end{equation}
	\end{remark}
	
	It turns out, as already proved in \cite{HS}, that solution of the $p$-Yang-Mills equation is smooth in Coulomb gauge and satisfies so-called $\varepsilon$-regularity estimates under a small energy assumption. But for sake of completeness, we give new proofs of those results  in the appendix. In particular, knowing solutions are regular, the $p$-Yang-Mills equation \eqref{equationdepYMdiv} rewrites as 
	
	\begin{equation} \label{equationdepYM}
		\diff_A^* F_A = \frac{p-2}{2} \star \frac{\diff |F_A|_h^2\wedge \star F_A}{1+|F_A|_h^2}.
	\end{equation}
	Let us now consider the second variation.
	\begin{proposition}
		If $(A_t)_{t\in(-\varepsilon, \varepsilon)}\in \mathfrak A_G^p(M)$ is a one parameter  family of weak connections where $A:=A_0$ is a $p$-Yang-Mills connection and $a:=\displaystyle {\frac{\diff}{\diff t}}_{|t=0} A_t\in   \mathrm{W}^{1,p}(M, T^*M\otimes \mathfrak{g})$ then \begin{equation} 
			\frac{\diff^2}{\diff t^2}_{|t=0}\mathcal{YM}_p(A_t) =p  \int_{M}(1+|F_A|_h^2)^{\frac{p}{2}-1} \left( |\diff_{A} a|_h^2 + \langle F_{A}, [a, a] \rangle_h + (p-2) \frac{\langle F_A, \diff_A a \rangle_h^2}{1+|F_A|_h^2} \right)\mathrm{vol}_h.  \label{variationseconde}
		\end{equation}
	\end{proposition} 
	
	\begin{proof}
		From \[ \diff_{A + a} b = \diff_{A} b + [a, b], \] \[ F_{A + a} - F_{A} = \diff_{A} a + a \wedge a, \] and the embedding $\mathrm{W}^{1, 2} \hookrightarrow \mathrm{L}^{4,2}$, we deduce, by substituting in \eqref{variationpremiere} : 
		\begin{align*}
			\mathrm{D} \mathcal{YM}_{A + a} (b) & =p\langle  (1+|F_A+d_Aa+a\wedge a|_h^2)^{\frac{p}{2}-1}( F_{A} + \diff_{A}
			a + a \wedge a), \diff_{A} b + [a, b] \rangle_h\\
			& = \mathrm{D} \mathcal{YM}_{A} (b) +p(p-2)\langle \left(1+\vert F_A\vert^2_h\right)^{\frac{p}{2}-2}\langle F_A, d_Aa\rangle_h F_A, d_Ab\rangle_h  \\
			&+p \langle  \left(1+\vert F_A\vert^2_h\right)^{\frac{p}{2}-1}F_{A}, [a, b] \rangle_h +p
			\langle  \left(1+\vert F_A\vert^2_h\right)^{\frac{p}{2}-1}\diff_{A} a, \diff_{A} b \rangle_h \\
			&+ \gdo{ \| a
				\|^2_{\mathrm{W}^{1, 2}} \| b\|_{\mathrm{W}^{1, 2}}}
		\end{align*}
		which implies the result.
		
	\end{proof}

	\begin{definition} $ Q_{p,A}$ is the quadratic form defined by the right hand side of \eqref{variationseconde}.
	\end{definition}
	
	A consequence of the gauge invariance is that $Q_{p,A}$ has a huge kernel, as shown in the following result :
	\begin{proposition}
		\[ \Imag{\diff_{A}} \subset \ker Q_{p,A}\]
	\end{proposition}
	
	\begin{proof} 
		Identical to proposition 2.3 of \cite{GL24}.
	\end{proof}

	\noindent Recall the index of a quadratic form $q$ is the non-negative integer $\ind q$ defined as \[ \ind q = \sup \left\{ \dim W \mid q_{|W} < 0\right\}.\]
	
	\begin{proposition} Let $\mathfrak{Q}_{p,A}(a) = Q_{p,A}(a) +\int_{M}(1+|F_A|_h^2)^{\frac{p}{2}-1}|\diff^*_A  a|_h^2 \,\mathrm{vol}_h$. The index and kernel of $Q_{p,A}$ can be expressed from those of $\mathfrak{Q}_{p,A}$ as follows : \begin{align*}
			\ind Q_{p,A} &= \ind {\mathfrak{Q}_{p,A}}\\ 
			\ker Q_{p,A} &=  \ker {\mathfrak{Q}_{p,A}} \oplus  \Imag \diff_{A}  
		\end{align*}
	\end{proposition}
	
	\begin{proof} 
		Identical to proposition 2.4 of \cite{GL24}.
	\end{proof}

	\begin{propdef} \label{finitudeindice} $\ind^0\mathfrak{Q}_{p,A}:= \ind {\mathfrak{Q}_{p,A}} + \dim  \ker {\mathfrak{Q}_{p,A}}$ is finite and \[\ind^0 \mathfrak{Q}_{p,A} = \ind {Q_{p,A}} + \dim \left( \ker {Q_{p,A}}\cap {\ker \diff_{A}^*}\right).\] We call this quantity the extended index of the quadratic form $Q_{p,A}$.
	\end{propdef}

	\begin{proof} 
		The equality comes from the previous proposition. We shall then prove the finiteness of $\ind^0\mathfrak{Q}_{p,A}$. Let $\mathcal{E}(\lambda)$ be the (possibly trivial) eigenspace associated with the eigenvalue $\lambda$ of the non-negative elliptic operator $\diff_{A}^* \diff_{A} + \diff_{A} \diff_{A}^*$ . For all $\Lambda\in\R$, we denote \[ \mathcal{E}^{\Lambda} = \bigoplus_{\lambda \leq \Lambda} \mathcal{E}(\lambda).
		\]
		For all $\Lambda \in \R$, $\dim \mathcal{E}^\Lambda < +\infty$. We choose $\Lambda > C \Vert F_{A} \Vert_{\mathrm{L}^\infty(M)}$ where $C$ is such that \[\left| \int_M \langle F_{A}, [a, a] \rangle_h \mathrm{vol}_h \right| \leq C  \int_M |F_{A}|_h\, |a|^2_h \mathrm{vol}_h.\]
		
		\noindent If $a\in  \left(\mathcal{E}^\Lambda\right)^\perp$, then
		\begin{align*}
			\mathfrak{Q}_{p,A}(a) &\geq  \int_M (1+|F_A|_h^2)^{\frac{p}{2}-1}\left( \vert d^*_Aa\vert_h^2+\vert d_Aa\vert^2_h +\langle F_A, [a,a]\rangle_h\right) \mathrm{vol}_h\\
			&\geq  \int_M \vert d^*_Aa\vert^2_h+\vert d_Aa\vert^2_h +\langle F_A, [a,a]\rangle_h \mathrm{vol}_h\\
			&\geq \Lambda \int_M  |a|^2_h \mathrm{vol}_h + \int_M \langle F_{A}, a\wedge a \rangle_h \mathrm{vol}_h\\
			&\geq \int_M \left(\Lambda - C|F_{A}|_h\right) |a|^2_h \mathrm{vol}_h.
		\end{align*} We conclude that if $W$ a subspace on which $\mathfrak{Q}_{p, A}$ is negative-definite, $W\cap \left(\mathcal{E}^\Lambda\right)^\perp = \{0\}$. Consequently, $\dim W \leq \mathrm{codim}\left(\mathcal{E}^\Lambda\right)^\perp = \dim \mathcal{E}^\Lambda <+\infty$ so $\ind \mathfrak{Q}_{p,A} < +\infty$. A similar argument can be applied to prove that $\ker \mathfrak{Q}_{p,A}$ is also finite dimensional. 
	\end{proof}

	\section{$\varepsilon$-regularity, $\mathrm{L}^2$ and $\mathrm{L}^{2,1}$ quantization}
	We state here the $\varepsilon$-regularity pointwise estimate, the proof of which is to be found in the appendix. Then we give the $L^2$ quantization, also proved in \cite{HS}, but we do it in a different setting which permits us to improve it to the stronger $L^{2,1}$ quantization.
	\subsection{$\varepsilon$-regularity}
	\begin{theorem}[\cref{epsregcurvaturepointwiseproof}] \label{epsregbis}
		Let $G$ be a compact Lie group. There exist $\varepsilon_{G}>0$ and constants $C_{G},p_G>0$ such that for all $p\in[2,2+p_G]$, for all $R\in]0,1]$ and for all $1$-form $A$ in  $\mathrm{W}^{1,2}\left(\mathrm{B}_R,\Lambda^1 \R^4\otimes  \mathfrak{g}\right)$ satisfying the small energy condition:
		$$\int_{ \mathrm{B}_R}\left|F_A\right|^2 \diff x <\varepsilon_{G},
		$$ and the $p$-Yang-Mills equation \eqref{equationdepYM}, then the following estimate holds: \begin{equation}
			R^4\left\|F\right\|_{\mathrm{L}^{\infty}\left(\mathrm{B}_{R/2}\right)}^2 \leq C_{G} \int_{\mathrm{B}_R}\left|F_A\right|^2 \diff x. 
		\end{equation}
	\end{theorem}

	\subsection{$\mathrm{L}^{2,\infty}$ bound on the curvature}
	Once one has $\varepsilon$-regularity, it is classical to derive the weak-$L^2$ quantization.
	
	\begin{lemma} There exists $ C_G> 0$ and $p_G>0$ such that for all $p\in[2,2+p_G]$ and for all $p$-Yang-Mills connection $A$ in $\mathrm{B}_R\backslash \overline{\mathrm{B}_r} $ with $0<4r<R<1$, if $A$ satisfies
		\[ \sup_{r \leq \rho \leq R /2} \| F_A \|_{\mathrm{L}^2(\mathrm{B}_{2\rho} \backslash \mathrm{B}_\rho)} \leq \varepsilon_{G},\] then \[ \|F_A\|_{\mathrm{L}^{2,\infty}(\mathrm{B}_{R}\backslash \mathrm{B}_{r})} \leq C_G \sup_{r \leq \rho \leq R/2} \| F_A \|_{\mathrm{L}^2(\mathrm{B}_{2\rho}\backslash \mathrm{B}_\rho)}.\]
	\end{lemma}
	It is well known that the dyadic energy goes to zero in the neck, in fact it is a direct consequence of the bubble-neck  decomposition, see section 3.1 of \cite{HS} or proof of theorem VII.3 of \cite{rivière2015variations} for the classical case. Then we obtain the weak-$L^2$ quantization as a corollary.\\
	
	\begin{proof} Denote $\varepsilon =\displaystyle \sup_{r \leq \rho \leq R/2} \| F_A \|_{\mathrm{L}^2(\mathrm{B}_{2\rho}\backslash \mathrm{B}_\rho)} $ and apply the $\varepsilon$-regularity result \eqref{epsregbis} on balls $\mathrm{B}(x,|x|/3)$ for $x\in \mathrm{B}_{R/2}\backslash \mathrm{B}_{2r}$: there exists $C>0$ such that for all $x\in \mathrm{B}_{R/2}\backslash \mathrm{B}_{2r}$, \[|x|^2|F_A|(x) \leq C \| F_A
		\|_{\mathrm{L}^2(\mathrm{B}(x,|x|/3)}\leq C \| F_A
		\|_{\mathrm{L}^2(\mathrm{B}_{4|x|/3}\backslash\mathrm{B}_{2|x|/3})}\leq C\varepsilon. \] Therefore \begin{align*}
			\|F_A\|_{\mathrm{L}^{2,\infty}(\mathrm{B}_{R/2}\backslash \mathrm{B}_{2r})} &\leq  \|C\varepsilon |x|^{-2}\|_{\mathrm{L}^{2,\infty}(\mathrm{B}_{R/2}\backslash \mathrm{B}_{2r})} \\
			&\leq C\varepsilon\||x|^{-1}\|_{\mathrm{L}^{4,\infty}(\mathrm{B}_{R/2}\backslash \mathrm{B}_{2r})} ^2 \\
			& \leq C\varepsilon\||x|^{-1}\|_{\mathrm{L}^{4,\infty}(\R^4)}^2\\
			&  \leq C|\mathrm{B}_1|^{1/2}\varepsilon.
		\end{align*}
	\end{proof}
	
	\subsection{$\mathrm{L}^{2,1}$ bound on the curvature}

	To have $\mathrm{L}^{2,1}$ curvature bound, we need to use a refined version of Uhlenbeck \cite[Corollary 2.2]{UhlenbeckKarenK1982CwLb} with an hypothesis of small curvature in $\mathrm{L}^{2,\infty}$ in the neck region, which is due to Rivière \cite{rivière2015variations, riviere2002interpolation}. 
	\begin{theorem} There exist $\varepsilon_G, C_G> 0$ with the following property: for all $R,
		r $ with $0<2r < R <1$, if $A \in \mathrm{W}^{1, 2} (\mathrm{B}_R \backslash
		\overline{\mathrm{B}_r}, \Lambda^1 \R^4 \otimes \mathfrak{g})$
		satisfies \[\| F_A \|_{\mathrm{L}^{2, \infty}
			(\mathrm{B}_R \backslash \overline{\mathrm{B}_r})} + \| F_A \|_{\mathrm{L}^2 (\mathrm{B}_{2 r} \backslash
			\overline{\mathrm{B}_{r}})} \leq \varepsilon_G\] then there exists $g
		\in \mathrm{W}^{2, 2} (\mathrm{B}_R \backslash \overline{\mathrm{B}_r}, G)$
		such that $\diff^* A^g = 0$ and
		\[ \int_{\mathrm{B}_R\backslash
			\overline{\mathrm{B}_r}}  \left( | \nabla A^g |^2 + \frac{ | A^g |^2}{|x|^2}  \right) \diff x
		\leq C_G \int_{\mathrm{B}_R \backslash \overline{\mathrm{B}_r}} | F_A
		|^2 \diff x. \]
		\label{recollementjauge}
	\end{theorem}
	\begin{proof}
		Combine \cite[Lemma 4.4]{uhlenbeck_chern_1985} to extend the connection up to gauge in the whole ball $\mathrm{B}_R$ with curvature $\mathrm{L}^2$-small in the ball $\mathrm{B}_r$. This implies that the curvature is small in $\mathrm{L}^{2,\infty}$ in the whole ball $\mathrm{B}_R$ and \cite[Theorem IV.4]{rivière2015variations} applies.
	\end{proof}

	\begin{remark}
		In the limiting case $r=R/2$, the hypothesis reduces to \[\| F_A \|_{\mathrm{L}^2 (\mathrm{B}_{R} \backslash
			\overline{\mathrm{B}_{R/2}})} \leq \varepsilon_0\] and the conclusion is a form of Uhlenbeck theorem which applies on the dyadic annulus $\mathrm{B}_R \backslash \overline{\mathrm{B}_{R/2}}$.
	\end{remark}
	
		
	\
	
	\noindent We are now able to prove the crucial part of this section:
	
	\begin{proposition} There exists $\varepsilon_G,p_G, C_G> 0$ such that, for all $p\in[2,2+p_G]$ and all $p$-Yang-Mills connection $A\in W^{1,2} (\mathrm{B}_R\backslash \overline{\mathrm{B}_r} ,  \Lambda^1 \R^4 \otimes \mathfrak{g})$, for all $R,r$ with $0<r<4r<R<1$, if $A$ satisfies
		\[ \sup_{r \leq \rho \leq R /2} \| F_A \|_{\mathrm{L}^2(\mathrm{B}_{2\rho} \backslash \mathrm{B}_\rho)} \leq \varepsilon, \] then
		\[ \|  F_{A} \|_{\mathrm{L}^{2,1} (\mathrm{B}_{R/2} \backslash \mathrm{B}_{  2r})} \leq C \|  F_{A} \|_{\mathrm{L}^2 (\mathrm{B}_{R} \backslash \mathrm{B}_{ r})}.\]
		\label{inegalitequantificationYM}
	\end{proposition}


	\begin{proof}
		The proof is inspired by \cite[lemma III.2]{riviere2002interpolation} on $\mathrm{L}^2$-quantization for Yang-Mills connections. Let's use the same notations as in \eqref{recollementjauge}, up to a gauge transformation, we can assume that the conclusion of this theorem is satisfied by $A$. 
		
		\noindent Let $\chi \in \mathcal{C}^{\infty}_c (\R^4, [0 ; 1])$ such that $\supp \chi \subset \text{B}_{R} \backslash \text{B}_{r}$, $\chi \equiv 1$ in $\text{B}_{R/2} \backslash
		\text{B}_{2r}$ and $\| | x | \nabla \chi
		\|_{\text{L}^{\infty}} \leq C$, with $C$ independent of $R,r$. 
		
		\noindent From Bianchi identity \eqref{equationdeBianchi}, we have $\diff (\chi F_A) = - \chi [A, F_A] + \diff
		\chi \wedge F_A$. Using Hölder inequality (Lorentz spaces version, see \cite[theorem 4.5]{Hunt1966}) and the improved Sobolev-embedding $\dot{\mathrm{W}}^{1,2}(\R^4) \hookrightarrow \mathrm{L}^{4,2}(\R^4)$ (\cite[theorem 8.1]{peetrelorentzsobolev}): 
		
		\begin{equation}
			\label{majorationL4/3}
			\begin{split}
				\| \chi [A, F_A] \|_{\text{L}^{4 / 3, 1} (\R^4)} & = 
				\| [\chi A, F_A] \|_{\text{L}^{4 / 3, 1}
					(\R^4)}\\
				& \leq C\| \chi A \|_{\text{L}^{4, 2} (\R^4)}  \|
				F_A \|_{\mathrm{L}^2 (\mathrm{B}_R \backslash \mathrm{B}_r)}\\
				& \leq C  \| \nabla(\chi A) \|_{\text{L}^{2} (\R^4)} \|
				F_A \|_{\mathrm{L}^2 (\mathrm{B}_R \backslash \mathrm{B}_r)}. 
			\end{split}
		\end{equation}
		We can use the fact we choose a good gauge to prove the following estimate \begin{equation}
			\label{gradientcutoff}
			\begin{split}
				\| \nabla(\chi A) \|_{\text{L}^{2} (\R^4)}^2 &\leq C \int_{\R^4} |\chi \nabla A|^2 + |\nabla \chi|^2| A |^2 \diff x \\
				&\leq C \int_{\text{B}_{R} \backslash \text{B}_{r}} |\nabla A|^2 + 
				\frac{1}{|x|^2} | A |^2 \diff x \\
				&\leq C \int_{\mathrm{B}_R \backslash
					\mathrm{B}_{r}} | F_A |^2 \diff x.
			\end{split}
		\end{equation} 
		
		
		\noindent Since $\diff \chi$ has support in the annuli 
		$(\mathrm{B}_{R} \backslash \mathrm{B}_{R/2})\cup (\mathrm{B}_{2r} \backslash \mathrm{B}_{r})$, using Hölder inequality we obtain
		\begin{align}    
			\| \diff \chi \wedge F_A \|_{\text{L}^{4 / 3, 1}
				(\R^4)}& = \| \diff \chi \wedge F_A \|_{\text{L}^{4 / 3, 1}
				(\mathrm{B}_{R} \backslash \mathrm{B}_{r})}\\ &\leq C {\| F_A \|_{\mathrm{L}^2 (\mathrm{B}_R \backslash \mathrm{B}_{r})}}  \| \diff \chi \|_{\text{L}^{4, 2}\left(\R^4 \right)} \notag \\
			& \leq C {\| F_A \|_{\mathrm{L}^2 (\mathrm{B}_R \backslash \mathrm{B}_{r})}} \left( \| |x|^{-1}\|_{\text{L}^{4, 2}(\mathrm{B}_{R} \backslash \mathrm{B}_{R/2})} +  \| |x|^{-1}\|_{\text{L}^{4, 2}(\mathrm{B}_{2r} \backslash \mathrm{B}_{r})}\right)\notag \\
			&\leq C {\| F_A \|_{\mathrm{L}^2 (\mathrm{B}_R \backslash \mathrm{B}_{r})}} \| |x|^{-1}\|_{\text{L}^{4, 2}\left(\mathrm{B}_{1} \backslash \mathrm{B}_{1/2}\right)} \label{scalintL42}\\ &\leq C {\| F_A \|_{\mathrm{L}^2 (\mathrm{B}_R \backslash \mathrm{B}_{r})}}. \label{cutoffcourbureestimate}     
		\end{align} To prove \eqref{scalintL42}, it is crucial that the $\mathrm{L}^{4,2}$ norm has the same scaling as the $\mathrm{L}^{4}$ norm\footnote{It is clear, using a dilation as a change of variables, that $\| |x|^{-1}\|_{\text{L}^{4}(\mathrm{B}_{R} \backslash \mathrm{B}_{R/2})} = \| |x|^{-1}\|_{\text{L}^{4}(\mathrm{B}_{1} \backslash \mathrm{B}_{1/2})}$. }. Combining \eqref{majorationL4/3}, \eqref{gradientcutoff} and \eqref{cutoffcourbureestimate}, we obtain \begin{equation}
			\| \diff (\chi F_A) \|_{\text{L}^{4 / 3, 1}
				(\R^4)} \leq C {\| F_A \|_{\mathrm{L}^2 (\mathrm{B}_R \backslash
					\mathrm{B}_{r})}}. \label{destimate}
		\end{equation} We now use the $p$-Yang-Mills equation \eqref{equationdepYM} to show the following identity \begin{align*}
			\diff (\star (\chi F_A) ) &= \diff
			\chi \wedge \star F_A - \chi [A, \star F_A]  -(p-2) \langle \chi\nabla F_A, F_A \rangle \wedge  \star \frac{F_A}{1+|F_A|^2}\\
			&= \diff
			\chi \wedge \star F_A - \chi [A, \star F_A]  -(p-2)\frac{|F_A|^2}{1+|F_A|^2} \left( \left\langle \nabla (\chi F_A), \frac{F_A}{|F_A|} \right\rangle \wedge  \star \frac{F_A}{|F_A|} + \diff \chi \wedge \star F_A \right),
		\end{align*} and by a similar reasoning than the one above we have \[ \| \diff\chi \wedge \star F_A \|_{\text{L}^{4 / 3, 1}
			(\R^4)} + \| \chi [A, \star F_A] \|_{\text{L}^{4 / 3, 1}
			(\R^4)} \leq C \| F_A \|_{\text{L}^{2}
			(\mathrm{B}_R \backslash \mathrm{B}_{r})} \] 
		so, using again \eqref{cutoffcourbureestimate}, we get
		 
		\begin{equation}\label{dstarestimate}
		\begin{split}
			\|\diff^{\ast} (\chi
			F_A) \|_{\text{L}^{4 / 3, 1} (\R^4)} 
			&\leq C {\| F_A
				\|_{\mathrm{L}^2 (\mathrm{B}_R \backslash \mathrm{B}_{r})}} +C(p-2){\| \nabla (\chi F_A)
				\|_{\mathrm{L}^{4/3,1} (\mathrm{B}_R \backslash \mathrm{B}_{r})}}\\ 
			&\leq C {\| F_A
				\|_{\mathrm{L}^2 (\mathrm{B}_R \backslash \mathrm{B}_{r})}} + C(p-2){\| \nabla (\chi F_A)
				\|_{\mathrm{L}^{4/3,1} (\R^4)}}. 		
		\end{split}
		\end{equation}
		We can now use Gaffney lemma (\cref{gaffney}) for $\chi F_A$, and thanks to \eqref{destimate} and \eqref{dstarestimate}, we obtain the following inequality
		\begin{align*}
			\| \nabla (\chi F_A) \|_{\text{L}^{4 / 3, 1} (\R^4)} 
			& \leq C  \left(  \| \diff (\chi F_A) \|_{\text{L}^{4 / 3,
					1} \left( \R^4  \right)} + \| \diff^{\ast} (\chi F_A)
			\|_{\text{L}^{4 / 3, 1} \left( \R^4 \right)}\right) \\
			& \leq C  \| F_A \|_{\mathrm{L}^2 (\mathrm{B}_R \backslash
				\mathrm{B}_{r})} + C(p-2){\| \nabla (\chi F_A)
				\|_{\mathrm{L}^{4/3,1} (\R^4)}}. 
		\end{align*} For $p\leq 2+ p_G:=2+1/2C$, we have \[  \| \nabla (\chi F_A) \|_{\text{L}^{4 / 3, 1} (\R^4)}\leq C  \| F_A \|_{\mathrm{L}^2 (\mathrm{B}_R \backslash
			\mathrm{B}_{r})}.\] By another Lorentz-Sobolev embedding \cite[theorem 8.1]{peetrelorentzsobolev}, we have $\dot{\mathrm{W}}^{1,(4/3,1)}(\R^4) \hookrightarrow \mathrm{L}^{2,1}(\R^4)$ and we deduce
		\[ \| \chi F_A \|_{\text{L}^{2, 1} (\R^4)} \leq C {\|
			F_A \|_{\mathrm{L}^2 (\mathrm{B}_R \backslash \mathrm{B}_{  r})}}  \] which implies 
		\[ \| F_A \|_{\text{L}^{2, 1} \left( \text{B}_{R / 2} \backslash
			\text{B}_{2 r} \right)} \leq C {\| F_A \|_{\mathrm{L}^2
				(\mathrm{B}_R \backslash \mathrm{B}_{r})}}.  \] 
		
	\end{proof}

	Using the $\mathrm{L}^{2,\infty}/\mathrm{L}^{2,1}$ duality, we finally obtain the full quantization.
	
	\begin{corollary} \label{L2quantization}
		There exists $\varepsilon_G, C_G> 0$ and $p_G>0$ such that for all $p\in[2,2+p_G]$ and for all $p$-Yang-Mills connection $A$ in $\mathrm{B}_R\backslash \overline{\mathrm{B}_r} $, for all $R,r$ with $0<r<4r<R<1$, if $A$ satisfies
		\[ \sup_{r \leq \rho \leq R /2} \| F_A \|_{\mathrm{L}^2(\mathrm{B}_{2\rho} \backslash \mathrm{B}_\rho)} \leq \varepsilon_{G}',\] then \[ \|F_A\|_{\mathrm{L}^{2}(\mathrm{B}_{R}\backslash \mathrm{B}_{r})} \leq C  \sup_{r \leq \rho \leq R/2} \| F_A \|_{\mathrm{L}^2(\mathrm{B}_{2\rho}\backslash \mathrm{B}_\rho)}\] and
		\[ \|F_A\|_{\mathrm{L}^{2,1}(\mathrm{B}_{R/2}\backslash \mathrm{B}_{2r})} \leq C  \sup_{r \leq \rho \leq R/2} \| F_A \|_{\mathrm{L}^2(\mathrm{B}_{2\rho}\backslash \mathrm{B}_\rho)}.\] 
	\end{corollary}
	
	Hence Theorem \ref{bubbletree} follows directly from this corollary and the classical bubble-neck decomposition, see section 3.1 of \cite{HS} or proof of theorem VII.3 of \cite{rivière2015variations} for the classical case.
	
	\section{Pointwise estimates in neck regions}
	
	The need for pointwise estimates improving those provided by $\varepsilon$-regularity in neck regions was already addressed in \cite[Section 3]{GL24} for the case $p = 2$. In the broader setting of $p$-Yang-Mills connections, similar estimates are expected to hold. However, in what follows, our arguments require estimates that are uniform with respect to $p$. In this section, we derive such pointwise bounds with enough precision to control their dependence as $p \to 2$.
	
	\subsection{An elliptic differential inequality in non-divergence form for the
		curvature}
	
	Consider $0<r<R<1$, $A_{r,R} = \text{B}_R \backslash \overline{\text{B}_r}$ and
	$A$ a solution of the $p$-Yang-Mills equation on $A_{r,R}$. 
	
	\begin{definition}
		Let $\mathcal{A}: A_{r,R}\rightarrow \End(\Lambda^1(A_{r,R})) $ defined by 
		\[  \langle \omega', \mathcal{A}\omega \rangle = \frac{\langle \omega \wedge \star
			F_A, \omega' \wedge \star
			F_A\rangle}{1 + | F_A |^2} \text{ for all } \omega'\in (\Lambda^1(A_{r,R})) . \]
	\end{definition} 
	\begin{remark}
		\label{remA}
		It is clear that $| \mathcal{A} | \leq 1$ and $| \nabla \mathcal{A} |
		\leq C \frac{| \nabla F_A |}{\sqrt{1 + | F_A |^2}}$.
	\end{remark}

	\begin{proposition} 
		\label{LpF}
		There exists $C>0$ such that for all $\gamma> 0$,
		\begin{equation}
			\mathcal{L}_p | F_A |^{\gamma} \leq C \gamma | F_A |^{\gamma+1} + \gamma
			| F_A |^{\gamma - 2} (\kappa (p, \gamma) | \diff | F_A | |^2 -(1-(p-2)^2) | \nabla F_A
			|^2) \label{pYMIEDP}
		\end{equation}
		where $\mathcal{L}_p u = \diff^{\ast} ((\mathrm{id}- (p - 2) \mathcal{A})
		\diff u)$ and $\kappa (p, \gamma) = (p-1) (2 - \gamma)_+$.
	\end{proposition}
	\begin{proof} Apply \cref{Bochnerpower} with $\beta=\gamma/2$.
	\end{proof}

	\begin{definition}
		The operator $\mathcal{L}'_p$ is defined as
		\[ \mathcal{L}_p' u =\mathcal{L}_p u - (p - 2) \nabla \mathcal{A} \cdot
		\nabla u = \Delta u + (p - 2) \mathcal{A} \cdot \nabla^2 u \]
		where we used the notations $\nabla \mathcal{A} \cdot \nabla u :=
		\partial_{\alpha} \mathcal{A}^{\alpha \beta} \partial_{\beta} u$ and
		$\mathcal{A} \cdot \nabla^2 u := \mathcal{A}^{\alpha \beta} {\partial
		}^2_{\alpha \beta} u$.
	\end{definition}
	
	Both expressions for $\mathcal{L}_p'$ will be useful in the following. Let's
	first derive an estimate for $\mathcal{L}_p' | F_A |^{\gamma}$.
	
	\begin{proposition} \label{propIEDP}
		There exist a constant $C>0$ and $p_G>0$ such that for all $p\in[2,2+p_G]$, there exists $\gamma = \gamma(p) = \frac{1}{2} + O(p-2)$ such that $\gamma(p)>1/2$ for $p>2$ and for all $p$-Yang-Mills connection with curvature $F_A$, the following holds
		\begin{equation}
			\mathcal{L}_p' | F_A |^{\gamma}\leq C | F_A |^{\gamma+1}.  \label{IEDPpYM}
		\end{equation}
	\end{proposition}
	\begin{proof} For $\gamma>0$, thanks to \cref{remA} and \cref{LpF} we have

		\begin{align*}
			\mathcal{L}_p' | F_A |^{\gamma} &\leq (p-2) \nabla \mathcal{A} \cdot \nabla |F_A|^\gamma  + C \gamma | F_A |^{\gamma + 1} + \gamma
			| F_A |^{\gamma - 2} (\kappa (p, \gamma) | \diff | F_A | |^2 -(1-(p-2)^2) | \nabla F_A
			|^2) \\
			&\leq C \gamma | F_A |^{\gamma+1} + \gamma
			| F_A |^{\gamma - 2} \left(\kappa (p, \gamma) | \diff | F_A | |^2 -(1-(p-2)^2) | \nabla F_A
			|^2 +C(p-2)\vert \nabla |F_A|\vert |\nabla F_A| \right)\\
			&\leq C \gamma | F_A |^{\gamma+1} + \gamma
			| F_A |^{\gamma - 2} \left(\kappa (p, \gamma) | \diff | F_A | |^2 - (1-C(p-2)-(p-2)^2)|\nabla F_A|^2\right)
		\end{align*} We now define $\gamma(p)$, when $p-2$ is small enough, to be the solution of the following equation \[\mu(p) = \frac{\kappa(p,\gamma)}{1-C(p-2)-(p-2)^2}\] where $\mu(p)$ is the constant of the Kato-Yau inequality for $p$-Yang-Mills connections (see \cref{KatoYau}), such that
		\begin{equation}
			\label{mup}
			\mu(p)|\diff |F_A||^2  \leq |\nabla_A F_A|^2 .
		\end{equation}
		More explicitly, \[ \gamma(p) = 2-\mu(p) \frac{1-C(p-2)-(p-2)^2}{p-1}.\] Notice that, since $\mu(p) = 3/2 +O(p-2)$, $\gamma(p) = 1/2 +O(p-2)$ when $p\to 2$ and since for $p>2$, $\mu(p)<3/2$, for $p-2$ small enough, then $\gamma(p)>2-3/2=1/2$. With this choice of $\gamma$, we obtain, for $p$ small enough, by \eqref{mup}, we have
		\begin{align*}
			\mathcal{L}_p' | F_A |^{\gamma} &\leq  C \gamma | F_A |^{\gamma+1} + \gamma
			| F_A |^{\gamma - 2}(1-C(p-2)-(p-2)^2) \left(\mu (p) | \diff | F_A | |^2 - | \nabla F_A
			|^2 \right)\\
			&\leq  C \gamma | F_A |^{\gamma+1} .
		\end{align*}
	\end{proof}

	\subsection{A comparison principle and supersolutions candidates}
	
	Here, we are interested in showing a comparison principle on $A_{r,R}$ for the operator
	\[\mathcal{L}'_{p,\varepsilon} := \mathcal{L}_p' - \varepsilon | x |^{- 2} \text{ with } \varepsilon \geq 0.\] We are going to use the following property:
	\begin{proposition} \label{comparisonprinciplegeneral} Let $\Omega$ a domain of $\R^n$ and consider an elliptic operator $L$  of the form \[ Lf=-\partial_\alpha(a^{\alpha\beta}\partial_\beta f) + b^{\beta} \partial_\beta f +cf \] with coefficients in $\mathrm{L}^\infty(\Omega)$. Assume that $L$ admits a positive supersolution,  \textit{i.e.} there exists $w>0$ such that $Lw\geq 0$. If $w\in \mathcal{C}^2(\overline{\Omega})$ then the operator $\tilde{L}$ defined as \[\tilde{L} \tilde{f} = w^{-1} L(w \tilde{f}) \] satisfies the maximum principle. In particular, every subsolution of $L$ that is non-positive on $\partial \Omega$ is non-positive in $\Omega$. 
	\end{proposition}
	
	\begin{proof} This proof can be found in \cite[Theorem 2.11]{han_elliptic_2000} for strong solutions.
		It is easy to check that \[\tilde{L}\tilde{f} = -\partial_\alpha(a^{\alpha\beta}\partial_\beta \tilde{f}) + \tilde{b}^{\beta} \partial_\beta \tilde{f} +\tilde{c} \tilde{f} \] with \begin{align*}
			\tilde{b}^\beta &= b^\beta - a^{\alpha\beta} w^{-1}\partial_\alpha w- a^{\alpha\beta} w^{-1}\partial_\beta w  \\
			\tilde{c} &= w^{-1} L w \geq 0.
		\end{align*} It is now clear that that $\tilde{L}$ satisfies the usual maximum principle property (see for instance \cite[Theorem 8.1]{gilbarg1977elliptic}).
	\end{proof}

	To use this proposition, we need positive supersolutions for $\mathcal{L}_{p,\varepsilon}'$. We are going to look for some of the form $\rho^{-\sigma}$, where $\rho=|x|$, for a suitable choice of $\sigma$.
	
	\begin{lemma}
		For $\sigma \in \R$, if $\rho = | x |$
		\[ - \nabla^2 \rho^{- \sigma} = \sigma \rho^{- (\sigma + 2)}  \left(
		\mathrm{Id} - (\sigma + 2) \frac{x x^{\top}}{\rho^2} \right). \]
		In particular
		\[  \mathcal{L}'_{p,\varepsilon} \rho^{- \sigma} = \rho^{- (\sigma + 2)} \left( \sigma (2 -
		\sigma) - \varepsilon + (p - 2) \sigma \mathcal{A}^{\alpha \beta} \left( \delta_{\alpha
			\beta} - (\sigma + 2) \frac{x_{\alpha} x_{\beta}}{| x |^2} \right) \right). \]
	\end{lemma}
	
	\begin{proof} Immediate from the formula $\mathcal{L}'_{p,\varepsilon} = \Delta +(p-2) \mathcal{A}\cdot \nabla^2 - \varepsilon \rho^{-2}$.
	\end{proof}	
	
	\begin{lemma} \label{supersolutionestimates} Assume $\sigma \geq 0$, then we have
		\[  - 2(p - 2) \sigma(\sigma +1) \leq\rho^{\sigma + 2} \mathcal{L}'_{p,\varepsilon} \rho^{- \sigma} - \left( \sigma (2 -
		\sigma) - \varepsilon \right)  \leq  2(p - 2) \sigma \]
	\end{lemma}
	
	\begin{proof}
		It is a consequence of the estimates of \cref{Aestimate} and the expression from the previous lemma.
	\end{proof}
	
	\begin{corollary} \label{supersolutionsLp''}
		Let $0\leq \varepsilon<1$, then there exists $p_G>0$ such that for $p\in[2,2+p_G]$ such that the operator $\mathcal{L}'_{p,\varepsilon}$ admits two positive supersolutions of the form $\rho^{-\delta_\pm(p,\varepsilon)}$ where $0\leq \delta_-(\varepsilon,p)<\delta_+(\varepsilon,p)$ are the two non-negative roots of the polynomial $X(2-X)-2(p-2)X(X+1)-\varepsilon$.
	\end{corollary}
	
	\begin{proof} If $\delta$ is a root of the above mentioned polynomial then using \cref{supersolutionestimates}, we get \[\mathcal{L}'_{p,\varepsilon}\rho^{- \delta}  \geq \rho^{-\delta - 2} \left( - 2(p - 2) \delta(\delta +1) + \delta (2 -
		\delta) - \varepsilon \right) =0.\]
		When $\varepsilon <1$ and $p-2$ is small enough, the polynomial does indeed have two non-negative roots: they are explicitly given by the following formula \[ \delta_{\pm}(\varepsilon,p) = \frac{3 - p}{1 + 2 (p - 2)} \left( 1 \pm \sqrt{1 -
			\varepsilon \frac{1 + 2 (p - 2)}{(3 - p)^2}} \right). \]
	\end{proof}
	
	\begin{proposition}
		\label{comparisonprinciple}  Let $0\leq \varepsilon<1$, then there exists $p_G>0$ such that for $p\in[2,2+p_G]$ such that the operator $\mathcal{L}'_{p,\varepsilon}$ satisfies the following comparison principle: if $\phi \in \mathrm{W}^{1,2}(A_{r,R})$ satisfies $\phi_{|\partial A_{r,R}} \leq 0$ and $\mathcal{L}'_{p,\varepsilon} \phi \leq 0$ then $\phi \leq 0$ in $A_{r,R}$.
	\end{proposition}
	\begin{proof}
		It is a consequence of \cref{comparisonprinciplegeneral} and \cref{supersolutionsLp''}.
	\end{proof}
	\subsection{Pointwise estimates for the curvature}

	\begin{theorem}\label{mainestimate}  There exists $\varepsilon_0,p_G, C> 0$ such that, for all $p\in[2,2+p_G]$ and all $p$-Yang-Mills connection $A$ on $\mathrm{B}_{2R}\backslash \mathrm{B}_{r/2}$, if $E = \| F_{A}
		\|_{\mathrm{L}^2(\mathrm{B}_{2R}\backslash \mathrm{B}_{r/2})}\leq \varepsilon_0$ then   \begin{equation}
			|F_{A}| \leq C E \omega_{p,R,r } \text{ in } D=\mathrm{B}_{R}\backslash \mathrm{B}_{r}
		\end{equation} \label{sharpestimateneck} where $\omega_{p,R,r }(x):= \frac{1}{| x |^2} \left( \left(
		\frac{| x |}{R} \right)^{ 2} + \left( \frac{r}{| x |} \right)^{ 2-\varepsilon_p}   \right)$ and $\varepsilon_p= O(p-2)$, with $\varepsilon_p>0$ for $p>2$.
	\end{theorem} 
	
	\begin{proof}
		Thanks to the $\varepsilon$-regularity estimates from \cref{epsregbis} we have  $| F | \leq C_0 E / | x
		|^2$ for all $x\in A_{r,R}$. Let $u := (| F | / C_0 E)^{\gamma}$ with $\gamma = \gamma (p)$ from \cref{propIEDP}, then by $\varepsilon$-regularity we have
		\begin{equation}
			\label{aprioriestimateespreg} u
			\leq | x |^{- 2 \gamma}.
		\end{equation} With these notations, \eqref{IEDPpYM}  becomes
		\[ \mathcal{L}'_{p,\varepsilon} u \leq 0, \] where $\varepsilon  = C_0 E\leq C_0 \varepsilon_0$. Writing $\delta_\pm := \delta_\pm (\varepsilon,p)$ (defined in \cref{supersolutionsLp''}), $\rho^{-\delta_+}$ and $\rho^{-\delta_-}$ are both supersolutions of $\mathcal{L}'_{p,\varepsilon}$. Considering a suitable linear combination of these two, we deduce using the comparison principle of \cref{comparisonprinciple} the estimate 
		\[ u \leq R^{- 2 \gamma} (\rho / R)^{- \delta_-} + r^{- 2 \gamma} (\rho
		/ r)^{- \delta_+}, \] since this inequality holds on $\partial A_{r,R}$, as soon as $C_0\leq 1$.\\

		Now, we can use this estimate again in \eqref{IEDPpYM} to deduce 
		\begin{equation}
			\mathcal{L}'_p u \leq C u^{1 + \frac{1}{\gamma}} \leq C R^{- 2
				(\gamma + 1)} (\rho / R)^{-  \left( 1 + \frac{1}{\gamma} \right)\delta_-} +
			C r^{- 2 (\gamma + 1)} (\rho / r)^{-  \left( 1 + \frac{1}{\gamma}
				\right)\delta_+} \label{intermediateestimate}
		\end{equation} We are now going to use the comparison principle for $ \mathcal{L}'_p =  \mathcal{L}'_{p,0}$. From \eqref{supersolutionestimates}, we know that for $\sigma \geq 0$, \[  \mathcal{L}'_p (-\rho^{- \sigma}) \geq \rho^{-\sigma -2} \sigma\left( \sigma-2 -2(p-2) \right). \]
		Let's apply the previous inequality with $\sigma = \sigma_{\pm} :=
		\left( 1 + \frac{1}{\gamma} \right)\delta_{\pm} - 2$. Note that
		\begin{eqnarray*}
			\sigma_+ (\varepsilon,p) & \underset{(\varepsilon,p)\rightarrow (0,2)}{\rightarrow} & 3 \delta_+ (0, 2) - 2 = 4\\
			\sigma_- (\varepsilon,p) & \underset{(\varepsilon,p) \rightarrow (0,2)}{\rightarrow} & 3 \delta_- (0, 2) - 2 = - 2
		\end{eqnarray*}
		so taking $\varepsilon$ and $p - 2$ small enough, we can assume  \[  \mathcal{L}'_p (-C \rho^{- \sigma_\pm}) \geq \rho^{-\sigma_\pm -2},  \]  where $C>0$ is independent of $\varepsilon$ and $p$.Therefore, from \eqref{intermediateestimate}, the exists some $C > 0$ such
		that
		\[ \mathcal{L}'_p u \leq \mathcal{L}'_p  (- C \psi) \]
		where
		\[ \psi = R^{- 2 \gamma + \sigma_-} \rho^{- \sigma_-} + r^{- 2 \gamma +
			\sigma_+} \rho^{- \sigma_+} \geq 0 .\]
		Notice that, by \cref{supersolutionsLp''}, when $\nu$ is equal to $\delta_- (0, p) = 0$ or $\delta := \delta_+
		(0, p) = \frac{2 (3 - p)}{1 + 2 (p - 2)}$, we have  $\mathcal{L}'_p \rho^{- \nu}
		\geq 0$ so the following holds for all $C_1, C_2\geq 0$,
		\begin{equation}
			\mathcal{L}'_p u \leq \mathcal{L}'_p  (- C \psi + C_1 + C_2 \rho^{-
				\delta}) . \label{comparison+harmonic}
		\end{equation}
		Consider $H = R^{- 2 \gamma} + r^{\delta - 2 \gamma} \rho^{- \delta}$, then $H$ satisfies
		$H \geq | x |^{- 2 \gamma} \geq u$ on $\partial A_{r,R}$. Furthermore
		\[ H - \psi = R^{- 2 \gamma} \left( 1 - \left( \frac{R}{\rho}
		\right)^{\sigma_-} \right) + r^{- 2 \gamma} \left( \frac{r}{\rho}
		\right)^{\delta} \left( 1 - \left( \frac{\rho}{r} \right)^{\delta - \sigma_+}
		\right) . \]
		Since $\rho \leq R$ and $\sigma_- < 0$, \ the first term is non-negative.
		Now when $p \rightarrow 2$,
		\[ \delta - \sigma_+ \rightarrow 2 - 4 = - 2 < 0 \]
		and since $r \leq \rho$, for $p-2$ small enough, the second term is also
		non negative and $H \geq \psi$. Finally, taking suitable $C_1$ and $C_2$, we get
		in \eqref{comparison+harmonic}
		\[ \mathcal{L}'_p u \leq \mathcal{L}'_p  (- C \psi + (C + 1) H) \]
		and $- C \psi + (C + 1) H \geq H \geq u$ on $\partial A_{r,R}$. Using \cref{comparisonprinciple} as previously stated, the following holds in $D$:\[ u \leq - C \psi + (C + 1) H \leq (C+1) H,\] and therefore
		\[ | F | = C_0 E u^{1 / \gamma} \leq C E H^{1 / \gamma} \leq C E
		\left( \frac{1}{R^2} + \frac{r^{\delta/\gamma - 2}}{\rho^{\delta/\gamma}} \right). \] Since $\delta \leq 2$ and $\gamma \geq
		1 / 2,$ we have $\varepsilon_p:= 4-\delta/\gamma \geq 0$,
		\[ \varepsilon_p = O (p -
		2) \] and \[ | F | \leq C E
		\left( \frac{1}{R^2} + \frac{r^{2-\varepsilon_p}}{\rho^{4-\varepsilon_p}} \right) \text{ on } A_{r,R}. \]
	\end{proof}
	
	\begin{remark}
		It is natural to obtain different exponents for both terms of the weight since the $p$-Yang-Mills equation is not invariant under inversion anymore.
	\end{remark}
	
	\section{Contribution of the necks to the second variation}
	The goal of this section is to derive sharp point-wise estimate in the neck in order to prove that the  second variation when restricted to the neck is  is non-negative.\\
	
	First, as a direct consequence of theorem \ref{mainestimate}, we have the following estimate
	\begin{corollary}
		\label{CoroF}
		There exist $\varepsilon,C,p_G>0$, such that,  for all $p\in[2,2+p_G]$ and for all $R>r>0$ and for all $p$-Yang-Mills connection $A$ on $\mathrm{B}_{2R}\backslash \mathrm{B}_{r/2}$, if $E = \| F_{A}
		\|_{\mathrm{L}^2(\mathrm{B}_{2R}\backslash \mathrm{B}_{r/2})}\leq \varepsilon$ then for $x\in \mathrm{B}_{R}\backslash \mathrm{B}_{r}$, \begin{equation}
			|F_{A}| \leq C E \left(1+ \left(\frac{R}{r}\right)^{C(p-2)} \right) \omega_{R,r } 
		\end{equation} where $\omega_{R,r}:= \omega_{2,R,r }$.
	\end{corollary}
	From the $\varepsilon$-regularity estimate and \cref{sharpestimateneck}, we deduce the following result:
	\begin{corollary} \label{neckC0estimates} There exist $\varepsilon,C>0$, such that, for all $R>r>0$ and for all $p$-Yang-Mills connection $A$ in $\mathrm{B}_{2R}\backslash \mathrm{B}_{r/2}$, if $E = \| F_{A}
		\|_{\mathrm{L}^2(\mathrm{B}_{2R}\backslash \mathrm{B}_{r/2})}\leq \varepsilon$, there exists a gauge $g\in W^{1,(4,\infty)}(\mathrm{B}_{R}\setminus \mathrm{B}_{r})$ in which, for all $x\in \mathrm{B}_{R}\backslash \mathrm{B}_{r}$, \begin{equation}
			|F_A (x)| + |A^g(x)|^2  \leq C E \left(1+ \left(\frac{R}{r}\right)^{C(p-2)} \right) \omega_{R,r }(x).
		\end{equation}
	\end{corollary}

	\begin{proof} 
		According to \cref{recollementjauge}, there exists $g \in \mathrm{W}^{1,(4,\infty)}(\mathrm{B}_{2R},G)$ which is a Coulomb gauge for $A$ on $\mathrm{B}_{2R}\setminus \mathrm{B}_{r/2}$. Since for all $x\in \mathrm{B}_{R}\backslash \mathrm{B}_{r}$, $\mathrm{B}_{\rho/2}(x) \subset \mathrm{B}_{2\rho} \backslash \mathrm{B}_{\rho} \subset  \mathrm{B}_{2R}\backslash \mathrm{B}_{r/2}$ where $\rho = 2|x|/3$, the $\varepsilon$-regularity estimates from \cref{epsreg}  give
		\[ |x| |A^g(x)| \leq C \frac{\rho}{2}  \|A^g  \|_{\mathrm{L}^\infty (\mathrm{B}_{\rho/4}(x))} \leq C  \| F_A \|_{\mathrm{L}^2 (\mathrm{B}_{\rho/2}(x))}\leq C \| F_{A} \|_{\mathrm{L}^2 (\mathrm{B}_{2\rho}\backslash\mathrm{B}_\rho)}.\] From corollary \ref{CoroF}, we get the following energy estimate  
		\begin{align*}
			\| F_{A} \|_{\mathrm{L}^2 (\mathrm{B}_{2\rho}\backslash\mathrm{B}_\rho)} &\leq C E |\mathrm{B}_{2\rho}\backslash\mathrm{B}_\rho|^{1/2} \left(1+ \left(\frac{R}{r}\right)^{C(p-2)} \right)  \sup_{y\in\mathrm{B}_{2\rho}\backslash\mathrm{B}_\rho} \omega_{R,r}(y)\\
			&\leq C E |x|^2\left(1+ \left(\frac{R}{r}\right)^{C(p-2)} \right) \omega_{R,r}(x) , 
		\end{align*}
		for any $x$ such that $\mathrm{B}_{\rho/2}(x) \subset \mathrm{B}_{2\rho} \backslash \mathrm{B}_{\rho} \subset  \mathrm{B}_{2R}\backslash \mathrm{B}_{r/2}$.
		Finally, since $\omega_{R,r}(x)\leq 2/|x|^2$, we have
		\begin{align*} |A^g(x)|^2& \leq \frac{C}{|x|^2} \| F_{A} \|_{\mathrm{L}^2 (\mathrm{B}_{2\rho}\backslash\mathrm{B}_\rho)}^2\\
			&\leq C E |x|^2\omega_{R,r}(x) \times \omega_{R,r}(x) \left(1+ \left(\frac{R}{r}\right)^{C(p-2)} \right)^2\\
			& \leq 2C E \omega_{R,r}(x)\left(1+ \left(\frac{R}{r}\right)^{2C(p-2)} \right) ,
		\end{align*}
		which achieves the proof.
	\end{proof}
	
	\noindent Furthermore, we have the following inequalities: \begin{proposition} \label{GaffneyPoincareNeck}
		There exists $C>0$ such that for all $R>r>0$, for all $a\in \mathrm{W}^{1,2}_0(\mathrm{B}_{R}\backslash \mathrm{B}_{r},\Lambda^1(\mathrm{B}_{R}\backslash \mathrm{B}_{r}) \otimes\mathfrak{g})$ we have ,  \begin{equation}
			\int_{\mathrm{B}_{R}\backslash \mathrm{B}_{r}} |a|^2 \omega_{R,r } \diff x  \leq C \int_{\mathrm{B}_{R}\backslash \mathrm{B}_{r}} |\diff a|^2 + |\diff^* a|^2 \diff x. 
		\end{equation}
	\end{proposition}
	\begin{proof}
		It suffices to combine the inequalities from Gaffney inequality \cref{gaffney}, Hardy inequality \cref{Hardyi} and the fact that $\omega_{R,r}(x)\leq 2 |x|^{-2}$.
	\end{proof}
	We are now in position to prove the crucial estimate for the proof of our main theorem.
	
	\begin{theorem} \label{pointcrucial1} There exists $c_0>0$ and for every $B>0$, there exist $\varepsilon>0$, such that, for all $R>r>0$ and for all $p$-Yang-Mills connection $A$ in $\mathrm{B}_{2R}\backslash \mathrm{B}_{r/2}$, if $(p-2) \ln (R/r)\leq B$ and $E = \| F_{A}
		\|_{\mathrm{L}^2(\mathrm{B}_{2R}\backslash \mathrm{B}_{r/2})}\leq \varepsilon$ then for all $a\in \mathrm{W}^{1,2}_0(\mathrm{B}_{R}\backslash \mathrm{B}_{r},\Lambda^1(\mathrm{B}_{R}\backslash \mathrm{B}_{r})\otimes\mathfrak{g})$,we have 
		\begin{equation}\int_{\mathrm{B}_{R}\backslash \mathrm{B}_{r}}(1+|F_A|^2)^{p/2-1} \left( |\diff_A a|^2  + |\diff_A^* a|^2 + \langle F_A, [a, a]\rangle \right) \diff x \geq c_0  \int_{\mathrm{B}_{R}\backslash \mathrm{B}_{r}} |a|^2 \omega_{R,r } \diff x. \label{neckpositive}
		\end{equation}
	\end{theorem}
	
	\begin{proof} 
		Using the fact that, for any gauge $g$, $\diff_{A^g} a^g=\diff a^g +[A^g,a^g]$ and $\diff_{A^g}^* a^g=\diff^* a^g -\star[A^g,\star a^g]$, where $a^g:=g^{-1} a g$, we get :\begin{align*}
			\int_{\mathrm{B}_{R}\backslash \mathrm{B}_{r}}  |\diff a^g|^2 \diff x &\leq \int_{\mathrm{B}_{R}\backslash \mathrm{B}_{r}}  2\left(|\diff_{A^g} a^g|^2 + |[A^g,a^g]|^2 \right) \diff x \\
			& \leq \int_{\mathrm{B}_{R}\backslash \mathrm{B}_{r}}  2\left(|g^{-1}(\diff_{A} a)g|^2 + |[A^g,a^g]|^2 \right) \diff x \\
			&\leq  \int_{\mathrm{B}_{R}\backslash \mathrm{B}_{r}}  2\left(|\diff_{A} a|^2 + |[A^g,a^g]|^2 \right) \diff x 
		\end{align*} and similarly \[\int_{\mathrm{B}_{R}\backslash \mathrm{B}_{r}}  |\diff^* a^g|^2 \diff x \leq \int_{\mathrm{B}_{R}\backslash \mathrm{B}_{r}}  2\left(|\diff_{A}^* a|^2+ \left|\star[A^g,\star a^g]\right|^2 \right) \diff x. \] We apply this with the gauge from \cref{neckC0estimates} and combine it with \cref{GaffneyPoincareNeck} to get the following lower bound of the left-hand side of \eqref{neckpositive}: \begin{align*}
			&\int_{\mathrm{B}_{R}\backslash \mathrm{B}_{r}} (1+|F_A|^2)^{p/2-1} \left( |\diff_A a|^2  + |\diff_A^* a|^2 + \langle F_A, [a, a]\rangle \right) \diff x \\
			&\geq \int_{\mathrm{B}_{R}\backslash \mathrm{B}_{r}} \frac{1}{2}\left(|\diff a^g|^2+ |\diff^* a^g|^2\right) -C\left(|A^g|^2 +| F_{A} |\right)|a^g|^2\diff x \\
			&\geq \left( \frac{1}{2C} - CE\right) \int_{\mathrm{B}_{R}\backslash \mathrm{B}_{r}} |a^g|^2 \omega_{R,r} \diff x \\
			&\geq \left( \frac{1}{2C} - CE\right) \int_{\mathrm{B}_{R}\backslash \mathrm{B}_{r}} |a|^2 \omega_{R,r} \diff x.
		\end{align*} We complete the proof by taking $\displaystyle \varepsilon \leq \frac{1}{4C^2}$ and $\displaystyle c_0 = \frac{1}{4C}$.
	\end{proof}

	\section{Morse index semi-continuity}

	This section is dedicated to the study of the Morse index stability for $p$-Yang-Mills connections when $p\to 2$. To simplify notations and clarify the exposition, we consider the case where there is only one bubble. It is easy to deduce the general case by splitting the manifold into neck regions, bubbles and thick part and to treat each neck region independently as in the case of single bubble.  Hence we will only be considering the following special case of \eqref{bubbletree}: a sequence of $(p_k)$-Yang-Mills connections $(A_k)$ on a closed 4-manifold $M$, with $p_k\to 2$, which converges into $\mathfrak C_{G,\loc}^\infty(M\setminus \{q\})$ to $A_\infty \in \mathfrak{U}_G^2(M)$ a Yang-Mills connection and there exists points $q_k \rightarrow q $ and scales $\delta_k\rightarrow 0$ such that $\phi_k^* A_k$ converges into $\mathfrak C_{G,\loc}^\infty(\R^4)$ to $\widehat{A}_\infty \in \mathfrak{U}^2_G(M)$  where $\phi_k(y)= q_k+\delta_k y$ in local coordinates and $\widehat{A}_\infty=\pi^* \tilde A_\infty$ where $\tilde A_\infty \in \mathfrak{U}^2_G(S^4)$ is a Yang-Mills connections and $\pi$ is the stereographic projection.\\
	
	We first prove a an estimate between $p_k-2$ and the scale of concentration $\delta_k$
	\begin{lemma} 
		\label{Lemmap}
		The sequence $(p_k-2)\log(1/\delta_k)$ is bounded.
	\end{lemma}
	\begin{proof} By construction\footnote{ This is classical to detect bubbles, in the bubble-neck decomposition see section 3.1 of \cite{HS} or proof of theorem VII.3 of \cite{rivière2015variations} for the classical case, by defining the bubble's scale as the first scale whose energy is equal to the $\varepsilon_0$ given by $\varepsilon$-regularity} of $\delta_k$ we have
		$$\int_{\mathrm{B}(q_k,\delta_k)} \vert F_k\vert^2 \, dx=\frac{\varepsilon_0}{2}.$$
		Moreover, thanks to Hölder, we have
		$$\int_{\mathrm{B}(q_k,\delta_k)} \vert F_k\vert^2 \, dx \leq \left( \int_{\mathrm{B}(q_k,\delta_k)} \vert F_k\vert^{p_k} \, dx\right)^\frac{2}{p_k} \vert \mathrm{B}(q_k,\delta_k) \vert^\frac{p_k-2}{p_k}.$$
		Then using the fact that $\Vert F_k\Vert_{\mathrm{L}^{p_k}}$ is bounded, we can conclude.
	\end{proof}
	
	Recall that for all $k\in \N$ and $a\in \mathrm{W}^{1,2}(M,T^*M\otimes \mathfrak{g})$, \begin{equation}
		Q_{p_k,A_k} (a)= \int_{M}(1+|F_{A_k}|_h^2)^{\frac{p_k}{2}-1} \left( |\diff_{A_k} a|_h^2 + \langle F_{A_k}, [a, a] \rangle_h + (p_k-2) \frac{\langle F_{A_k}, \diff_{A_k} a \rangle_h^2}{1+|F_{A_k}|_h^2} \right)\mathrm{vol}_h.
	\end{equation} The Morse index of $\mathcal{YM}_{p_k}$ at $A_k$ is defined by: \[ \mathrm{ind}_{\mathcal{YM}_{p_k}}(A_k):=  \ind Q_{p_k,A_k} =  \sup \{ \dim W \mid Q_{\left. p_k, A_k\right| W}  < 0 \}. \]  A first stability result is the lower semi-continuity one:
	\begin{theorem} \label{lowerstability}
		For $k$ large enough, \[\mathrm{ind}_{\mathcal{YM}_{p_k}}(A_k) \geq \mathrm{ind}_\mathcal{YM}(A_{\infty})+ \mathrm{ind}_\mathcal{YM}(\widehat{A}_{\infty}).\]
	\end{theorem}
	
	The proof is almost the same as for $p_k=2$ from \cite[theorem 4.1]{GL24} and rely on the following lemma. However, special care is needed because nothing is conformally invariant anymore.

	\begin{lemma} Let $a \in \mathrm{W}^{1, 2} (M, T^{\ast} M \otimes
		\mathfrak{g})$ and assume there exists $\eta > 0$  such that $\supp a
		\subset M\backslash \mathrm{B}_{\eta} (p)$. Then, up to a subsequence, there exists a sequence $(a_k) \subset \mathrm{W}^{1, 2} (M, T^{\ast} M \otimes \mathfrak{g})$, such that satisfying the following properties :
		\begin{itemize}
			\item $\supp a_k \subset M\backslash \mathrm{B}_{\eta/2} (p)$,
			
			\item $a_k \overset{\mathrm{L}^{4, 2}}{\rightarrow} a$,
			
			\item $Q_{p_k,A_k} (a_k) \rightarrow Q_{2,A_{\infty}} (a)$.
		\end{itemize} \label{approximationlemma}
	\end{lemma}
	
	\begin{proof}[Proof of \cref{approximationlemma}] Identical to the one of \cite[lemma 4.2]{GL24} for $p_k=2$. 
	\end{proof}

	\begin{proof}[Proof of \cref{lowerstability}]
		
		Let $a \in \mathrm{W}^{1,2}(M, T^*M\otimes \mathfrak{g})$, $\widehat{a} \in \mathrm{W}^{1,2}(\Sph^4, T^*\Sph^4\otimes \mathfrak{g})$. Consider $\chi \in
		\mathcal{C}^\infty_c ([0, 2[, [0, 1])$ such that $\chi_{| [0, 1]} = 1$. Introduce for $\eta>0$,
		\[  a_\eta (x) = \left( 1 - \chi \left( \frac{| x - p |}{\eta} \right)  \right) a (x), \]
		and, in stereographic coordinates,
		\[\widehat{a}_{\eta}(x) =\chi \left( 2 \eta |x| \right)
		\widehat{a} \left( x \right).\] Apply \cref{approximationlemma} to $a_{\eta}$ and to $\widehat{a}_{\eta}$ (replacing $M$ by $S^4$) to get two sequences $(a_{\eta,k})_k$ and $(\widehat{a}_{\eta,k})_k$, such that 
		\begin{align*}
			Q_{p_k,A_k} (a_{\eta,k}) &\underset{k\to+\infty}{\rightarrow} Q_{2, A_\infty} (a_{\eta}),\\
			Q_{p_k , \phi_k^*A_k} ( \widehat{a}_{\eta,k}) &\underset{k\to+\infty}{\rightarrow} Q_{2,\widehat{A}_\infty} (\widehat{a}_{\eta}),
		\end{align*} and for all $k$ large enough $a_{\eta,k}$ and $(\phi_k)_*\widehat{a}_{\eta,k}$ have disjoint support. We have \[
		Q_{p_k,A_k} (a_{\eta,k}+ (\phi_k)_*\widehat{a}_{\eta,k}) = Q_{p_k,A_k} (a_{\eta,k})+ Q_{p_k,A_k}( (\phi_k)_*\widehat{a}_{\eta,k}). \]  By contrast with \cite[theorem 4.1]{GL24},  the problem is not conformally invariant but we still have $Q_{p_k,A_k}( (\phi_k)_*\widehat{a}_{\eta,k}) -  Q_{p_k , \phi_k^*A_k} ( \widehat{a}_{\eta,k}) \underset{k\to+\infty}{\rightarrow} 0$ since \begin{align*}Q_{p_k,A_k}( (\phi_k)_*\widehat{a}_{\eta,k}) =& \int_{\mathrm{B}_{1/\eta}}(1+|F_{\phi_k^*A_k}|_{\phi_k^*h}^2)^{\frac{p_k}{2}-1} \left( |\diff_{\phi_k^*A_k} \widehat{a}_{\eta,k}|_{\phi_k^*h}^2 + \langle F_{\phi_k^*A_k}, [\widehat{a}_{\eta,k}, \widehat{a}_{\eta,k}] \rangle_{\phi_k^*h}  \right)\mathrm{vol}_{\phi_k^*h}\\
			& + (p_k-2) \int_{\mathrm{B}_{1/\eta}} (1+|F_{\phi_k^*A_k}|_{\phi_k^*h}^2)^{\frac{p_k}{2}-2} \langle F_{\phi_k^*A_k}, \diff_{\phi_k^*A_k} \widehat{a}_{\eta,k} \rangle_{\phi_k^*h}^2  \mathrm{vol}_{\phi_k^*h} \end{align*} so \[ Q_{p_k,A_k} (a_{\eta,k}+ (\phi_k)_*\widehat{a}_{\eta,k})  \underset{k\to+\infty}{\rightarrow}  Q_{2,A_\infty} (a_{\eta})+ Q_{2,\widehat{A}_\infty} (\widehat{a}_{\eta}).\]
		From the identity \[\diff_{A_\infty} a_{\eta} = -\diff \left(\chi\left(\frac{|\cdot - p|}{\eta}\right)\right) \wedge a + \left(1-\chi\left(\frac{|\cdot - p|}{\eta}\right)\right)\diff_{A_\infty}a \]  and using the fact that $t\mapsto t \chi'(t)$ is bounded, we obtain 
		\[|\diff_{A_\infty} a_{\eta}| \leq C\frac{|a|}{|x-p|} \mathbf{1}_{\mathrm{B}(p,2\eta)}+ |\diff_{A_\infty} a|.\] Using the embeddings $\mathrm{L}^{4,\infty} \cdot \mathrm{L}^{4,2} \hookrightarrow \mathrm{L}^2$ and $\mathrm{W}^{1,2}\hookrightarrow \mathrm{L}^{4,2}$, we deduce that the right-hand side is in $\mathrm{L}^2$ and by the dominated convergence theorem $\diff_{A_\infty} a_{\eta}\underset{\eta \to 0}{\rightarrow} \diff_{A_\infty} a$ in $\mathrm{L}^2$. It is straightforward that $a_{\eta}\underset{\eta \to 0}{\rightarrow} a$ in $\mathrm{L}^4$ so we deduce \[\lim_{\eta\to 0} Q_{2,A_\infty}(a_\eta)=Q_{2,A_\infty}(a). \] A similar argument show that \[\lim_{\eta\to 0} Q_{2,\widehat{A}_\infty}(\widehat{a}_\eta)=Q_{2,\widehat{A}_\infty}(\widehat{a})\] and therefore \[\lim_{\eta \to 0} \lim_{k\to+\infty} Q_{p_k,A_k} (a_{\eta,k}+ (\phi_k)_*\widehat{a}_{\eta,k}) = Q_{2,A_\infty} (a)+ Q_{2,\widehat{A}_\infty} (\widehat{a}).\] Now, if $Q_{2,{A}_{\infty}} (a) < 0$ and
		$Q_{2,\widehat{A}_{\infty}} (\widehat{a}) < 0$, for $\eta$ small enough and $k$ large enough, \[Q_{p_k,A_k}(a_{\eta,k}+(\phi_k)_* \widehat{a}_{\eta,k})<0.\] 
		Let $W\subset \mathrm{W}^{1,2}(M, T^*M\otimes \mathfrak{g})$, $\widehat{W} \subset \mathrm{W}^{1,2}(\Sph^4, T^*\Sph^4\otimes \mathfrak{g})$ of finite dimension such that $Q_{2,A_\infty |W}<0$ and $Q_{2,\widehat{A}_\infty |\widehat{W}}<0$. Consider $(\phi^1,\dots,\phi^n)$ (resp. $(\psi^1,\dots,\psi^m)$) a basis of $W$ (resp. of $\widehat{W}$), orthonormal for $Q_{2,A_\infty}$ (resp. for $Q_{2,\widehat{A}_\infty }$). For $\eta>0$ and $k\in \N$, denote $\mathcal{B}_{\eta,k}$ the family whose elements are the $\phi^i_{\eta,k}$ and $\widehat{\psi^j}_{\eta,k}$ as defined above. Consider $M_{\eta,k}$ the Gram matrix for the quadratic form $Q_{p_k,A_k}$ of this family. From the earlier discussion, $\displaystyle \lim_{\eta \to 0} \lim_{k\to+\infty} M_{\eta,k}= -I$ and an implicit function argument ensures the existence, for $k$ large enough and $\eta$ small enough, of invertible matrices $(P_{\eta,k})$ such that  $P_{\eta,k} ^\top M_{\eta,k}P_{\eta,k} = - I$.  In particular $\Vect{\mathcal{B}_{\eta,k}}$ has dimension $n+m$ and  $Q_{p_k,A_k}$ is negative definite on $\Vect{ \mathcal{B}_{\eta,k}}$. We have then proved \[\mathrm{ind}_{\mathcal{YM}_{p_k}}(A_k) \geq \mathrm{ind}_\mathcal{YM}(A_{\infty})+ \mathrm{ind}_\mathcal{YM}(\widehat{A}_{\infty}).\] \end{proof}
	
	The upper-semi-continuity result is more subtle.

	\begin{theorem}\label{semicontinuitesupindex}
		For $k$ large enough  \[\mathrm{ind}^0_{\mathcal{YM}_{p_k}}(A_k) \leq \mathrm{ind}^0_{\mathcal{YM}}(A_{\infty})+ \mathrm{ind}^0_{\mathcal{YM}}(\widehat{A}_{\infty}),\] where $\mathrm{ind}^0_{\mathcal{YM}_{p}}(A) :=\ind^0 {\mathfrak{Q}_{p,A}}$.
	\end{theorem} 
	According to \cref{finitudeindice}, the study of 
	$$\mathrm{ind}^0_{\mathcal{YM}_{p}}(A) = \ind {Q_{p,A}} + \dim \left( \ker {Q_{p,A}}\cap {\ker \diff_{A}^*}\right)$$ 
	can be done using the quadratic form $\mathfrak{Q}_{p,A}$, which has better properties. When studying upper-semi continuity of the index, it is even more convenient to work with the quadratic form $\mathcal{Q}_{p,A}$ defined as \[\mathcal{Q}_{p,A}(a) =  \int_{M} |\diff_{A} a|_h^2+|\diff^*_{A} a|_h^2 + (1+|F_A|_h^2)^{\frac{p}{2}-1} \langle F_{A}, [a, a] \rangle_h \mathrm{vol}_h. \] This expression is indeed simpler but encompass enough meaningful information: since $\mathcal{Q}_{p,A}\leq \mathfrak{Q}_{p,A}$, we have $\ind^0 {\mathfrak{Q}_{p,A}} \leq \ind^0 {\mathcal{Q}_{p,A}}$. Additionally, when $p=2$, $\mathcal{Q}_{2,A}=\mathfrak{Q}_{2,A}$. The key point of the proof is the representation of the quadratic forms  $\mathcal{Q}_{A_k}$ as $a\mapsto  \left( a, \mathcal{L}_k a \right)_k$ where $\mathcal{L}_k$ is a self-adjoint operator with respect to an inner product $\left(\cdot, \cdot \right)_k$. Due to general index consideration related to Sylvester's law of inertia (\cref{diagonalisationofquadraticforms}), we can choose the inner product freely : the index and nullity of $\mathcal{Q}_{A_k}$ are independent of such a choice but can be easily determined as dimensions of eigenspaces of $\mathcal{L}_k$. The issue is therefore the choice of a suitable inner product, having in mind the estimates of \eqref{pointcrucial1}.

	For $k\in\N$ and $\eta>0$ small enough, define, in geodesic coordinates around $q$,  the weight function $\omega_{\eta, k}$ by : \[ \omega_{\eta, k} (q + x)= \left\{\begin{array}{ll}
		\displaystyle \frac{1}{\eta^2} \left( 1 + \left(\frac{\delta_k}{\eta^2}\right)^2 \right) & \mathrm{if }\, | x | \geq \eta,\\ \displaystyle \frac{1}{|x|^2} \left( \left(\frac{|x|}{\eta}\right)^2 + \left(\frac{\delta_k}{\eta|x|}\right)^2 \right) & \mathrm{if }\, \delta_k / \eta \leq |
		x | \leq \eta,\\
		\displaystyle \frac{\eta^2}{\delta_k^2}\left( \left(\frac{\delta_k}{\eta^2}\right)^2 +\frac{(1+(1/\eta)^2)^2}{(1+|x|^2/\delta_k^2)^2} \right) & \mathrm{if } \, |
		x | \leq \delta_k / \eta.
	\end{array}\right. \] 
	
	\begin{remark}
		We could have simply extended $\omega_{\eta,k}$ by a constant (depending on $\eta$ and $k$) inside $\mathrm{B}_{\delta_k/\eta}$. The reason for the choice is because, after pull-back by the stereographic projection, it will become simply a constant.
	\end{remark}
	
	We naturally introduce the inner product $\langle \cdot ,\cdot \rangle _{\omega_{\eta,   k}}$ by \[\| a\|_{\omega_{\eta,
			k}}^2 = \int_{M} |a|_h^2 \omega_{\eta,k} \mathrm{vol}_h\] which defines an operator $\mathcal{L}_{\eta, k}$ by \[ \mathcal{Q}_{A_k} (a) = \langle a, \mathcal{L}_{\eta, k} a \rangle _{\omega_{\eta,   k}} \] explicitly given by \[\mathcal{L}_{\eta, k} a = \omega_{\eta, k}^{-1} \mathcal{L}_{A_k}a = \omega_{\eta, k}^{-1} \left( \Delta_{A_k}  a +
	(1+|F_{A_k}|_h^2)^{\frac{p_k-2}{2}}\star[\star F_{A_k}, a]\right).\]
	
	Let's now introduce the counterparts of $\omega_{\eta,k}$ and $\mathcal{L}_{\eta,k}$ for the principal limit $A_{\infty}$ and the bubble $\widehat{A}_{\infty}$: define, for $x\in M\backslash \{q\}$, \[\omega_{\eta,\infty}(x)= \lim_{k\to +\infty} {\omega_{\eta,k}} (x)= 1/\eta^2 \] with uniform convergence in $M\backslash\mathrm{B}_\eta(q)$, and, if $\pi : S^4\to \R^4$ is the stereographic projection 
	\[\widehat\omega_{\eta,\infty}(x)=(1+|\pi(x)|^2)^2 \lim_{k\to +\infty} \delta_k ^2 \phi_k^\ast  \left(\omega_{k,\eta}\right)(\pi(x)) = \left\{\begin{array}{ll} \displaystyle \frac{1}{\eta^2}(1+|\pi(x)|^{-2})^2
		& \mathrm{if } |
		\pi(x) | \geq 1/\eta\\
		\displaystyle \frac{(1 +
			\eta^2)^2}{\eta^2}  & \mathrm{if } | \pi(x)| \leq 1/ \eta
	\end{array}\right.\] with uniform convergence in $\pi^{-1}(\mathrm{B}_{1/\eta})$. These correspond to the limits of the weight function $\omega_{\eta,k}$ outside and inside the bubble.

	The estimates of \cref{neckC0estimates} as well as the $\mathrm{L}^2$-quantization, see \cref{L2quantization}, can be used to show the following global estimates:
	\begin{proposition}
		\label{pointcrucial2} There exist $\eta_0, \mu_0>0$ such that for all $\eta<\eta_0$, the following pointwise estimate holds in $M$ for all $k$:
		\begin{equation}
			(1+|F_{A_k}|_h^2)^{\frac{p_k-2}{2}}| F_{A_k} |_h
			\leq \mu_0 \omega_{\eta, k}, \label{gradientpoids}
		\end{equation}
		and
		\begin{equation}
			\mathrm{Sp} (\mathcal{L}_{\eta, k}) \subset [- \mu_0 ; + \infty [.
			\label{spectreminoré}
		\end{equation}
	\end{proposition}
	\begin{proof} It suffices to apply \cref{CoroF} with $r={\delta_k}/{\eta}$ and $R=\eta$, and to remark that thanks to \cref{Lemmap}, $$1+\left(\frac{\eta^2}{\delta_k}\right)^{C(p_k-2)}$$ is bounded by some $\mu_0$.
	
	\end{proof}
	
	Define now the operators $\mathcal{L}_{\eta, \infty}$ and $\widehat{\mathcal{L}}_{\eta, \infty}$ by \begin{align*}
		\mathcal{L}_{\eta, \infty} a &= \omega_{\eta, \infty}^{-1} \left( \Delta_{A_\infty}  a +
		\star [\star F_{A_\infty}, a]\right),\\
		\widehat{\mathcal{L}}_{\eta, \infty} \widehat{a} &= \widehat{\omega}_{\eta, \infty}^{-1} \left( \Delta_{\widehat{A}_\infty}  \widehat{a} +
		\star[ \star F_{\widehat{A}_\infty}, \widehat{a}]\right),
	\end{align*} such that \begin{align*}
		\mathcal{Q}_{2,A_{\infty}} (a) &= \langle a, \mathcal{L}_{\eta, \infty} a \rangle_{\omega_{\eta,
				\infty}},\\
		\mathcal{Q}_{2,\widehat{A}_{\infty}} (\widehat{a}) &= \langle \widehat{a}, \widehat{\mathcal{L}}_{\eta, \infty} \widehat{a} \rangle_{\widehat{\omega}_{\eta,
				\infty}}
	\end{align*} for $a\in \mathrm{W}^{1,2}(M, T^*M\otimes \mathfrak{g})$ and $\widehat{a} \in \mathrm{W}^{1,2}(\mathrm{S}^4, T^*\mathrm{S}^4\otimes \mathfrak{g})$. The operators $\mathcal{L}_{\eta,k}$ (resp. $\mathcal{L}_{\eta,\infty}$, $\widehat{\mathcal{L}}_{\eta, \infty}$) can be diagonalized with respect to the inner products involving the weights $\omega_{\eta,k}$ (resp. $\omega_{\eta,\infty}$, $\widehat{\omega}_{\eta,
		\infty}$) according to \cref{diagonalisationofquadraticforms}. This guarantees, as stated before, that studying the index of $A_k$ can done by analysing the spectrum of $\mathcal{L}_{\eta, k}$. As for \cite[proposition 4.5]{GL24}, the proof of \cref{semicontinuitesupindex} is a direct consequence of the following result. \begin{lemma} \label{prelim1}
		If $\eta>0$ is small enough, for all $(a_k)_{k \in \N}$ such that for all
		$k \in \N$, $a_k \in W_{\eta, k}$ and $\| a_k \|_{\omega_{\eta, k}} =
		1$, there exists $a_{\infty} \in \mathrm{W}^{1, 2} (M, T^*M \otimes \mathfrak{g})$ and $\widehat{a}_{\infty}\in \mathrm{W}^{1,2}(\R^4, T^*\R^4\otimes\mathfrak{g})$ such that, up to extraction, $a_k \rightharpoonup
		a_{\infty}$ and $\widehat{a}_k \rightharpoonup
		\widehat{a}_{\infty}$ in $\mathrm{W}^{1, 2}$, where $\widehat{a}_k := \phi_k^*a_k$, and for $\delta > 0$ small enough : \begin{equation}
			\int_{M\backslash \mathrm{B}_{\delta} (p)} (| \diff_{A_{\infty}}
			a_{\infty} |^2_h + | \diff_{A_{\infty}}^{\ast} a_{\infty} |^2_h + |
			a_{\infty} |^2_h) \mathrm{vol}_h = \lim_{k \rightarrow + \infty}
			\int_{M\backslash \mathrm{B}_{\delta} (p)} \left( \left|
			\diff_{A_{_k}} a_k \right|^2_h + | \diff_{A_k}^{\ast} a_k |^2_h + |
			a_k |^2_h \right) \mathrm{vol}_h
		\end{equation} and \begin{equation}
			\int_{\mathrm{B}_{1 / \delta}} (| \diff_{\widehat{A}_{\infty}}
			\widehat{a}_{\infty} |^2 + | \diff_{\widehat{A}_{\infty}}^{\ast}
			\widehat{a}_{\infty} |^2 + | a_{\infty} |^2) \diff x  = \lim_{k
				\rightarrow + \infty} \int_{\mathrm{B}_{1 / \delta}} \left( \left|
			\diff_{\widehat{A}_{_k}} \widehat{a}_k \right|^2 + | \diff_{\widehat{A}_k}^{\ast}
			\widehat{a}_k |^2 + | \widehat{a}_k |^2 \right) \diff x.
		\end{equation} Moreover $(a_{\infty},
		\widehat{a}_{\infty}) \neq (0, 0)$.
	\end{lemma}
	
	\begin{proof} This result is identical to the one of \cite[lemma 4.6]{GL24} for $p_k=2$ : it enough to replace in the proof the occurrences of $F_{A_k}$ by $(1+|F_{A_k}|^2)^{p_k/2-1}F_{A_k}$.

	\end{proof}

			\newpage

			\appendix
			\section{Appendix}
			
			\subsection{Useful lemmas}

			\begin{lemma} 
				\label{leminq}	
				Let $\alpha\in [0;1)$, we have, for all $u,v \in \R$,
				
				$$\frac{1}{2}\vert u-v\vert^{2\alpha+2} \leq \langle (1+\vert u\vert^2)^\alpha u-(1+\vert v\vert^2)^\alpha v,u-v\rangle.$$
				In particular, if $2\leq p< 3$ we have
				$$\frac{1}{2}\vert u-v\vert^p \leq \langle (1+\vert u\vert^2)^\frac{p-2}{2}u-(1+\vert v\vert^2)^\frac{p-2}{2}v,u-v\rangle$$
				and
				$$\frac{1}{2}\vert u-v\vert^{\frac{p}{2}+1} \leq \langle (1+\vert u\vert^2)^\frac{p-2}{4}u-(1+\vert v\vert^2)^\frac{p-2}{4}v,u-v\rangle.$$
			\end{lemma}
			
			\begin{proof}
				\begin{eqnarray*}
					\langle (1+\vert u\vert^2)^\alpha u-(1+\vert v\vert^2)^\alpha v,u-v\rangle&=& (1+\vert u\vert^2)^\alpha \vert u\vert^2 +(1+\vert v\vert^2)^\alpha\vert v\vert^2\\ 
					&&-\left( (1+\vert u\vert^2)^\alpha  +(1+\vert v\vert^2)^\alpha \right)\langle u,v\rangle \\
					&=&(1+\vert u\vert^2)^\alpha \vert u\vert^2 +(1+\vert v\vert^2)^\alpha\vert v\vert^2\\ 
					&&+\left( (1+\vert u\vert^2)^\alpha  +(1+\vert v\vert^2)^\alpha \right)\frac{\vert v-u\vert^2-\vert u\vert^2-\vert v\vert^2}{2} \\
					&=&\left((1+\vert u\vert^2)^\alpha -(1+\vert v\vert^2)^\alpha \right)\left( \frac{\vert u\vert^2-\vert v\vert^2}{2}\right)\\ 
					&&+\left( (1+\vert u\vert^2)^\alpha  +(1+\vert v\vert^2)^\alpha \right)\frac{\vert v-u\vert^2}{2} \\
					& \geq &\left( (1+\vert u\vert^2)^\alpha  +(1+\vert v\vert^2)^\alpha \right)\frac{\vert v-u\vert^2}{2},
				\end{eqnarray*}
				the inequality follows from the fact that the first term of previous line is non-negative. Finally,  we use the following  inequality, if $0\leq \beta\leq 1$, 
				\begin{equation}
					\label{eqab}
					(a+b)^{\beta}\leq  a^{\beta}  + b^{\beta} \text{ for all } a,b\geq 0,
				\end{equation}
				to deduce that 
				\begin{align*} \vert v-u\vert^{2\alpha+2} &= \vert v-u\vert^2 \vert v-u\vert^{2\alpha}\leq  \vert v-u\vert^2 (\vert v\vert +\vert u\vert)^{2\alpha}\leq   \vert v-u\vert^2 ((\vert v\vert^2)^\alpha +(\vert u\vert^2)^\alpha)\\
					&\leq   \vert v-u\vert^2 \left((1+\vert v\vert^2)^\alpha +(1+\vert u\vert^2)^\alpha\right),
				\end{align*}
				which achieves the proof.
			\end{proof}
			
			To prove \eqref{eqab}, it suffices to assume $a\geq b$ by symmetry and $b\not= 0$ then to prove that 
			$$ (k+1)^{\beta}\leq k^{\beta} +1 \text{ for all } k\geq 1,$$
			which can easily be done by function analysis.
			Notice that, if $\beta \geq 1$, we can use the following convexity inequality to obtain a similar result.
			$$ (a+b)^{\beta}\leq  2^{\beta-1 } (a^{\beta}  + b^{\beta}) \text{ for all } a,b\geq 0.$$

			Let $f(x)=x(1+x^2)^\alpha$ since $\vert f'(x)\vert=\vert (1+x^2)^\alpha +2\alpha(1+x^2)^{\alpha-1} x^2\vert \leq 2(1+x^2)^\alpha$ on $\R_+$ as soon as $\alpha\leq\frac{1}{2}$ then $f$ is Lipschitz. Let $2\leq p\leq 3$, since $0\leq \frac{p-2}{4}\leq \frac{1}{4}$, we have 
			\begin{equation}
				\label{Vin}
				\vert V(a)-V(b)\vert=\vert \sqrt{\rho(a)}a-\sqrt{\rho(b)}b\vert \leq 2 \max (\sqrt{\rho(a)},\sqrt{\rho(b)})\vert a-b\vert,
			\end{equation}
			where $V(x)=x(1+x^2)^{\frac{p-2}{4}}$.\\
			We similarly get, for $2\leq p\leq 3$ and $a,b\geq 0$, that 
			\begin{equation} 
				\label{Habestim}
				H(a-b)\leq \sqrt{2} (H(a) + b^p),
			\end{equation}
			where $H(x)=(1+x^2)^{\frac{p}{2}}$.\\
			\begin{lemma}[Böchner formula for general connections] There exists $C>0$, such that for any connection $A$ with $F$ as curvature must satisfy
				\begin{equation}
					\frac{1}{2} \Delta | F |^2 \leq |\diff_A^*F|^2 -\diff^* \omega - | \nabla_A F |^2 +C  | F |^3
					\label{Bochnergeneral}
				\end{equation} where $\omega$ is the $1$-form defined by duality via the equation \[\langle \omega, \alpha\rangle = \langle \alpha\wedge \diff_A^* F, F\rangle = - \langle \diff_A \star F,\alpha \wedge \star F \rangle \]
			\end{lemma}
			
			\begin{proof} We make use of the following Böchner identity : 	\[  \frac{1}{2} \Delta | F |^2 = \langle \nabla_A^{\ast} \nabla_A F, F \rangle
				- | \nabla_A F |^2\] Indeed, consider an arbitrary point $q$ and chose an orthonormal frame $(e_i)_i$ such that $\nabla e_i = 0$ at $q$, then  \begin{align*}
					\frac{1}{2} \Delta |F|^2 = -	\frac{1}{2} \Tr( \nabla^2 |F|^2) = -	\frac{1}{2} \sum_i \nabla^2_{e_i,e_i} |F|^2 
					&= - 	\frac{1}{2} \sum_i (\nabla_{e_i} \nabla|F|^2 )(e_i) \\
					&= - \sum_i (\nabla_{e_i} \langle F, \nabla_A F \rangle )(e_i)\\
					& = - \sum_i |(\nabla_A)_{e_i} F|^2 + \langle F, (\nabla_A^2)_{e_i,e_i} F  \rangle \\
					&= - |\nabla_A F|^2 - \langle F, \Tr(\nabla^2_A F)\rangle \\
					& = - |\nabla_A F|^2 + \langle F, \nabla_A^*\nabla_A F\rangle. 
				\end{align*} 
				From the Weitzenböck formula (see \cite[Theorem 3.10]{BourguignonLawson}), we get
				\[\nabla_A^{\ast} \nabla_A F = (\diff_A^*\diff_A + \diff_A \diff_A^*)F + Q(F)\] where $Q$ is quadratic. Hence
				\begin{equation}
					\frac{1}{2} \Delta | F |^2 \leq \langle \diff _A \diff _A^{\ast} F, F \rangle +C \vert F\vert^3  - | \nabla_A F |^2
				\end{equation}
				and for all test functions $\eta$,
				\begin{eqnarray*}
					\int	\eta \langle \diff _A \diff _A^{\ast} F, F \rangle 
					& = & \int \langle \diff_A (\eta \diff^*_A F), F\rangle - \int \langle \diff \eta \wedge \diff^*_A F, F\rangle\\
					&= & \int \langle \eta  \diff^*_A F,  \diff^*_A F \rangle
					- \int \langle \diff \eta, \omega \rangle \\
					&=&  \int \eta (|\diff^*_A F|^2- \diff^*\omega) \end{eqnarray*}
				which achieves the proof.
			\end{proof}

			\begin{lemma}[Kato-Yau inequality for $p$-Yang-Mills connections]
				Let $A$ be a smooth $p$-Yang-Mills connection with curvature $F$ and  $p\geq 2$. Then \[\mu(p)|\diff |F_A||^2  \leq |\nabla_A F_A|^2\] where \[\mu(p) = \max\left(2-p/2+\frac{1}{2(p-1)}, 1\right)\] \label{KatoYau}
			\end{lemma}
			
			\begin{proof} We write $\nabla = \nabla_A$ to simplify notations. Recall that for any $2$-form $\Omega$:\begin{eqnarray*}
					\diff_{A} \Omega (X_0, X_1, X_2) & = & \nabla_{X_0} (\Omega (X_1, X_2)) -
					\nabla_{X_1} (\Omega (X_0, X_2)) + \nabla_{X_2} (\Omega(X_0, X_1))\\
					&  & - \Omega ([X_0, X_1], X_2) + \Omega ([X_0, X_2], X_1) -\Omega ([X_1, X_2], X_0)
				\end{eqnarray*}
				Fix a point $q$ and consider coordinates $(x^i)_{1\leq i\leq4}$ around $q$ such that at $q$, $(\frac{\partial}{\partial x^i})_{1\leq i\leq4}$ is orthonormal and  \[\left|
				\frac{\partial}{\partial x^1} | F | \right| = | \diff | F | |.\] From the usual Kato inequality \[ | \diff | F | | \leq |\nabla_1 F|.\] For every $2$-form $\Omega$, $ \diff_{A} \Omega = 0$ if and only if for all $(i,j,k) \in
				\{ (1, 2, 3), (1, 2, 4), (1, 3, 4), (2, 3, 4) \}$ \begin{align*}
					0 &= \diff_{A} \Omega \left(
					\partial_{i}, \partial_{j},
					\partial_{k} \right) \\
					&= \nabla_{\partial_{i}} (\Omega( \partial_{j},
					\partial_{k})) + \nabla_{\partial_{j}} (\Omega( \partial_{k},
					\partial_{i})) + \nabla_{\partial_{k}} (\Omega( \partial_{i},
					\partial_{j}))\\
					&=\nabla_{i} \Omega_{j k} + \nabla_{j} \Omega_{k i} + \nabla_{k} \Omega_{i j}.
				\end{align*}   We write $F = \frac{1}{2} F_{i j} \diff x^i \wedge \diff x^j$. The Bianchi identity $\diff_A F=0$ reads
				\begin{equation} \left\{\begin{array}{l}
						\nabla_1 F_{23} + \nabla_2 F_{3 1} + \nabla_3 F_{1 2} = 0\\
						\nabla_1 F_{2 4} + \nabla_2 F_{4 1} + \nabla_4 F_{1 2} = 0\\
						\nabla_1 F_{3 4} + \nabla_3 F_{4 1} + \nabla_4 F_{1 3} = 0\\
						\nabla_2 F_{3 4} + \nabla_3 F_{4 2} + \nabla_4 F_{2 3} = 0
					\end{array}\right.
				\end{equation} The $p$-Yang-Mills equation $\diff_A^* \left( (1+|F|^2)^{\frac{p}{2}-1 } F\right)=0$ can be written as  $\diff_A \star F = (p - 2) G$ where 
				\[ | G |^2 = \left| \frac{| F | \diff | F | \wedge \star F}{1 + | F |^2}
				\right|^2 \leq | \diff | F | |^2. \] Noting that $(\star F)_{i j} = F_{k l}$ as long as $(k, l, i, j)$ is a cyclic permutation of $(1,2,3,4)$, we deduce similarly the following identities
				\begin{equation} 
					\left\{\begin{array}{l}
						\nabla_1 F_{41} + \nabla_2 F_{2 4} + \nabla_3 F_{34} = (p-2) G_{1
							23}\\
						\nabla_1 F_{3 1} + \nabla_2 F_{3 2} + \nabla_4 F_{34} =  (p-2) G_{1
							24}\\
						\nabla_1 F_{12} + \nabla_3 F_{3 2} + \nabla_4 F_{42} = (p-2) G_{13
							4}\\
						\nabla_2 F_{12} + \nabla_3 F_{1 3} + \nabla_4 F_{41} = (p-2)
						G_{234}
					\end{array}\right.
				\end{equation} 
				In any vector space $V$ with an inner product, if $a,b\in V$ and $\varepsilon\geq 0$, using Young inequality:
				\[ | a + \varepsilon b |^2 = | a |^2 + \varepsilon^2 | b |^2 + 2 \varepsilon
				\langle a, b \rangle \leq (1 + \varepsilon) | a |^2 + \varepsilon (1 +
				\varepsilon) | b |^2 \] At $q$, we deduce the following estimate
				\begin{align*}
					| \nabla_1 F |^2 = &\, | \nabla_1 F_{1 2} |^2 + | \nabla_1 F_{13} |^2 + |
					\nabla_1 F_{1 4} |^2 + | \nabla_1 F_{2 3} |^2 + | \nabla_1 F_{24} |^2 + |
					\nabla_1 F_{34} |^2\\
					= &\, | \nabla_3 F_{32} + \nabla_4 F_{42} - (p-2) G_{134} |^2 + |
					\nabla_2 F_{32} + \nabla_4 F_{34} - (p-2) G_{124} |^2 + | \nabla_2
					F_{31} + \nabla_3 F_{1 2} |^2\\
					& + | \nabla_2 F_{24} + \nabla_3 F_{34} - (p-2) G_{123} |^2 + |
					\nabla_2 F_{ 4 1} + \nabla_4 F_{1 2} |^2 + | \nabla_3 F_{41} + \nabla_4 F_{1
						3} |^2\\
					\leq &\, 2 (p-1) (| \nabla_3 F_{32} |^2 + | \nabla_4 F_{42} |^2 + |
					\nabla_2 F_{32} |^2 + | \nabla_4 F_{34} |^2 + | \nabla_2 F_{24} |^2 + |
					\nabla_3 F_{34} |^2)\\
					&  + 2 (| \nabla_2 F_{31} |^2 + | \nabla_3 F_{1 2} |^2 + | \nabla_2 F_{41} |^2 + | \nabla_4 F_{1 2} |^2 + | \nabla_3 F_{41} |^2 + | \nabla_4 F_{1 3}
					|^2)\\
					& + (p-1) (p-2) (| G_{134} |^2 + | G_{124} |^2 + | G_{123}
					|^2)
				\end{align*} Therefore \begin{align*}
					| \nabla_1 F |^2 \leq & \, 2(p-1)(|\nabla F|^2-|\nabla_1 F|^2)+(p-1)(p-2)|G|^2\\
					|\diff |F||^2  \leq & \, 2(p-1)|\nabla F|^2 -(4-p)(p-1)|\diff |F||^2 
				\end{align*}In conclusion \[(1 + (p-1) (4-p)) | \nabla_1 F |^2 \leq  2 (p-1)
				| \nabla F |^2.\] Since the inequality  $ | \nabla_1 F |^2 \leq  
				| \nabla F |^2$ is trivial, we have the desired result.
			\end{proof}
			
			\begin{remark}
				Notice that $\mu$ is a decreasing function of $p$ and that $\mu(2)=3/2$.
			\end{remark}

			\begin{lemma} \label{Aestimate} Let $\mathcal{A} : \Lambda^1(\R^4)\rightarrow \Lambda^1(\R^4)$ defined by the equation \[\langle \mathcal{A}(\alpha),\beta\rangle = \frac{\langle \alpha \wedge \star F, \beta \wedge \star F \rangle}{1+|F|^2} ,\] then 
				\[ 0 \leq \mathcal{A}^{\alpha \beta} \frac{x_{\alpha} x_{\beta}}{| x
					|^2} \leq \mathcal{A}^{\alpha \beta} \delta_{\alpha \beta}
				\leq 2 .\]
			\end{lemma}
			
			\begin{proof} We use the results of \cref{normofwedge}.
				
				\begin{align*}
					\mathcal{A}^{\alpha \beta} \frac{x_{\alpha} x_{\beta}}{| x |^2}  & =  \frac{1}{1 +
						| F |^2} \left| \sum_{\alpha} \frac{x^{\alpha}}{|x|} \diff x^{\alpha}
					\wedge \star F \right|^2\\
					& \leq  \frac{1}{1 + | F |^2} \sum_{\alpha}
					\frac{(x^{\alpha})^2}{|x|^2} \times \sum_{\alpha} | \diff x^{\alpha}
					\wedge \star F |^2\\
					& \leq \frac{\sum_{\alpha} | \diff x^{\alpha} \wedge \star F
						|^2}{1 + | F |^2} =\mathcal{A}^{\alpha \beta} \delta_{\alpha \beta}
				\end{align*}
				Furthermore, \[ \sum_{\alpha} | \diff x^{\alpha} \wedge \star F
				|^2 = 2|F|^2\] so \[\mathcal{A}^{\alpha \beta} \delta_{\alpha \beta} = \frac{2|F|^2}{1+|F|^2}\leq 2.\]
			\end{proof}
			
			\begin{corollary}[Böchner formula for $p$-Yang-Mills connections]
				\label{CBochnerpYM} There exists $C>0$, such that if $A$ is a $p$-Yang-Mills connection with curvature $F$, then $\phi = | F |^2$ satisfies
				\begin{equation}
					\mathcal{L}_p \phi \leq C  \phi ^{3 / 2}  - 2(1-(p-2)^2) | \nabla_A F |^2
					\label{BochnerpYM}
				\end{equation}
				where $\mathcal{L}_p \phi = \Delta \phi - (p - 2) \diff^*\left(\mathcal{A}(\diff \phi)\right) $ and $\mathcal{A} $ is defined in the previous lemma.
			\end{corollary}
			
			\begin{proof}
				The $p$-Yang-Mills equation writes
				\[ \diff _A \star F +  \frac{p-2}{2} \star \frac{\diff 
					\phi \wedge \star F}{1 + | F |^2} =0 \]
				so in this setting, \[ \langle\omega , \alpha \rangle =  \frac{p-2}{2} \frac{\langle\diff\phi \wedge \star F,\alpha \wedge \star  F  \rangle}{1+|F|^2}\] and, using \cref{KatoYau} and \cref{normofwedge} below, \[|\diff_A^* F|^2 \leq (p-2)^2 |\nabla_A F|^2\]
				then \eqref{Bochnergeneral} becomes
				\begin{eqnarray*}
					\frac{1}{2} \Delta \phi & \leq & \frac{1}{2} (p - 2) \diff^* \left(\mathcal{A}(\diff \phi) \right)  -(1-(p-2)^2) | \nabla _A F |^2 +C \phi^{3 / 2}.
				\end{eqnarray*}
				
			\end{proof}
			
			\begin{lemma} \label{normofwedge} Let $(M^n,h)$ a Riemannian manifold. If $X$ is a vector field (identified with the $1$-form $\langle X, \cdot \rangle$) and $\omega$ is a (vector-valued) differential $k$-form on $M$ then  
				\[ |\iota_X \omega|^2 + |X \wedge\omega|^2 = |X|^2|\omega|^2\] and in particular \[  |X \wedge\omega|^2 \leq |X|^2|\omega|^2,\] and the Hilbert-Schmidt norm of the map $X\mapsto X\wedge \omega$ is $(n-k)|\omega|^2$.
			\end{lemma}
			
			\begin{proof} Consider a (local) orthonormal frame $(e_1, \ldots, e_n)$ such that $X = |X|  e_1$, and denote by $(e^1,\dots, e^n)$ the associated coframe. When $I=\{i_1, \dots i_k\}\subset \{1,\dots, n\}$ with $i_1 <\dots <i_k$ we write $e^I = e^{i_1}\wedge \dots \wedge e^{i_k}$. Recall that $(e^I)_{|I|=k}$ is an orthonormal basis of $k$-forms. If $\omega$ is a $k$-form, write \[\omega = \sum_{|I|=k} \omega_I e^I.\]
				\begin{align*}
					| \iota_{e_1} \omega |^2 + | e^1 \wedge \omega |^2 & =  \left| \sum_{| I |
						= k} \omega_I \iota_{e_1} e^I \right|^2 + \left| \sum_{| I | = k} \omega_I
					e^1 \wedge e^I \right|^2\\
					& =  \left| \sum_{| I | = k, 1 \in I} \omega_I \iota_{e_1} e^I \right|^2 +
					\left| \sum_{| I | = k} \omega_I e^1 \wedge e^I \right|^2\\
					& =  \left| \sum_{| I | = k, 1 \in I} \omega_I \iota_{e_1} (e^1 \wedge
					e^{I\backslash \{ 1 \}}) \right|^2 + \left| \sum_{| I | = k, 1 \not\in I}
					\omega_I e^1 \wedge e^I \right|^2\\
					& =  \left| \sum_{| I | = k, 1 \in I} \omega_I e^{I\backslash \{ 1 \}}
					\right|^2 + \left| \sum_{| I | = k, 1 \not\in I} \omega_I e^{\{ 1 \} \cup I}
					\right|^2\\
					& =  \sum_{| I | = k, 1 \in I} | \omega_I |^2 + \sum_{| I | = k, 1 \not\in I}
					| \omega_I |^2\\
					& =  \sum_{| I | = k} | \omega_I |^2\\
					& =  | \omega |^2
				\end{align*}
				By homogeneity, we obtain :
				\[ | \iota_X \omega |^2 + | X \wedge \omega |^2 = | \omega |^2 | X |^2 \]
				and in particular
				\[ | X \wedge \omega |^2 \leq | \omega |^2 | X |^2.\] Moreover,
				\[ \sum_{i=1}^n | e^i \wedge \omega |^2 =  \sum_{\substack{i, I \\ | I | = k, i \not\in I}} |
				\omega_I |^2 = \sum_{| I | = k} (n - k) | \omega_I |^2 = (n - k) | \omega
				|^2.\]
			\end{proof}

			\begin{proposition} \label{Bochnerpower}
				There exists a constant $C > 0$ such that  if $A$ is a $p$-Yang-Mills connection with curvature $F$, then for all $\beta \geq 0$ and $p\geq 2$, then $\phi = | F |^2$ satisfies
				\[ \mathcal{L}_p \phi^{\beta} \leq C \beta  \phi^{\beta+1/2} +
				2 \beta \phi^{\beta-1} (\kappa (p, \beta) | \diff  | F |
				|^2 - (1-(p-2)^2) | \nabla_A F |^2), \]
				where $\kappa (p, \beta) = 2 (p-1) (1 - \beta)_+$.
			\end{proposition}
			
			\begin{proof} We have
				\begin{eqnarray*}
					\mathcal{L}_p \phi^{\beta} & = & \beta \phi^{\beta - 1} \mathcal{L}_p \phi +
					\beta (1 - \beta) \phi^{\beta - 2} \left(| \diff  \phi |^2 + (p - 2)  \frac{\vert \diff  \phi
						\wedge \star F|^2}{1+\vert F\vert^2} \right).
				\end{eqnarray*} Using \cref{CBochnerpYM} and \cref{normofwedge} , we have
				\begin{eqnarray*}
					\mathcal{L}_p \phi^{\beta} & \leq &\beta \phi^{\beta-2}(\phi \mathcal{L}_p \phi  + (1-\beta)_+ (p-1)|\diff \phi|^2)\\
					& \leq &\beta \phi^{\beta-2}(C\phi^{\frac{5}{2}}-2(1-(p-2)^2) \phi | \nabla_A F |^2 +4(1-\beta)_+ (p-1)\phi   | \diff 
					| F | |^2)\\
					& \leq &  C \beta  \phi^{\beta+1/2} +
					2 \beta \phi^{\beta-1} (2 (p - 1) (1 - \beta)_+  | \diff 
					| F | |^2-  (1-(p-2)^2)| \nabla_A F |^2).
				\end{eqnarray*}
			\end{proof}
			
			\begin{remark}
				Notice that $\beta \mapsto \kappa (p, \beta)$ is non-increasing and
				that for $\beta \geq 1$, $\kappa (p, \beta) = 0$.
			\end{remark}

			\begin{lemma}[Hardy inequality] \label{Hardyi}
				If $f\in \mathcal{C}^\infty_c(\R^4)$, \begin{equation}
					\label{Hardy} \int_{\R^4}  \frac{|f|^2}{|x|^2} \diff x \leq \int_{\R^4} |\nabla f |^2 \diff x.
				\end{equation} 
				Which is also true for $f\in \dot{W}^{1,2}(\R^4)$ by density.
			\end{lemma} 
			\begin{proof}
				see for instance \cite[Corollary 1.2.6]{HardyIneq}.
			\end{proof}
			
			\begin{lemma}\label{gaffney} Let $n\in \N^*$.
				For all $p\in ]1;+\infty[$, $q\in ]0;+\infty] $, there exists a constant $C_{p,q}>0$ such that for all differential forms $\omega$ on $\R^n$, \[\|\nabla\omega\|_{\mathrm{L}^{p,q}(\R^n)}\leq C_{p,q} \left(\|\diff \omega\|_{\mathrm{L}^{p,q}(\R^n)}+ \|\diff^{\ast} \omega\|_{\mathrm{L}^{p,q}(\R^n)} \right).\]
			\end{lemma}
			
			\begin{proof}
				See for instance \cite{troyanov2010hodge} for the $L^p$ version, then it suffices to apply Marcinkiewicz interpolation theorem, see \cite[Theorem 3.3.3]{Helein}.
			\end{proof}

			\subsection{Quadratic forms, index and diagonalization}
			In this section $(E, \langle \cdot, \cdot \rangle_E)$ is a real Hilbert space and $q$ is a continuous quadratic
			form on $E$.
			
			\begin{theorem}
				The following equalities hold
				\begin{eqnarray*}
					\sup \{ \dim F|F \subset E, q_{|F} < 0 \} & = & \sup \{ \dim F|F \subset
					E, q_{|F} \leq 0, F \cap \ker q = \{ 0 \} \}\\
					& = & \inf \left\{ \dim F|F \subset E, q_{|F^{\bot_q}} \geq 0 \right\}\\
					& =& \inf \left\{ \dim F|F \subset E, q_{|F^{\bot}} \geq 0 \right\}
				\end{eqnarray*}
				and we call this quantity the index of $q$, written as $\ind q$.
			\end{theorem}
			
			\begin{definition} The dimension of the kernel of the kernel of the quadratic form $q$ (\textit{i.e.} of the space of vectors of $E$ $q$-orthogonal to $E$) is called the nullity of $q$ and written as $\nul q$.
				
			\end{definition}
			
			\begin{definition}
				A quadratic form $q$ is called a Legendre form if its satisfies the
				following conditions:
				\begin{enumerate}
					\item $q$ is weakly lower semi-continuous,
					
					\item if $x_n \rightharpoonup x$ (weakly) and $q (x_n) \rightarrow q (x)$
					then $x_n \rightarrow x$ (strongly).
				\end{enumerate}
			\end{definition}
			
			\begin{proposition}
				If $q$ is Legendre, then $q$ has finite index and nullity.
			\end{proposition}

			\begin{theorem} \label{diagonalisationofquadraticforms}
				Assume there is a dense embedding of Hilbert spaces \[(E, \langle \cdot, \cdot \rangle_E) \hookrightarrow
				(\mathfrak{H}, \langle \cdot, \cdot \rangle_{\mathfrak{H}})\] and that $q$ is a Legendre form on $E$. Then there exists a unique operator
				$(\mathcal{L}, D (\mathcal{L}))$, self-adjoint for $\langle
				\cdot, \cdot \rangle_{\mathfrak{H}}$, with domain $D (\mathcal{L})
				\subset (E, \langle \cdot, \cdot \rangle_E)$ dense such that \[\forall u \in D
				(\mathcal{L}), q (u) = \langle \mathcal{L}u, u \rangle_{\mathfrak{H}}.\]
				Moreover
				\begin{align*}
					\ind q  &= \dim \bigoplus_{\lambda < 0} \ker (\mathcal{L}-
					\lambda),\\
					\nul q  &= \dim \ker \mathcal{L}.
				\end{align*} In particular these dimensions are finite and do not depend on $	(\mathfrak{H}, \langle \cdot, \cdot \rangle_{\mathfrak{H}})$.
			\end{theorem}
			
			\subsection{Regularity and $\varepsilon$-Regularity of $p$-Yang-Mills connections}
			\label{epsreg}

			Let $A \in \mathrm{W}^{1, p} \left( \mathrm{B}^4, \Lambda^1 \R^4
			\otimes \mathfrak{g} \right)$, with $2\leq p <3$, and $F$ its curvature. We want to study the regularity of $A$ if it satisfies the following equation in the weak sense :
			\begin{equation}
				\diff^{\ast}_A ((1 + | F |^2)^{\frac{p}{2} - 1} F) = 0 .\label{pYM}
			\end{equation} The abelian counterpart of this equation, namely \begin{equation}
				\diff^{\ast} ((1 + | \diff \omega |^2)^{\frac{p}{2} - 1}\diff \omega) = 0,\label{abelianpYM}
			\end{equation} falls within the scope of results by Uhlenbeck \cite{uhlenbeck_regularity_1977} and Hamburger \cite{hamburger_regularity_1992}, whereas its degenerate version \begin{equation}
				\diff^{\ast}_A ( | F |^{p-2} F) = 0,\label{degeneratepYM}
			\end{equation} had been studied by Isobe \cite{isobe_regularity_2008}, relying on the abelian case. In \cite{HS}, Hong and Schabrun gave a positive answer to the question of regularity. However, we wish to provide the reader with a more precise statement, with estimates independent of $p$.
			
			\noindent In the following, to stay closer to the notations of these references (e.g. \cite[eq. (2.8) \& (2.9)]{hamburger_regularity_1992}), we will denote, for $\Omega \in \Lambda^2
			\R^4 \otimes \mathfrak{g}$,
			\begin{align*}
				H (\Omega)&= (1 + | \Omega |^2)^{p / 2},\\
				\varrho (\Omega) &= (1 + | \Omega |^2)^{(p - 2)
					/ 2},\\
				V (\Omega) &= \sqrt{\varrho (\Omega)} \Omega
			\end{align*}
			Therefore, \eqref{pYM} becomes :
			\begin{equation}
				\label{main}
				\diff_A^{\ast} (\varrho (F) F) = 0.
			\end{equation}
			
			\begin{theorem}
				\label{Ereg} There exists $\varepsilon_G >0$, $p_G\in(2,3)$ and for every $\ell\in \N$,  there exists $C_{G,\ell}>0$ such that for all $R\in(0,1)$, $A\in \mathrm{W}^{1,p} \left(
				\mathrm{B}_R, \Lambda^1 \R^4 \otimes \mathfrak{g} \right)$ satisfying  \eqref{main} with $2\leq p\leq p_G $, and 
				$$ \int_{\mathrm{B}_R} \vert F_A\vert^2\, \diff x \leq \varepsilon_G,  $$
				there exists $g\in \mathrm{W}^{2,p}(\mathrm{B}_R,G)$ such that $A^g$ is smooth and
				$$ R^{\ell+1} \Vert \nabla^\ell A^g\Vert_{\mathrm{L}^\infty(\mathrm{B}_{R/2})} \leq C_{G, \ell} \sqrt{\int_{\mathrm{B}_R} |F_A|^2\diff x}.$$
			\end{theorem}
			
			In order to prove this theorem, we will use the following version of the gauge extraction of Uhlenbeck.

			\begin{theorem}[theorem 1.3 \cite{UhlenbeckKarenK1982CwLb} and theorem IV.1 \cite{rivière2015variations} ]
				\label{uhlenbeckgauge}
				Let $2\leq p<3$, there exists $\varepsilon_G >0$ and $C_G>0$ such that for all $A\in  \mathrm{W}^{1,p} \left(
				\mathrm{B}_R, \Lambda^1 \R^4 \otimes \mathfrak{g} \right)$ satisfying
				$$ \int_{\mathrm{B}_R} \vert F_A\vert^2\, \diff x \leq \varepsilon_G, $$
				there exists $g\in \mathrm{W}^{2,p}(\mathrm{B}_R,G)$ such that \begin{align*}
					\Vert \diff A^g\Vert_{\mathrm{L}^p} &\leq C_G \Vert F_A\Vert_{\mathrm{L}^p},\\
					\Vert \diff A^g\Vert_{\mathrm{L}^2} &\leq C_G \Vert F_A\Vert_{\mathrm{L}^2},\\
					\diff^* A^g&=0,\\
					i^*_{\partial \mathrm{B}_R} \star A^g&=0.
				\end{align*}
			\end{theorem}
			In fact in the reference given, the theorem is stated for $p=2$, but the proof is made for $2\leq p <4$, so the result holds as stated here, in particular $\varepsilon_G$ and $C_G$ can be chosen independently of $p$.\\
			
			The proof of \cref{Ereg} is divide in  parts, each of the following section is devoted to each of the following steps:
			\begin{enumerate}
				\item We have Morrey bounds on $F$ that imply Hölder-continuity for $A$.
				
				\item We have Campanato bounds for $V (F^g)$ for a suitable gauge change $g$ from which we can deduce that $F$
				is Hölder-continuous and so is $\nabla A$.
				
				\item We the prove that  $\nabla F \in \mathrm{L}^2$ by difference-quotient method and deduce
				that $F$ satisfies 
			
				\[ \diff^{\ast} F = \star [A, \star F] + \frac{p - 2}{2} \star \frac{\diff | F     |^2 \wedge \star F}{1 + | F |^2}. \]
			
				\item We conclude that $F$ and $A$ are smooth.
				
				\item We show that $|F|^2$ is a subsolution to an elliptic PDE to deduce pointwise estimates for the curvature.
				
				\item We bootstrap these estimates in the equation, seen as a perturbation of the case $p=2$.
			\end{enumerate}

			\subsubsection{Preliminary estimates}
			When $p=2$, the energy is invariant under scaling so the usual strategy is to derive estimates on a ball of radius $1$ and then use a dilation to extend then to balls of arbitrary radii. When $p>2$, this is not the case anymore so we choose to work on a ball of (fixed) radius $R$. We assume  that
			\begin{equation}
				\label{epsi}
				\int_{\mathrm{B}_R} \vert F_A\vert^2\, \diff x \leq \varepsilon_0 \leq \varepsilon_G, 
			\end{equation}
			and we consider $A \in  \mathrm{W}^{1,p} \left(
			\mathrm{B}_R, \Lambda^1 \R^4 \otimes \mathfrak{g} \right)$ being given in Coulomb gauge  given by \cref{uhlenbeckgauge}. As observed in \cite{HS}, we can adapt the proof of \cite[theorem 1.1]{isobe_regularity_2008} for the degenerate equation in our simpler setting.
			
			\paragraph{Estimates in the {\guillemotleft} abelian case {\guillemotright}} 
			\label{abelian}

			Let $\rho \in (0,R)$ and $a \in \mathrm{W}^{1, p} \left(
			\mathrm{B}_R, \Lambda^1 \R^4 \otimes \mathfrak{g} \right)$. We consider $E (\omega) = \int_{\mathrm{B}_\rho} H (\diff \omega)$. 
			Minimizing $E$
			under the constraint $i^{\ast}_{\partial \mathrm{B}_\rho} \omega =
			i^{\ast}_{\partial \mathrm{B}_\rho} a$, we get $\omega \in \mathrm{W}^{1, p} \left(
			\mathrm{B}_\rho, \Lambda^1 \R^4 \otimes \mathfrak{g} \right)$ solution of
			\begin{equation} 	\label{pYMab}
				\left\{\begin{array}{l}
					\diff^{\ast} (\varrho (\diff \omega) \diff \omega) = 0\\
					i^{\ast}_{\partial \mathrm{B}_\rho} \omega = i^{\ast}_{\partial \mathrm{B}_\rho} a
				\end{array}\right. 
			\end{equation} 
			Moreover, since $\omega$ is defined up to $\diff \phi$ with $\phi \in  \mathrm{W}^{1, p}_0 \left(\mathrm{B}_\rho,\mathfrak{g} \right)$, we can assume also that
			\begin{equation}
				\label{cgauge}
				\diff^{\ast} (\omega - a) = 0\, .
			\end{equation} For more details on this construction, see \cite[lemma 2.2]{isobe_regularity_2008}.
			\begin{lemma}
				\label{monoHdomega}
				For all $\tau \in (0,1)$ and $\rho\in [0,R]$, 
				\begin{equation}
					\int_{\mathrm{B}_{\tau\rho}} H(|\diff \omega|^2) \leq  C \tau^4 \int_{\mathrm{B}_\rho} H(|\diff \omega|^2),
				\end{equation} where $C$ is independent of $\rho$,  $\tau$ and $p$.
			\end{lemma}
			\begin{proof}
				According to \cite[theorem 1.10]{uhlenbeck_regularity_1977}, if $f$ is defined as \[ f(Q) = Q(1+Q)^{p/2-1} + \frac{1}{p}\left(2 - (1+Q)^{p/2}\right),\] then $f(|\diff \omega|^2)$ is subharmonic so for all $\tau \in (0,1)$ and $\rho\in [0,R]$, 
				\begin{equation}
					\fint_{\mathrm{B}_{\tau\rho}} f(|\diff \omega|^2) \leq   \fint_{\mathrm{B}_\rho} f(|\diff \omega|^2) \label{fsubharmonicf}
				\end{equation} 
				which gives
				\begin{equation}
					\int_{\mathrm{B}_{\tau\rho}} f(|\diff \omega|^2) \leq  \tau^4 \int_{\mathrm{B}_\rho} f(|\diff \omega|^2). \label{subharmonicf}
				\end{equation} 
				Additionally, since $t=(1+Q)^{-1} \in(0,1)$, \[ \frac{f(Q)}{(1+Q)^{p/2}} = 1-\frac{1}{p}+ \frac{2}{p}t^{p/2} - t\] is a decreasing function of $t$ and \[\frac{1}{p} \leq \frac{f(Q)}{(1+Q)^{p/2}} \leq \frac{p-1}{p}.\] In particular, for $p\in[2,3]$, \begin{equation}
					\frac{1}{3} H(\diff \omega)\leq f(|\diff \omega|^2) \leq \frac{2}{3} H(\diff \omega) \label{comparisonfH}
				\end{equation}  and we deduce, combining \eqref{subharmonicf} and \eqref{comparisonfH}, for all $\tau\in(0,1)$,
				\begin{equation}
					\label{monomega}
					\int_{\mathrm{B}_{\tau\rho}} H(\diff \omega) \leq C \tau^4 \int_{\mathrm{B}_\rho} H(\diff \omega).
				\end{equation} 
			\end{proof}
			
			The next lemma measures the error in replacing our equation by the abelian one.
			\begin{lemma}\label{lem_ecart lin-nnlin} There exists $\omega \in \mathrm{W}^{1, p} \left(
				\mathrm{B}_\rho, \Lambda^1 \R^4 \otimes \mathfrak{g} \right)$and $C_p>0$  such that
				\begin{equation}
					\begin{split}
						\int_{\mathrm{B}_\rho} | F - \diff \omega |^p &\leq C_p \left( \| A \|_{\mathrm{L}^{2p}(\mathrm{B}_{\rho})}^2 
						E^{\frac{1}{p'}} + \| A \|_{\mathrm{L}^{4}(\mathrm{B}_{\rho})}^{p'}E \right),\\
						& \leq C_p  \| A \|_{\mathrm{L}^{4}(\mathrm{B}_{\rho})} E^{\frac{1}{p'}}\left(\| A \|_{\mathrm{L}^{p^*}(\mathrm{B}_{\rho})}+\| A \|_{\mathrm{L}^{4}(\mathrm{B}_{\rho})}^{p'-1}
						E^{\frac{1}{p}} \right)
					\end{split}
					\label{ecart lin-nnlin}
				\end{equation} where, $\frac{1}{p}+\frac{1}{p'}=1$, $p^*=\frac{4p}{4-p}$ and  \[ E =\int_{\mathrm{B}_\rho} H (F) + | A |^{2 p} \]
				and $C_p$ is uniformly bounded in $p\in[2,3]$.
			\end{lemma}
			
			\begin{proof}
				We will use the following classical inequality, see \cref{leminq} for details, 
				\begin{equation}
					\label{classicalinequality}
					(\varrho (a) a - \varrho (b) b, a - b) \geq \frac{1}{2} | a - b |^p.
				\end{equation}
				If we take $a = A$, since $\varphi = A - \omega$ satisfies $i^{\ast}_{\partial
					\mathrm{B}_\rho} \varphi = 0$, thanks to  \ref{main} and \eqref{pYMab}, we get
				\begin{eqnarray*}
					\int_{\mathrm{B}_\rho} \varrho (F) \langle F, \diff \varphi \rangle & = & -
					\int_{\mathrm{B}_\rho} \varrho (F) \langle F, [A, \varphi] \rangle,\\
					\int_{\mathrm{B}_\rho} \varrho (\diff \omega) \langle \diff \omega, \diff
					\varphi \rangle & = & 0 .
				\end{eqnarray*}
				Now, using those equations and \eqref{classicalinequality}, we get
				\begin{equation}
					\begin{split}
						\int_{\mathrm{B}_\rho} | F - \diff \omega |^p & \leq  2 \int_{\mathrm{B}_\rho}
						\langle \varrho (F) F - \varrho (\diff \omega) \diff \omega, F - \diff
						\omega \rangle\\
						& \leq  2 \int_{\mathrm{B}_\rho} \langle \varrho (F) F - \varrho (\diff
						\omega) \diff \omega, \diff \varphi + A \wedge A \rangle\\
						& \leq  2 \left( \int_{\mathrm{B}_\rho} (\varrho (F) F - \varrho (\diff
						\omega) \diff \omega, A \wedge A) - \int_{\mathrm{B}_\rho} \varrho (F) \langle
						F, [A, A - \omega] \rangle \right)\\
						& \leq  4\left( \int_{\mathrm{B}_\rho} | \varrho (F) F - \varrho (\diff
						\omega) \diff \omega |\, | A |^2 + \int_{\mathrm{B}_\rho} | \varrho (F) F |
						\, | A | \, | A - \omega | \right)\\
						& \leq  4 \| \varrho (F) F - \varrho (\diff \omega) \diff \omega
						\|_{\mathrm{L}^{p'} \left( \mathrm{B}_\rho \right)} \| A \|_{\mathrm{L}^{2 p} \left(
							\mathrm{B}_\rho \right)}^2\\
						&   + 4 \| | A | \, | A - \omega | \|_{\mathrm{L}^p \left( \mathrm{B}_\rho
							\right)} \| \varrho (F) F \|_{\mathrm{L}^{p'} \left( \mathrm{B}_\rho \right)}.
					\end{split}
					\label{Fdo}
				\end{equation} 	
				Since, by construction,  $\| H (\diff \omega) \|_{\mathrm L^1 \left(\mathrm \mathrm{B}_\rho \right)}
				\leq \| H (\diff A) \|_{\mathrm L^1 \left( \mathrm \mathrm{B}_\rho \right)}$, we get
				\[ \| \varrho (F) F - \varrho (\diff \omega) \diff \omega \|_{\mathrm{L}^{p'} \left(
					\mathrm{B}_\rho \right)} \leq \| H (F) \|_{\text{$\mathrm{L}^1 \left(
						\mathrm{B}_\rho \right)$}}^{1 / p'} + \| H (\diff \omega) \|_{\text{$\mathrm{L}^1
						\left( \mathrm{B}_\rho \right)$}}^{1 / p'} \leq \| H (F) \|_{\mathrm{L}^1
					\left( \mathrm{B}_\rho \right)}^{1 / p'} + \| H (\diff A) \|_{\mathrm{L}^1 \left(
					\mathrm{B}_\rho \right)}^{1 / p'}. \]
				Moreover
				$$
				\| H (\diff A) \|_{\mathrm{L}^1 \left( \mathrm{B}_\rho \right)} =
				\int_{\mathrm{B}_\rho} (1 + | F_A - A \wedge A |^2)^{\frac{p}{2}} \leq 2^{p-1}
				\int_{\mathrm{B}_\rho} H (F) + | A |^{2 p} .$$
				hence we get 
				\begin{equation}
					\label{rhoFrhoA} 
					\| \varrho (F) F - \varrho (\diff \omega) \diff \omega \|_{\mathrm{L}^{p'}
						\left( \mathrm{B}_\rho \right)} \leq 2^{\frac{(p-1)^2}{p}+1}\left( \int_{\mathrm{B}_\rho} H (F) + | A
					|^{2 p} \right)^{1 / p'} .
				\end{equation}
				Additionally, since $d^*(A-\omega)=0$, $i_{\partial B_\rho}^*(A-\omega)$ vanishes and  $\frac{1}{p^{\ast}} + \frac{1}{4} = \frac{1}{p}$, we get
				\[ \| | A | \, | A - \omega | \|_{\mathrm{L}^p \left( \mathrm{B}_\rho \right)}
				\leq  \| A \|_{\mathrm{L}^4} \| A - \omega \|_{\mathrm{L}^{p^{\ast}}} \leq C_p  \| A \|_{\mathrm{L}^4} \| \diff A - \diff \omega \|_{\mathrm{L}^{p}} \] where $C_p$ is the Sobolev embedding constant. Then
				\begin{equation}
					\label{Amo}
					\| | A | \, | A - \omega | \|_{\mathrm{L}^p \left( \mathrm{B}_\rho \right)}
					\leq  C_p  \| A \|_{\mathrm{L}^4} \| F - \diff \omega \|_{\mathrm{L}^{p}} + C_p \| A \|_{\mathrm{L}^4} \| A \|_{\mathrm{L}^{2p}}^2. 
				\end{equation} 
				
				Hence
				\begin{equation}
					\label{Amobis}
					\begin{split}
						\| | A | \, | A - \omega | \|_{\mathrm{L}^p \left( \mathrm{B}_\rho \right)} &\| \varrho (F) F \|_{\mathrm{L}^{p'} \left( \mathrm{B}_\rho \right)}
						\leq  C_p ( \| A \|_{\mathrm{L}^4} \| F - \diff \omega \|_{\mathrm{L}^{p}} E^\frac{1}{p'}+  \| A \|_{\mathrm{L}^4} \| A \|_{\mathrm{L}^{2p}}^2E^\frac{1}{p'}) \\
						&\leq \frac{1}{2}\| F - \diff \omega \|_{\mathrm{L}^{p}}^p +C_p(\| A \|_{\mathrm{L}^4}^{p'} E+  \| A \|_{\mathrm{L}^4} \| A \|_{\mathrm{L}^{2p}}^2E^\frac{1}{p'}). \\
					\end{split}
				\end{equation} 
				Then thanks to \eqref{Fdo}, \eqref{rhoFrhoA} and \eqref{Amo}, we finally get
				\begin{equation*}
					\int_{\mathrm{B}_\rho} | F - \diff \omega |^p \leq C_p \left( \| A \|_{\mathrm{L}^{2p}(\mathrm{B}_{\rho})}^2 
					E^{\frac{1}{p'}} +\| A \|_{\mathrm{L}^4}^{p'} E\right).
				\end{equation*} 
				Using Young inequality we obtain the first estimate.
				Moreover, since $\frac{1/2}{p^*}+\frac{1/2}{4}=\frac{1}{2p}$, the interpolation inequality yields
				\begin{equation}
					\label{inter}
					\Vert A \Vert_{\mathrm{L}^{2p}(\mathrm{B}_\rho)}^2 \leq  \Vert A \Vert_{\mathrm{L}^{p^*}(\mathrm{B}_\rho)} \Vert A \Vert_{\mathrm{L}^{4}(\mathrm{B}_\rho)}.\end{equation}
				which proves the second estimate.
				
			\end{proof}
			
			\paragraph{From curvature estimates to connection estimates}
			
			\begin{lemma} \label{SoboCoulomb}
				In Uhlenbeck's Coulomb gauge on $\mathrm{B}_R$, there exists $C>0$, such that , for $2\leq p\leq 3\}$,we have
				\[ \|A\|_{\mathrm{L}^{p^*}(\mathrm{B}_R)} \leq C \|\diff A\|_{\mathrm{L}^{p}(\mathrm{B}_R)}  \leq C \|F\|_{\mathrm{L}^{p}(\mathrm{B}_R)} \]
			\end{lemma}
			\begin{proof}
				This is a consequence of Sobolev embeddings.
			\end{proof}
			\begin{lemma} There exists $C_p>0$ depending only in $p$ and uniformly bounded for $p\in[2,3]$ such that for all $\rho\in(0,R)$ and $\tau \in (0,1)$, \begin{equation}
					\int_{\mathrm{B}_{\tau\rho}} | A |^{2 p} \leq C_p \left( \left( \tau^4 + \varepsilon_0 ^p\right) \int_{\mathrm{B}_\rho} | A |^{2 p} + \varepsilon_0^p \int_{\mathrm{B}_\rho} | F |^p \right).
					\label{connectionestimates}
				\end{equation}    
			\end{lemma}
			
			\begin{proof}
				
				Consider $\omega'$ the harmonic one form such that $i^*_{\partial \mathrm{B}_\rho} \star
				(A - \omega') = 0$. We can check that $|\omega'|^2$ is subharmonic, and since $p\geq 2$, so is $|\omega|^p$ thus satisfies for all $\tau\in(0,1)$ \begin{equation}
					\fint_{\mathrm{B}_{\tau\rho}} | \omega' |^{2 p} \leq  \fint_{\mathrm{B}_\rho} | \omega' |^{2 p} 
				\end{equation} \textit{i.e.} for any $\tau\in(0,1)$,
				\begin{equation}
					\label{momega'}
					\int_{\mathrm{B}_{\tau\rho}} | \omega' |^{2 p} \leq \tau^4 \int_{\mathrm{B}_\rho} | \omega' |^{2 p} 
				\end{equation}

				\noindent We deduce, thanks to the previous inequalities, that
				\begin{eqnarray*}
					\int_{\mathrm{B}_{\tau\rho}} | A |^{2 p} & \leq & 2^{2p-1} \left(
					\int_{\mathrm{B}_{\tau\rho}} | \omega' |^{2 p} + \int_{\mathrm{B}_{\tau\rho}} | A - \omega'
					|^{2 p} \right)\\
					& \leq & 2^{2p-1} \left( C_p \tau^4 \int_{\mathrm{B}_\rho} |
					\omega' |^{2 p} + \int_{\mathrm{B}_\rho} | A - \omega' |^{2 p} \right)\\
					& \leq & 2^{2p-1} \left( C \tau^4 \int_{\mathrm{B}_\rho} | A + (A-\omega')|^{2 p} + \int_{\mathrm{B}_\rho} | A - \omega' |^{2 p} \right)\\ 
					& \leq & 2^{2p-1} \left( 2^{2p-1} C \tau^4 \int_{\mathrm{B}_\rho} |
					A |^{2 p} + (1 + 2^{2p-1} C)\int_{\mathrm{B}_\rho} | A - \omega' |^{2 p} \right)
				\end{eqnarray*}
				Since $\frac{1 / 2}{p^{\ast}} + \frac{1 / 2}{4} = \frac{1}{2 p}$, using an
				interpolation estimate, we get
				\[ \| A - \omega' \|_{\mathrm{L}^{2 p}} \leq  \| A - \omega'
				\|_{\mathrm{L}^4}^{1 / 2} \| A - \omega' \|^{1 / 2}_{\mathrm{L}^{p^{\ast}}} .\]
				We combine once again Sobolev embeddings, Gaffney inequality and the Coulomb and boundary conditions,
				\begin{equation}
					\| A - \omega' \|_{\mathrm{L}^{2 p} \left( \mathrm{B}_\rho \right)} \leq C_p \|
					\diff A \|_{\mathrm{L}^2 \left( \mathrm{B}_\rho \right)}^{1 / 2} \| \diff A
					\|_{\mathrm{L}^p \left( \mathrm{B}_\rho \right)}^{1 / 2} \leq C_p \| F
					\|_{\mathrm{L}^2 \left( \mathrm{B}_{R} \right)}^{1 / 2} \| \diff A
					\|_{\mathrm{L}^p \left( \mathrm{B}_\rho\right)}^{1 / 2} \label{sobolevestimate}
				\end{equation}
				and we obtain
				\begin{eqnarray*}
					\int_{\mathrm{B}_{\tau\rho}} | A |^{2 p} & \leq & C_p  \left( \tau^4 \int_{\mathrm{B}_\rho} | A |^{2 p} + \| F \|_{\mathrm{L}^2
						\left( \mathrm{B}_{R} \right)}^p \int_{\mathrm{B}_\rho} | \diff A |^p \right)
				\end{eqnarray*}
				since $\diff A = F - A \wedge A$, we conclude
				\begin{equation*}
					\int_{\mathrm{B}_{\tau\rho}} | A |^{2 p} \leq C_p \left( \left( \tau^4 + \varepsilon_0 ^p\right) \int_{\mathrm{B}_\rho} | A |^{2 p} + \varepsilon_0^p \int_{\mathrm{B}_\rho} | F |^p \right).
				\end{equation*}
				
			\end{proof}
			
			\paragraph{Morrey growth of the curvature}
			\begin{lemma}For every $\delta_0 >0$, $\varepsilon_0 $ can be chosen small enough such that \begin{equation}
					\int_{\mathrm{B}_{\rho}} H (F) \leq C \left( \frac{\rho}{R} \right)^{4 -
						\delta_0} \int_{\mathrm{B}_R} H (F) \label{estimeecourbure}
				\end{equation}
			\end{lemma}
			
			\begin{proof}
				Thanks to \eqref{epsi} and gauge extraction of Uhlenbeck, for every $\rho \in (0,R)$, there exists a Coulomb gauge $g_\rho$ in $\mathrm{B}_\rho$ such that if $A_\rho = A^{g_\rho}$, we have $i^{\ast}_{\partial \mathrm{B}_\rho} \star A_\rho = 0$,
				$\diff^{\ast} A_\rho = 0$ in $\mathrm{B}_\rho$ and
				\begin{equation}
					\label{gaugep}	
					\int_{\mathrm{B}_\rho} \frac{1}{\rho^p} | A_\rho |^p + | \nabla A_\rho |^{4-p} \leq C
					\int_{\mathrm{B}_\rho} | F |^p .
				\end{equation} We write $F_\rho$ the curvature of $A_\rho$. Consider $\omega$ as in section \ref{abelian} associated to $A_\rho$. We want to apply \eqref{ecart lin-nnlin} to $F_\rho$ and $\omega$. First, we deduce from \cref{SoboCoulomb} and Sobolev embedding that ,
				\begin{equation}
					\label{AR1} \begin{split}
						\| A_\rho \|_{\mathrm{L}^{p^*} \left( \mathrm{B}_\rho \right)} & \leq  \| F \|_{\mathrm{L}^p \left(
							\mathrm{B}_\rho \right)}\\
						\| A_\rho \|_{\mathrm{L}^{4} \left( \mathrm{B}_\rho \right)} & \leq  \| F \|_{\mathrm{L}^2 \left(
							\mathrm{B}_\rho \right)} \leq \varepsilon_0\\
						\| A_\rho \|^2_{\mathrm{L}^{2p} \left( \mathrm{B}_\rho \right)} & \leq  \| A_\rho \|_{\mathrm{L}^{p^*} \left( \mathrm{B}_\rho \right)} \| A_\rho \|_{\mathrm{L}^{4} \left( \mathrm{B}_\rho \right)} \leq \varepsilon_0 \| F \|_{\mathrm{L}^p \left(
							\mathrm{B}_\rho \right)},
					\end{split}
				\end{equation}
				hence,   
				\begin{equation}
					\label{AR2}
					\begin{split}
						\int_{\mathrm{B}_\rho} H(F_\rho) +|A_\rho|^{2p} & \leq \int_{\mathrm{B}_\rho} H(F) + C \varepsilon_0^p |F|^p\\
						& \leq  C \int_{\mathrm{B}_\rho} H(F)
					\end{split}
				\end{equation} and from \cref{lem_ecart lin-nnlin}, we obtain
				\begin{eqnarray*}
					\int_{\mathrm{B}_\rho} | F_\rho - \diff \omega |^p & \leq & C \| A_\rho \|_{\mathrm{L}^4\left( \mathrm{B}_\rho \right)} 
					\left(  \int_{\mathrm{B}_\rho} H (F) \right)^{1 / p'} \left( \|F
					\|_{\mathrm{L}^{p} \left( \mathrm{B}_\rho \right)} +  \| A_\rho
					\|_{\mathrm{L}^{4} \left( \mathrm{B}_\rho \right)}^{\frac{3p'}{4}-1} \left(\int_{\mathrm{B}_\rho} H (F) \right)^{1 / p'}\right)\\
					& \leq & C  \| A_\rho \|_{\mathrm{L}^4\left( \mathrm{B}_\rho \right)} \left(  \int_{\mathrm{B}_\rho} H (F) \right)^{1 / p'} (1+ \| A_\rho \|_{\mathrm{L}^4\left( \mathrm{B}_\rho \right)}^{\frac{3p'}{4}-1} )  \left(  \int_{\mathrm{B}_\rho} H (F) \right)^{1 / p} \\
					& \leq & C \varepsilon_0(1+\varepsilon_0^{\frac{3p'}{4}-1})  \int_{\mathrm{B}_\rho} H (F).
				\end{eqnarray*}
				If $\varepsilon_0 \leq1 $, then
				\begin{equation}
					\label{Fomega}
					\begin{split}
						\int_{\mathrm{B}_\rho} | F_\rho - \diff \omega |^p &  \leq  C \varepsilon_0
						\int_{\mathrm{B}_\rho} H (F).
					\end{split}
				\end{equation}
				Therefore, thanks to \cref{monoHdomega} and \eqref{Habestim}, we get, for every $\tau \in (0,1)$,
				\begin{eqnarray*}
					\int_{\mathrm{B}_{\tau\rho}} H (F) & \leq & 2^{p-1} \int_{\mathrm{B}_{\tau\rho}} | F_\rho - \diff \omega |^p +
					\int_{\mathrm{B}_{\tau\rho}} H (\diff \omega)\\
					& \leq & C \left( \int_{\mathrm{B}_\rho} | F_\rho - \diff \omega |^p + \tau^4 \int_{\mathrm{B}_\rho} H (\diff \omega) \right)\\
					& \leq & C \left( \int_{\mathrm{B}_\rho} | F_\rho - \diff \omega |^p +\tau^4 \int_{\mathrm{B}_\rho} H (F) \right)\\
					& \leq & C \left( \tau^4 + \varepsilon_0 \right) 	\int_{\mathrm{B}_{\rho}} H (F)
				\end{eqnarray*}
				where $C$ depends only on $p$ and is uniformly bounded. For every $\delta_0 > 0$, using a classical result (see for instance \cite[lemma 5.13]{GM}), we can chose $\varepsilon_0 = \varepsilon
				(\delta_0)$ small enough such that the last inequality implies \begin{equation}
					\int_{\mathrm{B}_{\rho}} H (F) \leq C \left( \frac{\rho}{R} \right)^{4 -
						\delta_0} \int_{\mathrm{B}_R} H (F)
				\end{equation}
			\end{proof}

			Since $| F |^p \leq H (F)$ and the estimate of the lemma is translation invariant, we deduce the following result:
			\begin{proposition}\label{H(F) morrey}$F$ is in the Morrey space $M_{\loc}^{p, 4 -
					\delta_0}(\mathrm{B}_{R})$, where
				$$M_{\loc}^{p, \lambda}(\Omega)=\{ u\in L^p_{loc}(\Omega)\, \text{s.t. }\sup_{B(x,\rho)\subset \Omega} \frac{1}{\rho^\lambda}\int_{B(x,\rho)} \vert u\vert^p \; <\infty\}$$
				and \[ \| F\|_{L^{p, 4 -
						\delta_0}(\mathrm{B}_{R/2})}^p \leq   \frac{C}{R^{4 -
						\delta_0} } \int_{\mathrm{B}_R} H (F).
				\]
			\end{proposition}
			
			\subsubsection{Hölder-continuity of $A$}
			\label{Aholder}
			
			\begin{proposition} For every $\delta>0$, $\varepsilon_0 $ can be chosen small enough such that 
				$\nabla A \in M_{\loc}^{p, 4 - \delta}$, and by Morrey's theorem, $A\in \mathcal{C}_{\loc}^{0, 1 - \delta / p}$ with the following estimates \begin{equation}
					\label{AHolderestimate} [A ]_{\mathcal{C}^{0, 1 - \delta / p}\left(\mathrm{B}_{R/2}\right)}^p \leq   \frac{C}{R^{4 - \delta} }\int_{\mathrm{B}_R} H(F), 
				\end{equation} and \begin{equation}
					\label{Apointwiseestimate} \|A \|_{\mathrm{L}^{\infty}\left(\mathrm{B}_{R/2}\right)}^p \leq   \frac{C}{R^{4 - p} }\int_{\mathrm{B}_R} H(F).
				\end{equation}
			\end{proposition}
			
			\begin{proof}
				We denote
				\[ \Phi (\rho) = \int_{\mathrm{B}_{\rho}} | A |^{2 p} + H (F) \] From \eqref{connectionestimates} and  \eqref{estimeecourbure}, we have for all $\rho\in(0,R),\tau\in(0,1)$,
				\begin{eqnarray*}
					\Phi (\tau\rho) & \leq & C \left( \left( \tau^{4 -
						\delta_0} + \varepsilon_0^p \right)
					\int_{\mathrm{B}_\rho} H (F) + \left( \tau^4 + \varepsilon^p \right) \int_{\mathrm{B}_\rho} |
					A |^{2 p} \right)\\
					& \leq & C \left( \tau^{4 -
						\delta_0} + \varepsilon_0^2 \right) \Phi (\rho)
				\end{eqnarray*}
				Thanks to classical result, see \cite[lemma 5.13]{GM}, if $\varepsilon_0$ is small enough,  for $\delta <\delta_0$  there exists $C>0$ such that 
				\begin{equation}
					\label{Estimphi}
					\Phi (\rho) \leq C \left( \frac{\rho}{R}\right)^{4 - \delta} \Phi (R) \leq C\left(\frac{\rho}{R}\right)^{4 - \delta} \int_{\mathrm{B}_R} H(F),\text{  for all } \rho \leq R. 
				\end{equation}
				Since $\diff A =
				F - A \wedge A$, we have
				\begin{equation}
					\label{dAp} 
					\int_{\mathrm{B}_{\rho}} | \diff A |^p \leq C \Phi (\rho) \leq C
					\left( \frac{\rho}{R} \right)^{4 - \delta}  \int_{\mathrm{B}_R} H(F)
				\end{equation}
				Let's deduce the same estimate for \[\psi (\rho) = \int_{\mathrm{B}_{\rho}} |
				\nabla A |^p.\] Let $\omega'$ an harmonic form as in section \eqref{connectionestimates}, thanks to \eqref{momega'}, Gaffney's inequality, the fact that $A$ is Coulomb and \eqref{dAp}, we get for all $\rho\in(0,R),\tau\in(0,1)$,
				
				\begin{eqnarray*}
					\psi (\tau \rho) & \leq & 2^{p-1} \left( \int_{\mathrm{B}_{\tau \rho}} | \nabla \omega'
					|^p + \int_{\mathrm{B}_{\tau\rho}} | \nabla (A - \omega') |^p \right)\\
					& \leq & C \left( \tau^4 \int_{\mathrm{B}_\rho} |
					\nabla \omega' |^p + \int_{\mathrm{B}_\rho} | \nabla (A - \omega') |^p \right)\\
					& \leq & C \left( \tau^4 \int_{\mathrm{B}_\rho} |
					\nabla A + \nabla (A-\omega') |^p + \int_{\mathrm{B}_\rho} | \nabla (A - \omega') |^p \right)\\
					& \leq & C \left( 2^{p-1} \tau^4 \psi (\rho) + (1+2^{p-1})
					\int_{\mathrm{B}_\rho} | \nabla (A - \omega') |^p \right)\\
					& \leq & C \left( \tau^4 \psi (\rho) +
					\int_{\mathrm{B}_\rho} | \diff (A - \omega') |^p + | \diff^{\ast} (A - \omega')
					|^p \right)\\
					& \leq & C \left( \tau^4 \psi (\rho) +
					\int_{\mathrm{B}_\rho} | \diff A |^p \right)\\
					& \leq & C \tau^4 \psi (\rho) + C \left(\frac{\rho}{R}\right)^{4 - \delta} \Phi (R)
				\end{eqnarray*}
				This implies
				\begin{equation}
					\label{Hestim} \psi (\rho) \leq C \left(\frac{\rho}{R}\right)^{4 - \delta} (\psi
					(R) + \Phi (R)) \leq  C\left(\frac{\rho}{R}\right)^{4 - \delta} \int_{\mathrm{B}_R} H(F)
				\end{equation}
				This concludes the proof.
			\end{proof}
			
			\subsubsection{Hölder-continuity of $F$}
			
			The goal of this section is the proof of the following proposition. \begin{proposition}
				$F$ and $\nabla A$ are Hölder continuous. \label{FHolderC}
			\end{proposition}
			
			\begin{proof} If $V (F)$ is Hölder-continuous, then by \cref{leminq}, we have
				\[ \langle V (a) - V (b), a - b \rangle \geq
				\frac{1}{C} | a-b |^{\frac{p}{2} + 1}, \]
				using Cauchy-Schwarz inequality \begin{equation}
					| a-b | \leq C | V (a) - V (b) |^{2/p} \label{VtoF}
				\end{equation}
				so $F$ also Hölder continuous. From $\diff A = F - A \wedge A$, we
				deduce that $\diff A$ is also Hölder continuous. Since $\diff^{\ast} A =	0$, $\nabla A$ is Hölder-continuous (see for instance \cite[theorem 2]{bolik_h_2001}).
				
				So it suffices to prove that $V (F)$ is Hölder-continuous. Let's denote, for any $ \mathfrak{g}$-valued $2$-form $\mu$ and $\rho > 0$,
				\[ \Psi (\mu, \rho) = \int_{\mathrm{B}_{\rho}} | V (\mu) - V (\mu)_{\rho} |^2
				\diff x \]
				where $f_\rho = \fint_{\mathrm{B}_{\rho}} f \, \diff x$.
				From Campanato's theorem (see  \cite[theorem 5.5]{GM}), it is enough to show that \[ \Psi(F,\rho)\leq C_A \rho^{4+2\sigma}\] for some $\sigma>0$.
				
				Let $h \in \mathcal{C}^{\infty}(\mathrm{B}_R,G)$ be defined by $h (x) = \exp \left( -\varphi(x)\right)$ where \[\varphi(x)=
				\chi (x) \sum_{i = 1}^4 x^i A_i (0) \] and $\chi$ is a cut-off at
				$0$ satisfying $|\diff\chi(x)|\leq C R^{-1}$ with $C$ independent of $R$. It is such that $A^h (0) = A (0) + \diff_0 h = A (0) - \sum_{i = 1}^4
				A_i (0) \diff x^i = 0.$ Since $h$ is smooth and $A^h$ is Hölder continuous, we have \begin{equation}
					|A^h|(x) \leq C_A |x|^{1-\delta/p} \label{pointwiseAh}
				\end{equation} where $C_A$ depends on the $\mathcal{C}^{0,1-\delta/p}$ norm of $A$. Write $F^h := F_{A^h} = h^{-1} F h$ and define $f(t) = h_t^{-1} F h_t$ where $h_t = \exp(-t\varphi)$, then \begin{align*}
					|F^h-F| = |f(1)-f(0)| &\leq \sup_{t\in [0,1]} |f'(t)| \\
					& \leq C_G |\varphi| |F| \\
					& \leq C_A |x| |F| 
				\end{align*} so, using \eqref{estimeecourbure}, we obtain
				\begin{align*}
					\int_{\mathrm{B}_\rho}|V(F^h)-V(F)|^2&= \int_{\mathrm{B}_\rho}(1+|F|^2)^{(p-2)/2}|F^h-F|^2 \\
					& \leq C_A \int_{\mathrm{B}_\rho}(1+|F|^2)^{(p-2)/2}|F|^2  |x|^2 \\
					&\leq C_A \rho^2 \int_{\mathrm{B}_\rho}(1+|F|^2)^{(p-2)/2}|F|^2 \\
					&\leq C_A \rho^2 \int_{\mathrm{B}_\rho}H(F) \\
					&\leq C_A \rho^{6-\delta}.
				\end{align*} Therefore, it is enough to show \begin{equation}
					\Psi(F^h,\rho)\leq C_A \rho^{4+2\sigma}
				\end{equation}  Let $\omega $ as in section \ref{abelian}, with boundary condition to be determined later. According to \cite[Lemma 2.1]{hamburger_regularity_1992}, we can apply  \cite[Theorem 4.1]{hamburger_regularity_1992} to $\diff \omega$: we get that $V (\diff \omega)$ satisfies, for some $\sigma \in] 0
				; 1 [$,
				\begin{equation}
					\label{Phiomega} \Psi (\diff \omega, \tau\rho) \leq C \tau^{4 + 2
						\sigma} \Psi (\diff \omega, \rho), \forall \rho\in (0,R),\forall \tau\in (0,1). 
				\end{equation} 
				On one hand, using Jensen's inequality, we have
				\begin{equation}
					\label{PF}
					\begin{split}
						\Psi (F^h, \rho) &= \int_{\mathrm{B}_{\rho}} | V (F^h) - V (F^h)_{\rho} |^2 \diff x\\
						&  \leq 4 \left( \int_{\mathrm{B}_{\rho}} | V (F^h) - V (\diff \omega) |^2 + | V (F^h)_\rho - V (\diff \omega)_\rho |^2 +| V (\diff \omega) - V (\diff \omega)_{\rho} |^2   \diff x\right) \\
						&  \leq 8 \left( \int_{\mathrm{B}_{\rho}} | V (F^h) - V (\diff \omega) |^2    \diff x + \Psi (\diff \omega, \rho) \right) 
					\end{split}
				\end{equation}
				On the other hand, by a similar computation, we have
				\begin{equation}
					\label{Pomega} \Psi (\diff \omega, \rho) \leq 8 \left( \int_{\mathrm{B}_{\rho}} | V
					(F^h) - V (\diff \omega) |^2 \diff x + \Psi (F^h, \rho) \right) .
				\end{equation} Then, thanks to  \eqref{Phiomega}, \eqref{PF} and \eqref{Pomega}, we have for all $\rho\in(0,R)$ and $\tau\in(0,1)$, \begin{align}
					\Psi (F^h, \tau \rho)  &\leq  C \left( \int_{\mathrm{B}_{\tau\rho}} | V (F^h) - V (\diff
					\omega) |^2 \diff x +  \tau^{4 + 2 \sigma}\left(\int_{\mathrm{B}_\rho} | V (F^h) - V (\diff
					\omega) |^2 \diff x + \Psi (F^h, \rho)\right)\right)\notag \\ 
					& \leq C \left( \int_{\mathrm{B}_{\rho}} | V (F^h) - V (\diff
					\omega) |^2 \diff x +  \tau^{4 + 2 \sigma}\Psi (F^h, \rho)\right)
				\end{align} 
				By  the Claim below, we can choose $\omega$ such that, for some $\gamma >0$  and for every $\rho \in (0,R)$, we have 
				\begin{equation}
					\int_{\mathrm{B}_\rho} | V (F^h) - V (\diff \omega) |^2 \diff x
					\leq C_A \rho^{4 + 2\gamma},
				\end{equation} where $C_A$ can depend on $A$.
				We deduce that $\Psi(F^h,\cdot)$ satisfies
				\[ \Psi (F^h,\tau \rho) \leq C_A \rho^{4 + 2\gamma} + C\tau^{4
					+ 2 \sigma} \Psi (F^h,\rho) \]
				and so
				\[ \Psi (F^h,\rho) \leq C \rho^{4 + 2\sigma'}(R^{-4 -2\sigma'}\Psi
				(F^h,R) + C_A), \] where $\sigma'=\min(\sigma,\gamma)$. Since this estimate in translation invariant, similarly to the proof of \cref{H(F) morrey}, we get that $V(F)$ is in the Campanato-space $\mathcal{L}_{\mathrm{loc}}^{2, 4+2\sigma'}$. Use then the embedding of this space in $\mathcal{C}_{\mathrm{loc}}^{\sigma'}$ to deduce the Hölder-continuity of $V(F)$.
			\end{proof}
			
			\textbf{Claim}  There exists $\gamma>0$ such that, for every $\rho\in(0,R)$, if $\omega $ is defined as in section \ref{abelian} with $i^{\ast}_{\partial \mathrm{B}_\rho} (\omega - A^h) = 0, \diff^{\ast}
			(\omega - A^h) = 0$, then \begin{equation}
				\int_{\mathrm{B}_\rho} | V (F^h) - V (\diff \omega) |^2 \diff x
				\leq C_A \rho^{4 + 2\gamma}, \label{Vestimate}
			\end{equation} where $C_A$ can depend on $A$.\\
			
			\begin{proof}
				We are going to apply estimate \eqref{ecart lin-nnlin} to $F^h$ and $\omega$. Recall that   \[
				\int_{\mathrm{B}_\rho} | F^h - \diff \omega |^p  \leq C_p  \| A^h \|_{\mathrm{L}^{4}(\mathrm{B}_{\rho})} E^{\frac{1}{p'}}\left(\| A^h \|_{\mathrm{L}^{p^*}(\mathrm{B}_{\rho})}+\| A^h \|_{\mathrm{L}^{4}(\mathrm{B}_{\rho})}^{\frac{1}{p-1}}
				E^{\frac{1}{p}} \right)
				\] where \[ E =\int_{\mathrm{B}_\rho} H (F) + | A^h |^{2 p}. \] Notice that, by contrast to the proof of \eqref{estimeecourbure}, $A^h$ is not in Coulomb gauge anymore. However, by construction, $A^h$ satisfies pointwise estimates of the form \eqref{pointwiseAh}. We can therefore bound every term of the previous inequality. For $q\in\{4,p^*,2p\}$,
				\begin{align*}
					\| A^h \|_{\mathrm{L}^q(\mathrm{B}_\rho)} &\leq C_A \left( \int_{\mathrm{B}_\rho} | x |^{q(1-\delta/p)}
					\right)^{1 / q}\\
					&\leq  C_A \rho^{1-\delta/p+4/q}\end{align*} so using \eqref{estimeecourbure}
				\begin{align*}   \int_{\mathrm{B}_\rho} H(F) + |A^h|^{2p}
					& \leq  C_A \left(\rho^{4 - \delta}+ \rho^{4+2(p-\delta)} \right)\\
					& \leq  C_A \rho^{4 - \delta},
				\end{align*} (where we used that $\delta< 4\leq 2p$) so we deduce  \begin{align*}
					\int_{\mathrm{B}_\rho} | F^h - \diff \omega |^p & \leq   C_A \rho^{2-\delta/p} \rho^{(4-\delta)/p'} \left(\rho^{1-\delta/p +4/p^*} + \rho^{(2-\delta/p)/(p-1)} \rho^{(4-\delta)/p}\right) \\
					& \leq C_A \rho^{2-\delta/p} \rho^{(4-\delta)/p'} \left(\rho^{(4-\delta)/p } + \rho^{(2-\delta/p)/(p-1)} \rho^{(4-\delta)/p}\right)\\
					&\leq   C_A \rho^{6-\delta(1+1/p)}\left(1+  \rho^{(2-\delta/p)/(p-1)}\right)
				\end{align*} for some constant $C_A$ depending on $A$. Since $p> 1$ and $\delta/p<2$, \begin{equation}
					\int_{\mathrm{B}_\rho} | F^h - \diff \omega |^p\leq C_A \rho^{6-\delta(1+1/p)}. \label{EstimFg}
				\end{equation} Moreover, using \eqref{Vin}, we have
				\begin{eqnarray*}
					\int_{\mathrm{B}_\rho} | V (F^h) - V (\diff \omega) |^2 & \leq & 4
					\int_{\mathrm{B}_\rho} (1 + | F |^2 + | \diff \omega |^2)^{\frac{p - 2}{2}} | F^h
					- \diff \omega |^2\\
					& \leq & 4 \left( \int_{\mathrm{B}_\rho} (1 + | F |^2 + | \diff \omega
					|^2)^{\frac{p}{2}} \right)^{\frac{p - 2}{p}} \left( \int_{\mathrm{B}_\rho} |
					\diff \omega - F^h |^p \right)^{2 / p}
				\end{eqnarray*}
				and thanks to \eqref{Estimphi} and \eqref{EstimFg}, we have
				\begin{equation*}
					\int_{\mathrm{B}_\rho} (1 + | F |^2 + | \diff \omega |^2)^{\frac{p}{2}} 
					\leq  2^{\frac{3p}{2}-1} \left( \int_{\mathrm{B}_\rho} H (F) + \int_{\mathrm{B}_\rho} | \diff \omega -
					F^h |^p\right) \leq C_A \rho^{4 - \delta}
				\end{equation*}
				which finally gives,
				\begin{align*}
					\int_{\mathrm{B}_\rho} | V (F^h) - V (\diff \omega) |^2 & \leq  C_A \rho^{(4-\delta)(p-2)/p+2(6-\delta(1+1/p))/p} \\
					& \leq  C_A \rho^{4 - \delta + 2(2-\delta/p)/p}
				\end{align*}
				Finally, $2\gamma = - \delta + 2(2-\delta/ p)/p$ satisfies
				\[ \gamma = \frac{1}{2}\left(- \delta + \frac{2}{p}\left(2 - \frac{\delta}{p} \right)\right) = \frac{2}{p} -
				\delta \left( \frac{1}{2} + \frac{1}{p^2} \right) \geq \frac{2}{3} - \frac{3\delta}{4}\]
				and we choose $\delta$ (independently of $p$) such that $\gamma >0$ and estimate \eqref{Vestimate} holds.
			\end{proof}

			\subsubsection{Weak differentiability of $F$}
			
			\begin{proposition}
				$F$ is in $\mathrm{W}^{1,2}_{\loc}(\mathrm{B}_R)$.
			\end{proposition}
			
			\begin{proof}
				Define $W(\Omega) = \rho(\Omega) \Omega$ for any $2$-form $\Omega$. Then \eqref{main} gives for every $\varphi \in \mathrm{W}^{1, p}_0 \left( \mathrm{B}_R \right)$:
				\[ 0 = \int_{\mathrm{B}_R}  \langle W(F), \diff_A \varphi \rangle, \] in other words, \begin{equation}
					\label{DW} \int_{\mathrm{B}_R}  \langle W(F), \diff\varphi \rangle = -\int_{\mathrm{B}_R}  \langle W(F), [A,\varphi ] \rangle .
				\end{equation} Let's
				choose any test function $\eta \in \mathcal{C}^{\infty}_c \left( \mathrm{B}_{R} \right)$. We use the following notations for the difference quotient
				\[ D_{i, h} u (x) = \frac{u (x + h e_i) - u (x)}{h}, \]
				and
				\[ D_{i, h}^* u (x) = \frac{u (x - h e_i) - u (x)}{h}. \] Let's consider $\varphi = D_{i, h}^{\ast}
				D_{i, h} A$ as a test function. We have \begin{align}
					\int_{\mathrm{B}_R} \langle W(F), \diff \varphi\rangle   & = \int_{\mathrm{B}_R} \langle W(F), D_{i,h}^* \diff (\eta^2D_{i,h}A)\rangle\\
					& =  \int_{\mathrm{B}_R} \eta^2 \langle D_{i, h} W(F), D_{i,h}\diff A\rangle  + \int_{\mathrm{B}_R}2\eta \langle W(F), \diff \eta \wedge D_{i,h} A\rangle\\
					& =  \int_{\mathrm{B}_R} \eta^2 \langle D_{i, h} W(F), D_{i, h} F -D_{i,h}(A\wedge A) \rangle
					+   \int_{\mathrm{B}_R}2\eta \langle D_{i,h}W(F), \diff \eta \wedge D_{i,h} A\rangle \label{Wlowerestimate}.
				\end{align} Notice that \[ D_{i,h}W(F) = \int_{0}^1 W'((1-t)F +t F_{i,h}) D_{i,h}F \diff t \]where $F_{i,h}(x) = F(x+h e_i)$. Since \[ W'(\Omega)\xi = (1+|\Omega|^2)^{(p-2)/2} \left( \xi + (p-2) \frac{\langle\Omega, \xi \rangle}{1+|\Omega|^2}\Omega\right), \] \begin{equation}
					\begin{split} \label{WtoF}
						|D_{i,h} W(F)| &\leq (p-1)| D_{i,h}F| f_{i,h},\\
						\langle D_{i,h} W(F),D_{i,h}F \rangle &\geq | D_{i,h}F|^2 f_{i,h},
					\end{split}
				\end{equation} where \[f_{i,h} := \int_{0}^1 (1+|(1-t)F+t F_{i,h}|)^{(p-2)/2} \diff t . \] We can then deduce from \eqref{Wlowerestimate}, using Young inequality, that \begin{align*}
					\int_{\mathrm{B}_R} \langle W(F), \diff \varphi\rangle &\geq \frac{1}{4} \int_{\mathrm{B}_R} \eta^2 f_{i,h}|D_{i,h} F|^2 + \int_{\mathrm{B}_R} \eta \langle D_{i,h}W(F), 2\diff \eta\wedge D_{i,h}A- \eta D_{i,h}(A\wedge A) \rangle \\
					& \geq \frac{1}{4} \int_{\mathrm{B}_R}\eta^2 f_{i,h}|D_{i,h} F|^2 - \int_{\mathrm{B}_R} \eta |D_{i,h} F|  (p-1) f_{i,h} |2\diff \eta\wedge D_{i,h}A- \eta D_{i,h}(A\wedge A) | \\
					& \geq \frac{1}{8} \int_{\mathrm{B}_R}\eta^2 f_{i,h}|D_{i,h} F|^2 - C(p-1)^2\int_{\mathrm{B}_R} f_{i,h} \left(|\diff \eta|^2 |D_{i,h}A|^2 + \eta^2 |D_{i,h}(A\wedge A) |^2\right).
				\end{align*}
				
				From \eqref{DW}, we have
				\begin{align*}
					\int_{\mathrm{B}_R} \langle W(F), \diff \varphi \rangle  & = \int_{\mathrm{B}_R} \langle W(F),[A,D_{i,h}^*(\eta^2 D_{i,h}A)]
					\rangle \\
					& = \int_{\mathrm{B}_R} \eta^2 \langle D_{i, h} W(F),[A,D_{i, h}A] \rangle  -   \int_{\mathrm{B}_R}\eta^2   \langle W(F), [D_{i,h}A, D_{i,h}A]\rangle,
				\end{align*} so using Young inequality and \eqref{WtoF} again
				\begin{eqnarray*}
					\left| \int_{\mathrm{B}_R} (W(F), \diff \varphi) \right| & \leq & \frac{1}{16}\int_{\mathrm{B}_R} \eta^2 f_{i,h}  \vert D_{i,h}F\vert^2\\ & &+   C(p-1)^2 \int_{\mathrm{B}_R} f_{i,h}  \eta^2|A|^2 | D_{i,
						h} A |^2 + C\int_{\mathrm{B}_R} \eta^2|W(F)| |D_{i,h}A|^2.
				\end{eqnarray*}

				We deduce
				\[ \int_{\mathrm{B}_R} \eta^2  f_{i,h} \vert D_{i,h} F\vert^2 \leq C \int_{\mathrm{B}_R} f_{i,h}(   | D_{i,
					h} A |^2( |\diff \eta|^2 +\eta^2|A|^2) + \eta^2|D_{i,h}(A\wedge A)|^2) +\eta^2 | W(F) |  | D_{i, h}A
				|^2. \]
				When $h$ goes to $0$, this proves that $\partial_i F \in \mathrm{L}^2_{\loc}({\mathrm{B}_{R}})$ and  \[ \int_{\mathrm{B}_R} \eta^2 \varrho(F)  \vert \partial_i F\vert^2 \leq C \int_{\mathrm{B}_R} \varrho(F)   ( |\diff \eta|^2 +\eta^2|A|^2 +\eta^2 | F| )| \partial_i A |^2. \]
			\end{proof}
			
			\subsubsection{Smoothness of $A$}
			
			\begin{proposition} If $p-2$ is small enough, $A$ is smooth.
			\end{proposition} 
			\begin{proof}
				With the regularity information obtained up to this point, the $p$-Yang-Mills equations rewrites
				\[ \diff^{\ast} F = \star [A, \star F] + \frac{p - 2}{2} \star \frac{\diff | F     |^2 \wedge \star F}{1 + | F |^2}. \] This is a second-order non-linear PDE for $A$ and since $\diff^* A=0$, the highest order term is 
				\begin{equation}
					\label{Lp}
					\mathbb{L}_p A := \Delta A - (p-2) \star \frac{\langle \nabla \diff A, F \rangle \wedge \star F}{1+|F|^2}.
				\end{equation}
				This is, for $p-2$ small enough, an elliptic operator with Hölder continuous coefficients (depending only on $F$) and form the equation, we get that $\mathbb{L}_p A$ is Hölder continuous. From standard Schauder regularity theory, this gives that $A$ is in $\mathcal{C}^{2,\sigma}$ for some $\sigma \in (0,1)$. This improves not only the regularity of $\mathbb{L}_p A$ but also the one of the coefficients of $\mathbb{L}_p$. Bootstrapping this argument, we get that $A$ is smooth.
			\end{proof}
			
			\subsubsection{Pointwise estimate for $F$}
			
			\begin{theorem}
				There exists $\varepsilon_0, C > 0$ such that for all $R > 0$ and every smooth
				connection $A$ on $\mathrm{B}_R$ satisfying the $p$-Yang-Mills equation
				\eqref{equationdepYMdiv} and the small energy condition
				\[ \int_{\mathrm{B}_R} | F |^2 < \varepsilon_0, \]
				then
				\[ \sup_{\mathrm{B}_{R / 2}}  | F | \leq \frac{C}{R^2}
				\sqrt{\int_{\mathrm{B}_R} | F |^2} . \] \label{epsregcurvaturepointwiseproof}
			\end{theorem}
			
			\begin{proof}
				Recall that $\phi = | F |^2$ satisfies the following estimate
				\[ \mathcal{L}_p \phi \leq C \phi (1 + \phi^{1 / 2}) \]
				where $\mathcal{L}_p$ is defined in \eqref{Lp}. Define for $\sigma \in [0,
				R]$: $f (\sigma) = (R - \sigma)^4 \max_{\overline{\text{B}}_{\sigma}} \phi$.
				$f$ is a continuous non-negative function vanishing at $\sigma = R$ so there
				exists $\sigma_0 \in [0, R [$ such that $f$ is maximal at $\sigma =
				\sigma_0$. Now since $\phi$ is continuous, there exists $x \in
				\overline{\text{B}}_{\sigma_0}$ such that $\phi_0 :=
				\max_{\overline{\text{B}}_{\sigma_0}} \phi = \phi (x_0)$. Introduce $\rho_0
				= \min ((R - \sigma_0) / 2, 1)$. By definition of $\sigma_0$
				\begin{equation}
					\max_{\overline{\text{B}}_{\rho_0} (x_0)} \phi \leq
					\max_{\overline{\text{B}}_{\rho_0 + \sigma_0}} \phi \leq \left(
					\frac{R - \sigma_0}{R - (\sigma_0 + \rho_0)} \right)^4
					\max_{\overline{\text{B}}_{\sigma_0}} \phi = 16 \phi_0 .
					\label{phiestimate1}
				\end{equation}
				Define now $r_0 = \rho_0 \phi_0^{1 / 4}$. Let's prove by contradiction that,
				if $\delta_0$ is small enough then $r_0 < 1$. Assume that $r_0 \geq 1$
				and consider $v = | \alpha^{\ast} F |^2$ where $\alpha (y) = x_0 + \phi_0^{-
					1 / 4} y$ (defined on $\text{B}_{r_0} \supset \text{B}_1$). We have the
				following expression for $v$:
				\[ v (y) = \frac{1}{\phi_0} | F |^2 \left( x_0 + \frac{y}{\phi_0^{1 / 4}}
				\right) = \frac{1}{\phi_0} \phi \left( x_0 + \frac{y}{\phi_0^{1 / 4}}
				\right), \]
				so, from \eqref{phiestimate1}, $v$ satisfies
				\begin{eqnarray*}
					v (0) & = & 1\\
					\sup_{\overline{\text{B}}_{r_0}} v & = &
					\sup_{\overline{\text{B}}_{\rho_0} (x_0)}  \frac{\phi}{\phi_0} \leq
					16.
				\end{eqnarray*}
				Since $r_0 \geq 1$ and $\rho_0 \leq 1$ then $\phi_0 \geq 1$.
				Consider $\tilde{\mathcal{L}}_p$ defined as
				\[ \tilde{\mathcal{L}}_p v = \partial_{\alpha} ((\delta^{\alpha \beta} + (p
				- 2) \mathcal{A}^{\alpha \beta} \circ \alpha) \partial_{\beta} v) \]
				Then, in $\text{B}_{r_0}$:
				\begin{eqnarray*}
					\tilde{\mathcal{L}}_p v & = & \frac{1}{\phi_0} \times \frac{1}{\phi_0^{1 /
							2}}  (\mathcal{L}_p \phi) (x_0 + \phi_0^{- 1 / 4} y)\\
					& \leq & \frac{C}{\phi_0^{3 / 2}} (\phi + \phi^{3 / 2}) (x_0 +
					\phi_0^{- 1 / 4} y)\\
					& \leq & C \left( \frac{1}{\phi_0^{1 / 2}} v + \left( \frac{\phi
						(x_0 + \phi_0^{- 1 / 4} y)}{\phi_0} \right)^{1 / 2} v \right)\\
					& \leq & C \left( v + \sqrt{16} v \right)\\
					& \leq & C v.
				\end{eqnarray*}
				This proves that $v$ is a subsolution for the operator
				$\tilde{\mathcal{L}}_p - C$. We can apply M{\"o}ser-Harnack inequality (see for example \cite[Theorem 4.1]{han_elliptic_2000}) to get an estimate of the form
				\[ v (0) \leq C \int_{\text{B}_1} v \]
				where $C > 0$ is a constant (independent of $p$). Recall that $v (0) = 1$ and by conformal
				invariance of the Yang-Mills energy
				\[ \int_{\text{B}_1} v = \int_{\text{B}_1} | \alpha^{\ast} F |^2 =
				\int_{\text{B}_{\phi_0^{- 1 / 4} (x_0)}} | F |^2 = \int_{\text{B}_{\rho_0
						/ r_0 (x_0)}} | F |^2 \leq \int_{\text{B}_{\rho_0 (x_0)}} | F |^2
				\leq \int_{\text{B}_R} | F |^2 \leq \varepsilon_0, \]
				so $1 \leq C \varepsilon_0$. This becomes a contradiction by choosing
				$\varepsilon_0$ small enough.
				
				We have proved that $r_0 < 1$ \textit{i.e.} $\phi_0 < \rho_0^{- 4}$.
				Consider now $w = | \beta^{\ast} F |^2$ where $\beta (y) = x_0 + \rho_0 y$.
				We have the following expression for $w$:
				\[ w (y) = \rho_0^4 | F |^2 (x_0 + \rho_0 y) = \rho_0^4 \phi (x_0 + \rho_0
				y), \]
				so, using \eqref{phiestimate1} as before, $w$ satisfies
				\begin{align*}
					w (0) & =  \rho_0^4 \phi_0 = r_0^4\\
					\sup_{\overline{\text{B}}_1} w & =  \sup_{\overline{\text{B}}_{\rho_0}
						(x_0)} \rho_0^4 \phi = \rho_0^4 \phi_0 \sup_{\overline{\text{B}}_{\rho_0}
						(x_0)}  \frac{\phi}{\phi_0} \leq 16 r_0^4 < 16.
				\end{align*}
				Define this time the operator $\check{\mathcal{L}}_p$ as
				\[ \check{\mathcal{L}}_p w = \partial_{\alpha} ((\delta^{\alpha \beta} + (p
				- 2) \mathcal{A}^{\alpha \beta} \circ \beta) \partial_{\beta} w) \]
				Then
				\begin{align*}
					\check{\mathcal{L}}_p w & =  \rho_0^4 \times \rho_0^2  (\mathcal{L}_p
					\phi) (x_0 + \rho_0 y)\\
					& \leq  C \rho_0^6 (\phi + \phi^{3 / 2}) (x_0 + \rho_0 y)\\
					& \leq  C (\rho_0^2 w + (\rho_0^4 \phi (x_0 + \rho_0 y))^{1 / 2}
					w)\\
					& \leq  C \left( w + \sqrt{16} w \right)\\
					& \leq  C w,
				\end{align*}
				\textit{i.e.} $w$ is a subsolution for the operator $\check{\mathcal{L}}_p
				- C$. Apply again M{\"o}ser-Harnack inequality to obtain
				\[ r_0^4 = w (0) \leq C \int_{\text{B}_1} w. \]
				Using once more the conformal invariance of the energy,
				\[ \int_{\text{B}_1} w = \int_{\text{B}_{\rho_0 (x_0)}} | F |^2 \leq
				\int_{\text{B}_R} | F |^2 . \]
				We deduce
				\[ f (\sigma_0) = (R - \sigma_0)^4 \phi_0 = 16 r_0^4 \leq C
				\int_{\text{B}_R} | F |^2 \]
				and therefore for all $\sigma \in [0, R)$,
				\[ (R - \sigma)^4 \sup_{\overline{\text{B}}_{\sigma}} | F |^2 \leq C
				\int_{\text{B}_R} | F |^2 . \]
				Taking $\sigma = R / 2$, we finally obtain
				\[ \sup_{\overline{\text{B}}_{R / 2}} | F | \leq \frac{C}{R^2}
				\sqrt{\int_{\text{B}_R} | F |^2} . \]
			\end{proof}
			
			\newpage
			
			\printbibliography[heading=bibintoc,title=References]

@article {DaRiSchla,
	AUTHOR = {Da Lio, Francesca and Rivi\`ere, Tristan and Schlagenhauf,
	Dominik},
	TITLE = {Morse index stability for {S}acks--{U}hlenbeck approximations
	for harmonic maps into a sphere},
	JOURNAL = {Nonlinear Anal.},
	FJOURNAL = {Nonlinear Analysis. Theory, Methods \& Applications. An
	International Multidisciplinary Journal},
	VOLUME = {264},
	YEAR = {2026},
	PAGES = {Paper No. 113987},
	ISSN = {0362-546X,1873-5215},
	MRCLASS = {35J92 (35J47 35J50 53A10 53C43 58E05 58E12 58E20)},
	MRNUMBER = {4976778},
	DOI = {10.1016/j.na.2025.113987},
	URL = {https://doi-org.univ-eiffel.idm.oclc.org/10.1016/j.na.2025.113987},
}

@article {LiZhu,
	AUTHOR = {Li, Jiayu and Zhu, Xiangrong},
	TITLE = {Energy identity and necklessness for a sequence of
	{S}acks-{U}hlenbeck maps to a sphere},
	JOURNAL = {Ann. Inst. H. Poincar\'e{} C Anal. Non Lin\'eaire},
	FJOURNAL = {Annales de l'Institut Henri Poincar\'e{} C. Analyse Non
	Lin\'eaire},
	VOLUME = {36},
	YEAR = {2019},
	NUMBER = {1},
	PAGES = {103--118},
	ISSN = {0294-1449,1873-1430},
	MRCLASS = {58E20 (35R01)},
	MRNUMBER = {3906867},
	MRREVIEWER = {Nobumitsu\ Nakauchi},
	DOI = {10.1016/j.anihpc.2018.04.002},
	URL = {https://doi-org.univ-eiffel.idm.oclc.org/10.1016/j.anihpc.2018.04.002},
}

@misc{Schlagen,
	title={Stability of the Morse Index for the $p$-harmonic Approximation of Harmonic Maps into Homogeneous Spaces}, 
	author={Dominik Schlagenhauf},
	year={2025},
	eprint={2506.10761},
	archivePrefix={arXiv},
	primaryClass={math.AP},
	url={https://arxiv.org/abs/2506.10761}, 
}

@book {Helein,
	AUTHOR = {H\'elein, Fr\'ed\'eric},
	TITLE = {Harmonic maps, conservation laws and moving frames},
	SERIES = {Cambridge Tracts in Mathematics},
	VOLUME = {150},
	EDITION = {Second},
	NOTE = {Translated from the 1996 French original,
	With a foreword by James Eells},
	PUBLISHER = {Cambridge University Press, Cambridge},
	YEAR = {2002},
	PAGES = {xxvi+264},
	ISBN = {0-521-81160-0},
	MRCLASS = {58E20 (35A22 35J15 53C43 58E12)},
	MRNUMBER = {1913803},
	MRREVIEWER = {Andreas\ Gastel},
	DOI = {10.1017/CBO9780511543036},
	URL = {https://doi-org.univ-eiffel.idm.oclc.org/10.1017/CBO9780511543036},
}

@article {SaU,
	AUTHOR = {Sacks, J. and Uhlenbeck, K.},
	TITLE = {The existence of minimal immersions of {$2$}-spheres},
	JOURNAL = {Ann. of Math. (2)},
	FJOURNAL = {Annals of Mathematics. Second Series},
	VOLUME = {113},
	YEAR = {1981},
	NUMBER = {1},
	PAGES = {1--24},
	ISSN = {0003-486X},
	MRCLASS = {58E12 (53C42 58E20)},
	MRNUMBER = {604040},
	MRREVIEWER = {John\ C.\ Wood},
	DOI = {10.2307/1971131},
	URL = {https://doi.org/10.2307/1971131},
}

@article{UhlenbeckKarenK1982CwLb,
author = {Uhlenbeck, Karen K},
copyright = {Copyright 1982 Springer-Verlag},
issn = {1432-0916},
journal = {Communications in mathematical physics},
number = {no. 1},
pages = {31-42},
publisher = {Springer-Verlag},
title = {Connections with Lp bounds on curvature},
volume = {83},
year = {1982},
}

@article {BourguignonLawson,
	AUTHOR = {Bourguignon, Jean-Pierre and Lawson, Jr., H. Blaine},
	TITLE = {Stability and isolation phenomena for {Y}ang-{M}ills fields},
	JOURNAL = {Comm. Math. Phys.},
	FJOURNAL = {Communications in Mathematical Physics},
	VOLUME = {79},
	YEAR = {1981},
	NUMBER = {2},
	PAGES = {189--230},
	ISSN = {0010-3616,1432-0916},
	MRCLASS = {58E20 (53C05 81E10)},
	MRNUMBER = {612248},
	MRREVIEWER = {M.\ F.\ Atiyah},
	URL = {http://projecteuclid.org/euclid.cmp/1103908963},
}

@book {FU,
	AUTHOR = {Freed, Daniel S. and Uhlenbeck, Karen K.},
	TITLE = {Instantons and four-manifolds},
	SERIES = {Mathematical Sciences Research Institute Publications},
	VOLUME = {1},
	EDITION = {Second},
	PUBLISHER = {Springer-Verlag, New York},
	YEAR = {1991},
	PAGES = {xxii+194},
	ISBN = {0-387-97377-X},
	MRCLASS = {57R55 (57M40 57N05 57R57 57R60 58D27 58G10)},
	MRNUMBER = {1081321},
	DOI = {10.1007/978-1-4613-9703-8},
	URL = {https://doi.org/10.1007/978-1-4613-9703-8},
}

@article{DGR22,
  title={Morse index stability for critical points to conformally invariant lagrangians},
  author={Da Lio, Francesca and Gianocca, Matilde and Rivi{\`e}re, Tristan},
  journal={arXiv preprint arXiv:2212.03124, To appear at JEMS},
  year={2022}
}

@InCollection{rivière2015variations,
 Author = {Rivi{\`e}re, Tristan},
 Title = {The variations of {Yang}-{Mills} {Lagrangian}},
 BookTitle = {Geometric analysis. In honor of Gang Tian's 60th birthday},
 ISBN = {978-3-030-34952-3; 978-3-030-34955-4; 978-3-030-34953-0},
 Pages = {305--379},
 Year = {2020},
 Publisher = {Cham: Birkh{\"a}user},
 Language = {English},
 DOI = {10.1007/978-3-030-34953-0_15},
 zbMATH = {7225730},
 Zbl = {1460.58010}
}

@article{Hunt1966,
author={Hunt, Richard A.},
title={ON {$L^{p,q}$} SPACES},
journal={L'Enseignement Math{\'e}matique},
year={1966},
publisher={Fondation L'Enseignement Math{\'e}matique},
volume={12},
number={4},
pages={249},
issn={0013-8584},
doi={10.5169/seals-40747},
url={https://doi.org/10.5169/seals-40747}
}

@article {troyanov2010hodge,
	AUTHOR = {Troyanov, Marc},
	TITLE = {On the {H}odge decomposition in {$\Bbb R^n$}},
	JOURNAL = {Mosc. Math. J.},
	FJOURNAL = {Moscow Mathematical Journal},
	VOLUME = {9},
	YEAR = {2009},
	NUMBER = {4},
	PAGES = {899--926, 936},
	ISSN = {1609-3321,1609-4514},
	MRCLASS = {58A14 (42B20 46F05)},
	MRNUMBER = {2663996},
	MRREVIEWER = {Gilles\ Carron},
	DOI = {10.17323/1609-4514-2009-9-4-899-926},
	URL = {https://doi-org.univ-eiffel.idm.oclc.org/10.17323/1609-4514-2009-9-4-899-926},
}

@article{peetrelorentzsobolev,
     author = {Peetre, Jaak},
     title = {Espaces d'interpolation et th\'eor\`eme de {Soboleff}},
     journal = {Annales de l'Institut Fourier},
     pages = {279--317},
     publisher = {Institut Fourier},
     address = {Grenoble},
     volume = {16},
     number = {1},
     year = {1966},
     doi = {10.5802/aif.232},
     zbl = {0151.17903},
     mrnumber = {36 #4334},
     language = {fr},
     url = {https://aif.centre-mersenne.org/articles/10.5802/aif.232/}
}

@book{gilbarg1977elliptic,
  title={Elliptic partial differential equations of second order},
  author={Gilbarg, David and Trudinger, Neil S and Gilbarg, David and Trudinger, NS},
  volume={224},
  number={2},
  year={1977},
  publisher={Springer}
}

@article{riviere2002interpolation,
  title={Interpolation spaces and energy quantization for Yang--Mills fields},
  author={Riviere, Tristan},
  journal={Communications in Analysis and Geometry},
  volume={10},
  number={4},
  pages={683--708},
  year={2002},
  publisher={International Press of Boston}
}

@article{laurainriviere2014angular,
  title={Angular energy quantization for linear elliptic systems with antisymmetric potentials and applications},
  author={Laurain, Paul and Riviere, Tristan},
  journal={Analysis \& PDE},
  volume={7},
  number={1},
  pages={1--41},
  year={2014},
  publisher={Mathematical Sciences Publishers}
}

@book {Grafakos1,
    AUTHOR = {Grafakos, Loukas},
     TITLE = {Classical {F}ourier analysis},
    SERIES = {Graduate Texts in Mathematics},
    VOLUME = {249},
   EDITION = {Third},
 PUBLISHER = {Springer, New York},
      YEAR = {2014},
     PAGES = {xviii+638},
      ISBN = {978-1-4939-1193-6; 978-1-4939-1194-3},
   MRCLASS = {42-01 (42Bxx)},
  MRNUMBER = {3243734},
MRREVIEWER = {Atanas\ G.\ Stefanov},
       DOI = {10.1007/978-1-4939-1194-3},
       URL = {https://doi.org/10.1007/978-1-4939-1194-3},
}

@article {RiviereLin,
    AUTHOR = {Lin, Fang-Hua and Rivi\`ere, Tristan},
     TITLE = {Energy quantization for harmonic maps},
   JOURNAL = {Duke Math. J.},
  FJOURNAL = {Duke Mathematical Journal},
    VOLUME = {111},
      YEAR = {2002},
    NUMBER = {1},
     PAGES = {177--193},
      ISSN = {0012-7094,1547-7398},
   MRCLASS = {58E20},
  MRNUMBER = {1876445},
MRREVIEWER = {Ernst\ C.\ Kuwert},
       DOI = {10.1215/S0012-7094-02-11116-8},
       URL = {https://doi.org/10.1215/S0012-7094-02-11116-8},
}

@article {PR14,
    AUTHOR = {Petrache, Mircea and Rivi\`ere, Tristan},
     TITLE = {Global gauges and global extensions in optimal spaces},
   JOURNAL = {Anal. PDE},
  FJOURNAL = {Analysis \& PDE},
    VOLUME = {7},
      YEAR = {2014},
    NUMBER = {8},
     PAGES = {1851--1899},
      ISSN = {2157-5045,1948-206X},
   MRCLASS = {46E35 (35B45 35J57 58J05 70S15)},
  MRNUMBER = {3318742},
MRREVIEWER = {Yoichi\ Miyazaki},
       DOI = {10.2140/apde.2014.7.1851},
       URL = {https://doi.org/10.2140/apde.2014.7.1851},
}

@article {GL24,
	AUTHOR = {Gauvrit, Mario and Laurain, Paul},
	TITLE = {Morse index stability for {Y}ang-{M}ills connections},
	JOURNAL = {Int. Math. Res. Not. IMRN},
	FJOURNAL = {International Mathematics Research Notices. IMRN},
	YEAR = {2025},
	NUMBER = {16},
	PAGES = {Paper No. rnaf250, 33},
	ISSN = {1073-7928,1687-0247},
	MRCLASS = {53C07 (57K40 58E15)},
	MRNUMBER = {4948567},
	DOI = {10.1093/imrn/rnaf250},
	URL = {https://doi-org.univ-eiffel.idm.oclc.org/10.1093/imrn/rnaf250},
}

@book {han_elliptic_2000,
	AUTHOR = {Han, Qing and Lin, Fanghua},
	TITLE = {Elliptic partial differential equations},
	SERIES = {Courant Lecture Notes in Mathematics},
	VOLUME = {1},
	EDITION = {Second},
	PUBLISHER = {Courant Institute of Mathematical Sciences, New York; American
	Mathematical Society, Providence, RI},
	YEAR = {2011},
	PAGES = {x+147},
	ISBN = {978-0-8218-5313-9},
	MRCLASS = {35Jxx (35-01 35B50)},
	MRNUMBER = {2777537},
}

@book {GM,
	AUTHOR = {Giaquinta, Mariano and Martinazzi, Luca},
	TITLE = {An introduction to the regularity theory for elliptic systems,
	harmonic maps and minimal graphs},
	SERIES = {Appunti. Scuola Normale Superiore di Pisa (Nuova Serie)
	[Lecture Notes. Scuola Normale Superiore di Pisa (New
	Series)]},
	VOLUME = {11},
	EDITION = {Second},
	PUBLISHER = {Edizioni della Normale, Pisa},
	YEAR = {2012},
	PAGES = {xiv+366},
	ISBN = {978-88-7642-442-7; 978-88-7642-443-4},
	MRCLASS = {35-02 (35B65 35J20 35J60 58E20)},
	MRNUMBER = {3099262},
	DOI = {10.1007/978-88-7642-443-4},
	URL = {https://doi.org/10.1007/978-88-7642-443-4},
}

@article{bolik_h_2001,
	title = {H. {Weyl}'s boundary value problems for differential forms},
	url = {https://api.semanticscholar.org/CorpusID:124735710},
	journal = {Differential and Integral Equations},
	author = {Bolik, Jürgen},
	year = {2001},
}

@article {HS,
	AUTHOR = {Hong, Min-Chun and Schabrun, Lorenz},
	TITLE = {The energy identity for a sequence of {Y}ang-{M}ills
	{$\alpha$}-connections},
	JOURNAL = {Calc. Var. Partial Differential Equations},
	FJOURNAL = {Calculus of Variations and Partial Differential Equations},
	VOLUME = {58},
	YEAR = {2019},
	NUMBER = {3},
	PAGES = {Paper No. 83, 27},
	ISSN = {0944-2669,1432-0835},
	MRCLASS = {58E15},
	MRNUMBER = {3946427},
	MRREVIEWER = {Daniele\ Bartolucci},
	DOI = {10.1007/s00526-019-1535-y},
	URL = {https://doi.org/10.1007/s00526-019-1535-y},
}

@article {LiW,
	AUTHOR = {Li, Yuxiang and Wang, Youde},
	TITLE = {A weak energy identity and the length of necks for a sequence
	of {S}acks-{U}hlenbeck {$\alpha$}-harmonic maps},
	JOURNAL = {Adv. Math.},
	FJOURNAL = {Advances in Mathematics},
	VOLUME = {225},
	YEAR = {2010},
	NUMBER = {3},
	PAGES = {1134--1184},
	ISSN = {0001-8708,1090-2082},
	MRCLASS = {58E20 (35J62)},
	MRNUMBER = {2673727},
	MRREVIEWER = {Andreas\ Gastel},
	DOI = {10.1016/j.aim.2010.03.020},
	URL = {https://doi.org/10.1016/j.aim.2010.03.020},
}

@article {Waldron,
	AUTHOR = {Waldron, Alex},
	TITLE = {Long-time existence for {Y}ang-{M}ills flow},
	JOURNAL = {Invent. Math.},
	FJOURNAL = {Inventiones Mathematicae},
	VOLUME = {217},
	YEAR = {2019},
	NUMBER = {3},
	PAGES = {1069--1147},
	ISSN = {0020-9910,1432-1297},
	MRCLASS = {53E99 (35R01 58E15)},
	MRNUMBER = {3989258},
	MRREVIEWER = {Yu\ Zheng},
	DOI = {10.1007/s00222-019-00877-2},
	URL = {https://doi.org/10.1007/s00222-019-00877-2},
}

@article{uhlenbeck_regularity_1977,
	title = {Regularity for a class of non-linear elliptic systems},
	volume = {138},
	issn = {0001-5962,1871-2509},
	url = {https://doi.org/10.1007/BF02392316},
	doi = {10.1007/BF02392316},
	number = {3-4},
	journal = {Acta Mathematica},
	author = {Uhlenbeck, K.},
	year = {1977},
	mrnumber = {474389},
	pages = {219--240},
}

@inproceedings{isobe_regularity_2008,
	title = {A regularity result for a class of degenerate {Yang}-{Mills} connections in critical dimensions},
	url = {https://api.semanticscholar.org/CorpusID:122311636},
	author = {Isobe, Takeshi},
	year = {2008},
}

@article{hamburger_regularity_1992,
	title = {Regularity of differential forms minimizing degenerate elliptic functionals.},
	volume = {431},
	url = {http://eudml.org/doc/153451},
	journal = {Journal für die reine und angewandte Mathematik},
	author = {Hamburger, C.},
	year = {1992},
	pages = {7--64},
}

@book{HardyIneq,
	author = {Balinsky, Alexander A and Evans, W. Desmond and Lewis, Roger T and Lewis, Roger T and Evans, W. Desmond},
	address = {Cham},
	booktitle = {Universitext},
	copyright = {These electronic books are licensed by OhioLINK and may be under copyright protection. Please see the Acceptable Use Guidelines for more information, or contact your librarian.},
	edition = {1},
	isbn = {3319228706},
	issn = {0172-5939},
	language = {eng},
	publisher = {Springer International Publishing},
	series = {Universitext},
	title = {The Analysis and Geometry of Hardy’s Inequality},
	year = {2015},
}

@article{uhlenbeck_chern_1985,
	title = {The {Chern} classes of {Sobolev} connections},
	volume = {101},
	issn = {1432-0916},
	url = {https://doi.org/10.1007/BF01210739},
	doi = {10.1007/BF01210739},
language = {en},
number = {4},
urldate = {2025-11-21},
journal = {Communications in Mathematical Physics},
author = {Uhlenbeck, Karen K.},
month = dec,
year = {1985},
pages = {449--457},
}

		\end{document}